\documentclass{article}
\usepackage[english]{babel}
\usepackage[T1]{fontenc}
\usepackage[utf8]{inputenc}
\RequirePackage{amsthm,amsmath,amsfonts,amssymb}
\RequirePackage[numbers]{natbib}

\allowdisplaybreaks

\newcommand{\ac}[1]{\textcolor{red}{add citation}}

\usepackage{enumerate}
\usepackage{color}
\usepackage{multirow}\usepackage{graphicx}
\usepackage{amsfonts}
\usepackage{amsmath}
\usepackage{amssymb}

\usepackage{color}
\usepackage{dsfont}

\usepackage{pgfplots}
\usetikzlibrary{patterns}
\usepackage{amssymb}
\usepackage{graphicx}
\usepackage{amsmath}
\usepackage{bm}
\usepackage{ltablex}
\usepackage{algorithm} 
\usepackage{algorithmic}  
\usepackage[algo2e]{algorithm2e} 
\usepackage{caption}

\usetikzlibrary{decorations.pathreplacing}

\usetikzlibrary{calc,patterns}

\usepackage{url}
\usepackage{fancyhdr}
\usepackage{indentfirst}

\usepackage{enumerate}
\usepackage{dsfont}
\usepackage{tabulary}
\usepackage{booktabs}
\usepackage[colorlinks=true,citecolor=blue]{hyperref}
\usepackage{amsthm}
\usepackage{color}
\usepackage{comment}
\usepackage{tikz-cd}
\usepackage{graphicx}
\usepackage{amsfonts}
\usepackage{longtable}
\usepackage{amsmath}
\usepackage{amssymb}
\usepackage{url}
\usepackage{bbm}
\usepackage{fancyhdr}
\usepackage{indentfirst}
\usepackage{enumerate}
\usepackage{subcaption}
\usepackage{dsfont}
\pgfplotsset{compat=1.7}
\usepackage[colorlinks=true,citecolor=blue]{hyperref}
\usepackage{cleveref}
\usepackage{amsthm}
\usepackage{color}
\renewcommand{\cite}{\citet}
\usepackage{comment}
\usepackage[margin=0.8in]{geometry}

\usepackage{algorithm}

\renewcommand{\d}{\,\mathrm{d}}

\newcommand{\p}{\mathbb{P}}

\usepackage{xcolor}

\newcommand{\E}{\mathbb{E}}    
\newcommand{\R}{\mathbb{R}}    
\newcommand{\N}{\mathbb{N}}    

\theoremstyle{plain}
\newtheorem{theorem}{Theorem}
\newtheorem{corollary}[theorem]{Corollary}
\newtheorem{lemma}[theorem]{Lemma}
\newtheorem{proposition}[theorem]{Proposition}

\theoremstyle{definition}

\theoremstyle{remark}
\newtheorem{remark}{Remark}

\usepackage{tikz}

\usetikzlibrary{arrows.meta,positioning}

\newcommand{\ee}{\varepsilon}

\newcommand{\bx}{\mathbf{x}}
\newcommand{\by}{\mathbf{y}}

\def\d{\mathrm{d}}

\def\lawis{\buildrel \mathrm{law} \over \sim}

\newcommand{\bone}{ {\mathbbm{1}} }

\usepackage{mathrsfs}

\newcommand{\z}{\mathbf{0}}

\newcommand{\X}{\mathbb{H}}

\renewcommand{\ge}{\geqslant}
\renewcommand{\le}{\leqslant}
\renewcommand{\geq}{\geqslant}
\renewcommand{\leq}{\leqslant}
\renewcommand{\epsilon}{\varepsilon}

\newcommand{\bZ}{{\mathbb{Z}}}

\title{Viral Quasispecies Evolution as a Branching Random Walk on the Hypercube}

\author{Jose Blanchet\thanks{Department of Management Science and Engineering, Stanford University. Email: jose.blanchet@stanford.edu} \and Zhenyuan Zhang\thanks{Department of Mathematics, Stanford University. Email: zzy@stanford.edu}}
\begin{document}
\maketitle

\begin{abstract}
We study a continuous-time nearest-neighbor branching random walk on the $d$-dimensional $b$-ary hypercube $\{0,1,\dots,b-1\}^d$ as a model for viral quasispecies evolution under mutation and replication. Motivated by mutagenic antiviral treatments and evolutionary-safety questions, we analyze the first passage time to a fixed target genotype at Hamming distance $m$, corresponding to the first appearance of a prescribed collection of mutations. We derive sharp asymptotics for these first passage times, uniformly for $m\le d/L$ as $d\to\infty$ (where $L>0$ is a large constant), and identify a phase transition in first-passage scaling at $\rho=e$, where $\rho$ denotes the effective growth parameter. In the slow-branching regime $\rho\in(1,e)$ relevant to mutagenic treatment scenarios, the first passage time is asymptotically affine in the genome length $d$ and the target distance $m$. In particular, when replication is fixed and mutation exceeds branching, increasing the mutation rate can delay the first appearance of a prescribed genotype by order $d$, providing a quantitative perspective on evolutionary safety.
\end{abstract}

\section{Introduction}

In the theory of molecular sequence dynamics, the classic Jukes--Cantor model \citep{jukes1969evolution}
describes nucleotide substitutions by a continuous-time Markov process on sequence space, while
branching-process models provide a natural description of replication and population growth
\citep{athreya2004branching}. Motivated by within-host RNA viral evolution and mutagenic antiviral
interventions, we study a branching random walk on sequence space (BRWSS) that combines these two
mechanisms in one spatial model, with branching used as a proxy for viral replication dynamics
\citep{sanjuan2010viral}.

We work on the $d$-dimensional $b$-ary hypercube
$\X=\X_d^{(b)}:=\{0,1,\dots,b-1\}^d$, endowed with Hamming distance
$d_{\mathrm H}(\cdot,\cdot)$. Each vertex represents a genotype; in the RNA setting one has $b=4$,
corresponding to the four nucleotides A, C, G, and T.\footnote{For notation, we write $T$ in place
of $U$ when representing RNA sequences.}
A particle at a genotype branches
at rate $\lambda_1>0$ (producing one additional offspring at the same genotype) and mutates at rate
$\lambda_2>0$ (changing one uniformly chosen coordinate to one of the other $b-1$ symbols,
uniformly). We define the effective growth parameter as
\begin{align*}
\rho:=\exp(\lambda_1/\lambda_2).
\end{align*}
After time-rescaling by $\lambda_2$, we set $\lambda_2=1$ and write $\lambda_1=\log\rho$.
In these units, the expected population size at time $t$ is $\rho^t$. This formulation keeps the
finite sequence-space geometry explicit, including back mutations, while remaining simple enough for
sharp asymptotic analysis.

For many RNA viruses, the per-site mutation probability per replication is empirically of order $1/d$
(Drake's rule; see \citep{drake1999mutation}). In our continuous-time setting, one lineage
experiences on average $\lambda_2/\lambda_1$ mutation events between replications, so the expected
number of mutations per site per replication is $\lambda_2/(\lambda_1 d)$. This motivates asymptotic
regimes in which $\rho$ remains of order one as $d\to\infty$, including both constant $\rho$ and the
ultra-slow branching regime $\rho(d)\to 1$.

Under mutation and selection, viral populations are often described as quasispecies: mutant clouds
concentrated near a master sequence, a highest-fitness genotype \citep{lauring2010quasispecies,wilke2005quasispecies}. Our model does not include
explicit genotype-dependent fitness, but that language is still useful for framing the spatial
question we study: starting from a genotype near the master sequence, how long does it take a
growing mutant population to first produce a fixed genotype at a prescribed Hamming distance? This
bridge from biological motivation to mathematics leads naturally to a first-passage problem on the
hypercube.

That question is particularly relevant for mutagenic antiviral strategies. Increasing the mutation
rate can suppress viable growth, but it also changes the time scale on which specific genotype
configurations are first produced. If the mutation rate is pushed too high, there is the risk of an
\emph{error catastrophe} \citep{eigen1971selforganization,eigen1988molecular,nowak1989error}, in which copying errors accumulate over
generations and lead to nonviable genomes. In that case, the quasispecies may lose the superiority
of the master sequence (sublethal mutagenesis/error threshold), or the population may even undergo a
sharp collapse due to nonfunctional and nonviable viral genomes (lethal mutagenesis/extinction
threshold) \citep{bull2007theory,domingo2012viral,perales2011lethal,tejero2015theories}.

Classical treatments of error catastrophe usually proceed through deterministic quasispecies
equations or multi-type branching-process descriptions based on spectral properties of
mutation--selection dynamics
\citep{eigen1971selforganization,antoneli2013viral,demetrius1985polynucleotide,fabreti2019stochastic,hermisson2002mutation}.
Our BRWSS model instead keeps the underlying sequence-space geometry explicit. In particular, the
BRWSS model exhibits a localization--delocalization threshold at $\rho=e$: for $\rho>e$ the population remains
concentrated near the master sequence, whereas for $\rho<e$ it delocalizes, 
 and first passage to a
fixed target typically occurs on the $d$ scale. In the latter case, Theorem~\ref{thm:main} (applied with $m=1$) shows
that it often takes time on the order of $d$ for the population to revisit the original master
sequence.

These considerations are central to mutagenic antiviral strategies. By significantly increasing the
mutation rate and driving the quasispecies toward a lethal zone, mutagenic antiviral drugs suppress
the virus population. Meanwhile, the sublethal zone may instead increase genetic diversity without
collapsing the quasispecies \citep{domingo2012viral}. A prototypical example is Molnupiravir, which
induces transition-heavy mutational patterns in SARS-CoV-2, especially G-to-A and C-to-T, and has
been discussed in connection with possible downstream evolutionary consequences
\citep{sanderson2023molnupiravir,kosakovsky2023anti}. Closely related work on evolutionary safety
compares the cumulative production of viable mutants with and without treatment
\citep{lobinska2023evolutionary,lobinska2024evolutionary}. In particular,
\citep{lobinska2023evolutionary} concludes from a mean-field ODE model that Molnupiravir is narrowly
evolutionarily safe under current parameter estimates. That criterion depends not only on aggregate
mutant load, but also on when potentially consequential genotypes first appear, which is one reason
the first-passage viewpoint is informative.

\paragraph{Main contribution.}
Fix $m\in\{1,\dots,d\}$ and $\bx_m\in\X_{d,m}:=\{\bx\in\X:d_{\mathrm H}(\bx,\mathbf 0)=m\}$.
We define the first passage time
\begin{align*}
\tau_{d,m}:=\inf\{t\ge 0:\exists v\in V_t,\ \eta_v(t)=\mathbf 0\},
\end{align*}
for the BRW initiated at $\bx_m$, where $V_t$ is the set of particles alive at time $t$ and
$\eta_v(t)$ is the location of particle $v$ at time $t$. By hypercube symmetry, this is equivalent
to starting from $\mathbf 0$ and hitting a fixed target at Hamming distance $m$; we use the
convention above throughout. Thus, ``target at distance $m$'' always means a fixed genotype, not an
arbitrary vertex in the sphere.

Our main results give sharp high-probability asymptotics for $\tau_{d,m}$ as $d\to\infty$, uniformly
over explicit ranges of $m$, and identify a phase transition at $\rho=e$ in first-passage scaling.
In the Molnupiravir and evolutionary-safety scenarios described above, mutagenic treatment
effectively reduces the viable branching rate relative to mutation, so the relevant slow-branching
regime is $\rho\in(1,e)$; concern about specific nucleotide positions also points to larger
prescribed mutation levels $m$. In this regime, combining Theorem~\ref{thm:main} and
Proposition~\ref{thm:const rho expansion}(iii) yields, for $m=o(d)$,
\begin{align*}
\tau_{d,m}=x_0(b,\rho)\,d+\big(1+o(1)\big)\,r(b,\rho)\,m+O_{\mathbb P}(1).
\end{align*}
Thus the leading approximation is affine in the genome length and the target distance, with explicit
constants. In particular, Corollary~\ref{coro:lambda2mono} shows that, with replication fixed and
$\lambda_2'>\lambda_2>\lambda_1$, increasing the mutation rate can delay first passage by an amount
of order $d$. This gives a direct quantitative handle on evolutionary-safety questions phrased in
terms of arrival times of fixed target genotypes.

\paragraph{Results by regime.}
\begin{enumerate}
    \item \textbf{Slow branching: $\rho\in(1,e)$.}
    First passage occurs on the linear scale $t=\Theta(d)$. By Theorem~\ref{thm:main}, uniformly for
    $m\in[1,d/L_1]$,
    \begin{align*}
    \tau_{d,m}=t_{d,m}+O_{\mathbb P}(1),
    \end{align*}
    where $t_{d,m}$ is the unique positive solution of the first-moment equation~\eqref{eq:1}. By
    Proposition~\ref{thm:const rho expansion}(iii),
    \begin{align*}
    t_{d,m}=x_0(b,\rho)\,d+r(b,\rho)\,m+O\!\left(\frac{m^2}{d}\right).
    \end{align*}
    Therefore,
    \begin{align*}
    \tau_{d,m}
    =x_0(b,\rho)\,d+r(b,\rho)\,m
    +O\!\left(\frac{m^2}{d}\right)+O_{\mathbb P}(1),
    \end{align*}
    and, in particular, if $m=o(d)$ with $m\to\infty$, then
    $\tau_{d,m}=x_0d+rm+o_{\mathbb P}(m)$. In this regime the target is reached only after
    exploration on a macroscopic genomic scale, and increasing mutation slows first passage at
    leading order.

    \item \textbf{Fast branching: $\rho>e$.}
    First passage occurs on a shorter pre-mixing scale. Theorem~\ref{thm:main2} shows that, with
    \begin{align*}
    t_{d,m}:=\frac{m\,W\!\left((\log\rho-1)(b-1)d/m\right)}{\log\rho-1},
    \end{align*}
    uniformly for $m\in[1,L\sqrt d/\log d]$,
    \begin{align*}
    \tau_{d,m}=t_{d,m}+O_{\mathbb P}(\log d).
    \end{align*}
    This regime is governed by atypically direct trajectories to the target before the process has
    mixed over sequence space.

    \item \textbf{Ultra-slow branching: $\rho(d)\to 1$.}
    Theorem~\ref{thm:main rho->1} shows that the first-moment centering remains correct only after an
    explicit correction at the natural scale. Under regular variation assumptions on $\log\rho(d)$,
    uniformly for $m\in[1,d/L_3]$,
    \begin{align*}
    \tau_{d,m}
    =t_{d,m}-\frac{-\log\log\rho}{\log\rho}
    +\frac{O_{\mathbb P}(1)}{\log\rho}.
    \end{align*}
    Hence the ultra-slow branching regime interpolates between weak growth and long waiting times, with a
    correction that is invisible at the first-moment level alone.
\end{enumerate}

The three regimes above characterize a phase transition in \emph{first-passage-time} scaling at
$\rho=e$: for $\rho\in(1,e)$ the relevant scale is $t=\Theta(d)$, while for $\rho>e$ first passage
is governed by pre-mixing trajectories and occurs on a shorter scale.

\paragraph{Relation to prior work.}
Our paper lies at the intersection of mathematical biology and extremal statistics of spatial
branching processes. On the biological side, related branching or BRW models have been used for
quasispecies and mutation--selection dynamics, somatic hypermutation, and epidemics
\citep{lauring2010quasispecies,antoneli2013viral,demetrius1985polynucleotide,fabreti2019stochastic,balelli2016branching,balelli2018random,ermakova2019branching}.
Finite-population analogues include Moran and Wright--Fisher quasispecies models
\citep{Cerf2015MoranMemoirs,Cerf2015WrightFisher,CerfDalmau2022Quasispecies}.

Methodologically, our work differs from common large-population or mean-field approaches in three ways: the
finite-genome geometry is explicit, back mutations are retained, and first appearance times are
defined at the particle level in a finite-particle BRW. This contrasts with infinite-genome
approximations in which each mutation creates a new type \citep{durrett2008probability} and with
settings where one studies type frequencies or mass profiles rather than fixed-target arrival times
\citep{bansaye2011limit,meleard2015stochastic,avena2020parabolic,konig2020branching}. Our
first-passage formulation therefore complements mutant-load analyses by directly encoding the arrival
time of prescribed genotypes in a finite sequence space.

From the probability side, our results complement work on bulk mass profiles, intermittency, and
one-dimensional extrema in spatial branching systems
\citep{avena2020parabolic,konig2020branching,addario2009minima,bramson1978maximal,arguin2016extrema,zeitouni2016branching}.
They also connect to recent first-passage studies for branching processes on other state spaces such as $\R^d$ 
\citep{blanchet2024first,blanchet2024tight,zhang2024modeling}. For the hypercube specifically,
\citep{balelli2016branching} obtained upper bounds on partial cover times for coalescing BRW, whereas
\citep{konig2020branching} studied mass distribution in random environment from a mean-field
perspective. Our focus is instead on sharp high-probability asymptotics for first passage to a fixed
target genotype on the $b$-ary hypercube.

\paragraph{Organization of the paper.}
Section~\ref{sec:model} formalizes the BRWSS model and notation. Section~\ref{sec:const rho}
analyzes the constant $\rho$ regime, including explicit expansions and the phase transition at $\rho=e$.
Section~\ref{sec:ultraslow} treats the ultra-slow branching regime $\rho(d)\to 1$.
Section~\ref{sec:discussion} discusses biological interpretation, limitations, and open directions.
Technical proofs are collected in the appendices.

\section{Model setup and notation}\label{sec:model}

We encode the nucleotides as symbols in $\{0,1,\dots,b-1\}$, and study the more general case of $b$ symbols where $b\geq 2$ is an integer. For DNA or RNA, $b=4$, and for a haploid two-state reduction, one can take $b=2$ by grouping bases into purines (A, G) and pyrimidines (C, T). 
Let $\X=\X_d^{(b)}:=\{0,1,\dots,b-1\}^d$ denote the $d$-dimensional $b$-ary hypercube, $\z=(0,\dots,0)\in\X$, and $\X_{d,m}:=\{\bx\in\X:d_{\mathrm{H}}(\bx,\z)=m\}$, where $d_{\mathrm{H}}(\cdot,\cdot)$ is the Hamming distance.\footnote{For two binary vectors $\bx$ and $\by$, $d_{\mathrm{H}}(\bx,\by)$ is the number of entries that are different. If $b=2$, this coincides with the $L^1$ norm (Manhattan distance) between $\bx$ and $\by$.} 
The continuous-time branching random walk (BRW) on the hypercube $\X$ can be formally described as follows. Start with a single particle at the origin $\z\in\X$. Each alive particle independently carries two independent exponential clocks of rates $\lambda_1,\lambda_2>0$ respectively. The first exponential clock of rate $\lambda_1$ governs the occurrence of the branching event: when it rings, the particle branches into two particles at the same location; the second clock of rate $\lambda_2$ governs the mutation event. At a mutation event, choose a coordinate $i$ uniformly from $\{1,\dots,d\}$, then replace the symbol by a uniformly chosen element from the other $b-1$ symbols.

Consider $m\in\{1,\dots,d\}$ and $\bx_m\in \X_{d,m}$, a vertex of Hamming distance $m$ from the origin. In this paper, we investigate the first passage times (FPT) of the continuous-time BRW initiated at $\bx_m$ and hitting $\z$, as a function of both $m$ and $d$.\footnote{By symmetry, this is equivalent to our previous consideration of the FPT from $\z$ to $\bx_m$ in the introduction.} Let $\tau_{d,m}$ denote the FPT of the BRW to $\z$, i.e., the first time some particle reaches $\z$ when the process starts from $\bx_m$. Formally, we have
$$\tau_{d,m}:=\inf\{t\geq 0:\,\exists v\in V_t,\, \eta_v(t)=\z\},$$
where $V_t$ is the set of all particles at time $t$ of a branching random walk on $\X_d^{(b)}$ initiated from $\bx_m$, and $\eta_v(t)$ is the location of $v$ at time $t$. 
By symmetry, starting at the master sequence (e.g. the origin) and accumulating $m$ prescribed mutations is equivalent to starting at a sequence $m$ steps away and hitting the master sequence. For completeness, we define $\tau_{d,0}:=0$.

 We have chosen to work with the continuous-time model (over the discrete-time one) for both branching and mutation events for four reasons. First, the continuous-time model delivers overlapping generations of biological interest and aligns with the classic Jukes--Cantor model and its variations. Second, and importantly for mathematical convenience, time-scaling is easier to perform in continuous-time, due to the scale-invariance of the exponential distribution. Therefore, we may without loss of generality assume that $\lambda_2=1$. For our convenience, we further introduce the parameterization $\lambda_1=\log\rho$, where $\rho>1$. 
As a consequence, the expected population size is $\rho^t$ at time $t$. Third, transition probabilities are technically easier to work with in continuous time (see \eqref{eq:qij} below) than in discrete time (see \citep{kac1947random}). And, finally, there will be no periodicity issues in continuous time, while the discrete-time random walk on the hypercube is periodic. 

We have assumed that each branching event replaces one particle by \textit{two}, but our results can be extended with minor modifications to the setting where each branching event replaces one particle by $n$ particles, independently with probability $p_n$, where $n\geq 0$, $\sum_n p_n=1$, and $\mu:=\sum_n np_n>1$. In particular, particles die independently with rate $p_0\lambda_1$. In this case, the parameterization $\lambda_1=\log\rho$ is replaced by $\lambda_1(\mu-1)=\log\rho$. Our main results carry over, with the understanding that the BRWSS process is conditioned upon the survival event (i.e., at each time, there is a particle alive).

\paragraph{Notation.} We write $A=O(B)$ if there exists a constant $C>0$ such that $A\leq CB$, and $A\asymp B$ if $A=O(B)$ and $B=O(A)$. Denote by $\N_2=\{2,3,\dots\}$. Recall that a sequence of real-valued random variables $\{\xi_n\}$ is \textit{tight} if for any $\ee>0$, there is a compact set $K\subseteq \R$ such that $\p(\xi_n\not\in K)<\ee$ for all $n$. Given a positive sequence $\{x_n\}_{n\in\N}$, we write $O_\p(x_n)$ for a sequence of random variables $\{X_n\}_{n\in\N}$ such that $\{X_n/x_n\}_{n\in\N}$ is tight, and similarly $o_\p(x_n)$ if $X_n/x_n\to 0$ in probability. 

\section{The constant-\texorpdfstring{$\rho$}{} regime}
\label{sec:const rho}

To state our main results, we first introduce the key \textit{first-moment equation}, given by
\begin{align}
    \rho^tb^{-d}\big(1+(b-1)e^{-\frac{bt}{(b-1)d}}\big)^{d-m}\big(1-e^{-\frac{bt}{(b-1)d}}\big)^m=1.\label{eq:1}
\end{align}

To gain intuition from~\eqref{eq:1}, note that as shown in Lemma~\ref{thm:bingham} below, the factor
$b^{-d}\big(1+(b-1)e^{-\frac{bt}{(b-1)d}}\big)^{d-m}\big(1-e^{-\frac{bt}{(b-1)d}}\big)^m$
is the time-$t$ transition probability for the underlying continuous time random walk to move between two genotypes at Hamming distance $m$. Its behavior depends sharply on the time scale.

If $t$ is proportional to $d$, then $e^{-\frac{bt}{(b-1)d}}$ is of constant order, and taking logarithms in~\eqref{eq:1} yields a balance between two terms of order $d$. This suggests that the relevant solution $t_{d,m}$ is of order $d$, corresponding to a regime in which the walk has had time to explore a nontrivial fraction of coordinates.

In contrast, if $t\ll d$, then $e^{-\frac{bt}{(b-1)d}}=1-\frac{bt}{(b-1)d}+O(t^2/d^2)$, and the left-hand side of~\eqref{eq:1} admits the approximation
\begin{equation}\label{eq:premix_balance}
\rho^t b^{-d}\big(1+(b-1)e^{-\frac{bt}{(b-1)d}}\big)^{d-m}\big(1-e^{-\frac{bt}{(b-1)d}}\big)^m
\approx \Big(\frac{\rho}{e}\Big)^t\Big(\frac{t}{(b-1)d}\Big)^m.
\end{equation}
The extra factor $e^{-t}$ explains why the threshold $\rho=e$ appears. When $\rho>e$, the right-hand side of~\eqref{eq:premix_balance} can reach one already at times of order $m\log d$, while for $\rho<e$ it cannot, and the relevant solution occurs on the $t=\Theta(d)$ scale. This heuristic is made precise in the results below (Theorems \ref{thm:main} and \ref{thm:main2}).

\subsection{The case \texorpdfstring{$\rho\in(1,e)$}{}}
 The following proposition describes, in the case $\rho\in(1,e)$, the solution to the first-moment equation \eqref{eq:1}.

\begin{proposition}\label{thm:const rho expansion}
    Fix $b\in\N_2$ and $\rho\in(1,e)$. Then there exists a large constant $L_1>0$ (possibly depending on $b,\rho$) such that the following statements hold uniformly for $m\in[1, d/L_1]$.
    \begin{enumerate}[(i)]
        \item There exists a unique strictly positive solution to \eqref{eq:1}.
        \item Denote by $x_0=x_0(b,\rho)$ the unique positive solution to 
  \begin{align}
      x_0\log\rho-\log b+\log\big(1+(b-1)e^{-\frac{bx_0}{b-1}}\big)=0\label{eq:x0  def}
  \end{align}
  and $r=r(b,\rho)$ the unique positive root to
  \begin{align}
      r\log\rho=\log\Big(\frac{1+(b-1)e^{-\frac{bx_0}{b-1}}}{1-e^{-\frac{bx_0}{b-1}}}\Big)+\frac{br}{b-1+e^{\frac{bx_0}{b-1}}}.\label{eq:r   def}
  \end{align}
  Then for every fixed $b\in\N_2$, $\rho\mapsto x_0(b,\rho)$ is decreasing, 
  \begin{align}
      \lim_{\rho\to 1^+}x_0(b,\rho)=\lim_{\rho\to e^-} r(b,\rho)=\infty,\quad\text{ and }\quad \lim_{\rho\to 1^+}r(b,\rho)=\lim_{\rho\to e^-} x_0(b,\rho)=0.\label{eq:x0r limits}
  \end{align}
  \item Denote by $t=t_{d,m}$ the unique solution to \eqref{eq:1}. Then as $d\to\infty$,
 \begin{align}
     t=x_0d+rm+O\Big(\frac{m^2}{d}\Big).\label{eq:x0r}
 \end{align} In particular, uniformly for $m=O(\sqrt{d})$, 
   $$ t=x_0d+rm+O(1).$$
    \end{enumerate}
\end{proposition}

We refer to Figure \ref{f} for a plot of the constants $x_0$ and $r$ defined by \eqref{eq:x0  def} and \eqref{eq:r   def}, as well as a comparison between the solution $t_{d,m}$ and its approximation \eqref{eq:x0r}. Our next result shows that the first passage time is concentrated near $t_{d,m}$.

\begin{figure}[t]
    \centering
    \begin{subfigure}[b]{0.47\linewidth}
        \centering
        \includegraphics[width=0.91\linewidth]{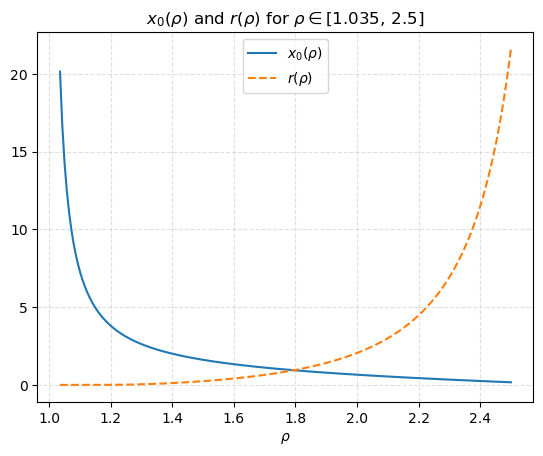}
        \caption{$x_0$ and $r$ as functions of $\rho$}
    \end{subfigure}
    \hfill
    \begin{subfigure}[b]{0.47\linewidth}
        \centering
        \includegraphics[width=1.03\linewidth]{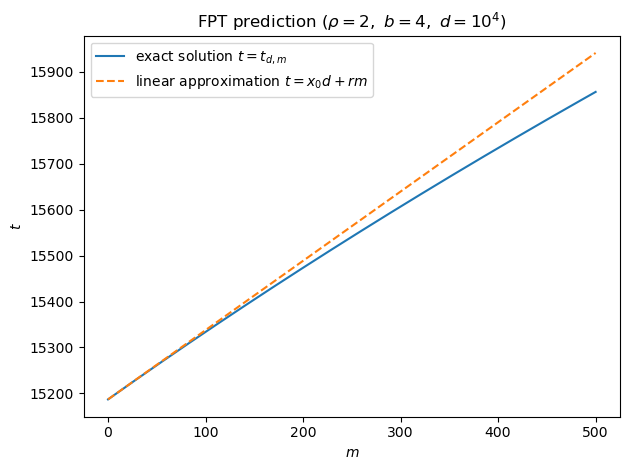}
        \caption{FPT predictions as a function of $m$}
    \end{subfigure}
    \caption{(a) Values of $x_0$ and $r$ as functions of $\rho\in[1.035,2.5]$ for $b=2$ (see definitions in \eqref{eq:x0  def} and \eqref{eq:r   def}). If $\rho$ is close to $1$, the term $x_0d$ dominates in \eqref{eq:x0r}; if $\rho$ is close to $e$, the term $rm$ dominates for $m$ large, where we recall that $m$ is the Hamming distance of the target from the origin. (b) FPT predictions from solving the first-moment equation \eqref{eq:1} and from the asymptotic expansion \eqref{eq:x0r} with $\rho=2,~b=4,$ $d=10^4$, and $m\in[0,500]$. }
    \label{f}
\end{figure}

\begin{theorem}\label{thm:main}
 Suppose that $b\in\N_2$ and $\rho\in(1,e)$. Then there exists a large constant $L_1>0$ (possibly depending on $b,\rho$) such that uniformly for $m\in[1, d/L_1]$ as $d\to\infty$,
        \begin{align*}
            \tau_{d,m}=t_{d,m}+O_\p(1),
        \end{align*}
        where the $O_\p(1)$ is tight and $t_{d,m}$ is the unique positive solution to \eqref{eq:1}.
\end{theorem}

We conjecture that, under the setting of Theorem \ref{thm:main}, if $m$ is a fixed constant that does not depend on $d$, then the $O_\p(1)$ fluctuation term converges in law as $d\to\infty$.

Using Theorem \ref{thm:main} we can obtain valuable insights about the behavior of $\tau_{d,m}$. Clearly, $\tau_{d,m}$ is decreasing in the branching rate $\lambda_1$ by a direct coupling argument, but the dependence on the mutation rate $\lambda_2$ is more subtle, because increasing $\lambda_2$ both accelerates exploration of the sequence space and, after time rescaling, decreases the effective branching strength per mutation time. In the slow branching regime $0<\lambda_1<\lambda_2$ and for targets at sublinear distance $m=o(d)$, the leading order of the first passage time is of order $d$, and we can quantify its dependence on $\lambda_2$ explicitly as the next corollary shows; the proof, which follows by time-scaling and comparison, is given after the proof of Theorem \ref{thm:main}.

\begin{corollary}[Monotonicity in the mutation rate]\label{coro:lambda2mono}
Fix $b\in\N_2$ and $\lambda_1>0$, and let $\lambda_2>\lambda_1$. Consider a BRWSS on $\X=\{0,1,\dots,b-1\}^d$ with branching rate $\lambda_1$ and mutation rate $\lambda_2$, and let $\tau_{d,m}^{(\lambda_1,\lambda_2)}$ denote the first passage time to a fixed target genotype at distance $m=m(d)$ from the origin, where $m=o(d)$ as $d\to\infty$. Then
\begin{align}
\frac{\tau_{d,m}^{(\lambda_1,\lambda_2)}}{d}
=\frac{1}{\lambda_2}\,x_0\Big(b,e^{\lambda_1/\lambda_2}\Big)+o_\p(1),\label{eq:lambda2:asymp}
\end{align}
where $x_0(b,\rho)$ is the positive solution of \eqref{eq:x0  def}. In particular, for any $\lambda_2'>\lambda_2>\lambda_1$,
\begin{align}
\tau_{d,m}^{(\lambda_1,\lambda_2')}-\tau_{d,m}^{(\lambda_1,\lambda_2)}
=\Bigg(\frac{x_0\big(b,e^{\lambda_1/\lambda_2'}\big)}{\lambda_2'}-\frac{x_0\big(b,e^{\lambda_1/\lambda_2}\big)}{\lambda_2}\Bigg)d+o_\p(d),\label{eq:lambda2:diff}
\end{align}
and the coefficient in parentheses is strictly positive. Hence, holding $\lambda_1$ fixed, increasing the mutation rate increases the first passage time by an amount of order $d$ when $m=o(d)$.
\end{corollary}

Another corollary of Theorem \ref{thm:main} provides the asymptotic order of the cover times of branching random walks on the hypercube. Results on cover times are motivated by the somatic hypermutation of B-cells \citep{balelli2016branching}. Denote by $\tau_{\mathrm{cov}}(d)$ the cover time of BRW on $\X_d^{(b)}$.

\begin{corollary}\label{coro:cover}
 Let $b\in\N_2$, and fix $\rho\in(1,e)$. Then the sequence of random variables $\{\tau_{\mathrm{cov}}(d)/d\}_{d\in\N}$ is tight. 
\end{corollary}

\subsection{The case \texorpdfstring{$\rho>e$}{}}
Recall from the discussion after~\eqref{eq:1} that in the pre-mixing regime the relevant transition probability carries an $e^{-t}$ factor, so the effective growth term becomes $(\rho/e)^t$, which is the origin of the threshold $\rho=e$.
Motivated by this, we define
\begin{align}
    t_{d,m}:=\frac{mW(\frac{(\log\rho-1)(b-1)d}{m})}{\log\rho-1},\label{eq:tdm:rho>e}
\end{align}
where $W(x)$ is the Lambert $W$ function and satisfies $W(x)e^{W(x)}=x$ for $x>0$. In particular, $t_{d,m}$ satisfies
\begin{align}
    \Big(\frac{\rho}{e}\Big)^{t_{d,m}}\Big(\frac{t_{d,m}}{(b-1)d}\Big)^m=1.\label{eq:tdm}
\end{align}
One can also prove via Taylor expansion arguments that $t_{d,m}$ is $O(1)$ away from being a solution to \eqref{eq:1}. Figure \ref{fig:plot sol} below shows the solution to \eqref{eq:1} for a range of $\rho$ around $\rho=e$. We refer to Lemma \ref{lemma:key} below for further properties and asymptotics of $t_{d,m}$.
\begin{figure}
    \centering
    \includegraphics[width=0.65\linewidth]{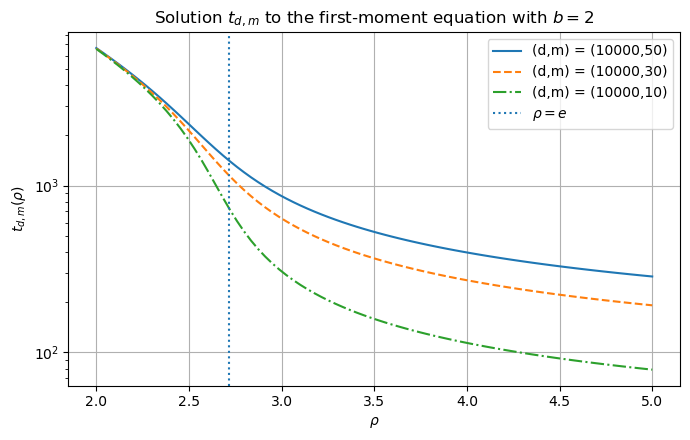}
    \caption{Plots of solutions to the first-moment equation \eqref{eq:1} for $\rho\in[2,5]$ with $b=2$. The $y$-axis is on a logarithmic scale.}
    \label{fig:plot sol}
\end{figure}

\begin{theorem}\label{thm:main2}
    Let $b\in\N_2$, $L>0$ be arbitrary, and $\rho>e$. Then 
    uniformly for $m\in[1,L\sqrt{d}/\log d]$ as $d\to\infty$,
     \begin{align}
            \tau_{d,m}=t_{d,m}+O_\p(\log d),\label{eq:t!}
        \end{align}
        where $t_{d,m}$ is defined in \eqref{eq:tdm:rho>e}. 
\end{theorem}

The restriction on $m$ is imposed to keep the pre-mixing approximation uniform and the proof sharp. Suppose that we are given a BRWSS with branching rate $\lambda_1$ and mutation rate $\lambda_2$, where $0<\lambda_2<\lambda_1$, and consider $d$ large enough and $m=O(\sqrt{d}/\log d)$. A similar analysis as before shows that if $\lambda_1$ increases, the FPT decreases at leading order; if $\lambda_2$ increases, the FPT increases at leading order. The latter claim follows directly from Theorem \ref{thm:main2}, Lemma \ref{lemma:key}(i) below, and the fact that if $1=\lambda_2<\lambda_2'<\log\rho$,
$$\frac{1}{\log\rho-1}<\frac{1}{\log\rho-\lambda_2'}=\frac{1}{\lambda_2'(\frac{\log\rho}{\lambda_2'}-1)}.$$

\begin{remark}
    The error term in \eqref{eq:t!} can be shown to be optimal. In fact, with positive probability, the first particle moves to the state $m+1$ before branching. Conditioning on this event, we see that the FPT cannot be concentrated in a window smaller than $t_{d,m+1}-t_{d,m}$, which is of order $t_{d,m}/m$ according to Lemma \ref{lemma:key}(iii) below, and further, is of order $\log d$ by Lemma \ref{lemma:key}(i) below in the range $m=O(\sqrt{d})$. 
    Formally, for a random variable $X$ and $\ell>0$, define the concentration function
$$Q(\ell,X):=\sup_{x\in\R}\p(x\leq X\leq x+\ell).$$ 
Then under the setting of Theorem \ref{thm:main2}, for some $\delta>0$, 
    $$\limsup_{d\to\infty}\sup_{m\in[1,L\sqrt{d}/\log d]} Q(\delta\log d,\tau_{d,m})\leq 1-\delta.$$
 \end{remark}

Theorems \ref{thm:main} and \ref{thm:main2} together indicate a phase transition regime of the FPT asymptotics in the branching rate $\log\rho$.  In the following, we sketch the arguments for the upper and lower bounds for the first passage time $\tau_{d,m}$. We start from a lemma on computing the transition probabilities for the simple random walk on $\X$.

\begin{lemma}\label{thm:bingham}
 Let $b\in\N_2$ and $\bx,\by\in \X$ with $m=d_{\mathrm{H}}(\bx,\by)$. The probability that the continuous-time simple random walk on $\X$ (with unit transition rate) transitions from the state $\bx$ to the state $\by$ at time $t$ depends only on $m$. Moreover, denoting the probability by $q_m(t)$, we have
    \begin{align}
        q_{m}(t)=b^{-d}\big(1+(b-1)e^{-\frac{bt}{(b-1)d}}\big)^{d-m}\big(1-e^{-\frac{bt}{(b-1)d}}\big)^m.\label{eq:qij}
    \end{align}
\end{lemma}

In the special case $b=2$, the simple random walk on $\{0,1\}^d$ can be identified as the Ehrenfest urn model (by projecting onto the number of ones), and the proof of Lemma \ref{thm:bingham} can be found in \citep{bingham1991fluctuation}. The case of general $b\in\N_2$ follows analogously. 
This result justifies the form of the left-hand side of \eqref{eq:1} in the following sense. Suppose that we initiate a BRW from $\bx_m\in \X_{d,m}$. By Lemma \ref{thm:bingham}, the expected number of particles at the origin $\z$ at time $t$ is given by
\begin{align}
\rho^tq_{m}(t)=\rho^tb^{-d}\big(1+(b-1)e^{-\frac{bt}{(b-1)d}}\big)^{d-m}\big(1-e^{-\frac{bt}{(b-1)d}}\big)^m.\label{eq:first moment}
\end{align}
In other words, if $t_{d,m}$ is a solution to \eqref{eq:1}, then at time $t_{d,m}$, on average there is one particle at the origin $\z$.

If $\rho\in(1,e)$, we apply a direct second moment computation. Suppose that we start from $\bx_m\in \X_{d,m}$ and set $t=t_{d,m}>0$ such that the first moment \eqref{eq:first moment} is equal to one. It is then not hard to derive the lower bound of $\tau_{d,m}$ using Markov's inequality, by considering the total amount of time spent in the target state $\z$. Next, we apply the many-to-two formula \citep{ikeda1969branching,sawyer1976branching} to compute the second moment and show that it is of order $O(1)$. However, this does not suffice since the second moment method (or the Paley--Zygmund inequality) only gives an upper bound with positive probability (but not probability close to one). We then apply a bootstrapping argument, where we first evolve the process from the state $\bx_m$ for a finite time, and show that the second moment is $O(1)$ uniformly in an $O(1)$ neighborhood of $\bx_m$. The key technicality here lies in the estimates of the second moments, as it often involves sums over multiple products of binomial coefficients and exponentials.

If $\rho\in(e,\infty)$, a direct second moment computation fails. Note that in this case, the majority of the particles will concentrate near the origin at short times, due to the large branching rate. Therefore, it is reasonable that for a short Hamming distance $m$, the trajectory leading to the first passage event will be ``straight'' without detours. To formulate this idea, suppose that the BRW is initiated at some $\bx_m\in\X_{d,m}$ and we consider the first-passage event to $\z$. We project the hypercube $\{0,1,\dots,b-1\}^d$ to the set $\{0,1,\dots,d\}$ according to the Hamming distance from $\z$. Then, we show that, with high probability, the optimal trajectory mutates $m$ times before reaching $\z$, and each mutation decreases the Hamming distance from $\z$. Given the event that there are $m$ mutations in the time interval $[0,t]$, the mutations occur at times i.i.d.~sampled from $[0,t]$, and hence can be described using an empirical process. We apply classic results on ballot theorems for empirical processes to show that a ``typical'' trajectory does not venture ``too much'' beyond $m-sm/t$ at time $s\in[0,t]$, and we apply a second moment count to such trajectories. Our approach is inspired by the modified second moment method applied in the study of extrema of one-dimensional spatial branching processes \citep{addario2009minima,roberts2013simple}.

\section{The ultra-slow branching regime}
\label{sec:ultraslow}

The goal of this section is to study the ultra-slow branching regime where $\rho(d)\to 1$ as $d\to\infty$. As we will show, the first moment does not predict the correct FPT asymptotics. Throughout this section, we assume $b=2$ for simplicity, i.e., we consider BRW on the hypercube $\{0,1\}^d$. 
  Recall the first-moment equation \eqref{eq:1}.    We first prove that a unique solution $t_{d,m}$ exists and provide an asymptotic expansion of $t_{d,m}$.

\begin{proposition}\label{thm:non-const rho expansion}
    Suppose that $\rho=\rho(d)\to 1^+$ satisfies that $\log\rho(d)$ is regularly varying of index $\gamma\in[-1,0)$ and $\log\rho(d)\geq L_2/d$ for some large enough constant $L_2>0$ and $1\leq m\leq d/L_1$ for $L_1$ large enough.     Then there exists a unique strictly positive solution to \eqref{eq:1}. Moreover, denote by $t=t_{d,m}$ the unique solution to \eqref{eq:1}. Then
$$t=\frac{d\log 2}{\log\rho(d)}+o(1).$$
\end{proposition}

\begin{theorem}\label{thm:main rho->1}
    Suppose that $\rho=\rho(d)$ satisfies the assumptions of Proposition \ref{thm:non-const rho expansion}. 
    Then there exists a large constant $L_3>0$, independent of $\rho$, such that uniformly for $m\in[1, d/L_3]$ as $d\to\infty$,
\begin{align}
    \tau_{d,m}=t_{d,m}-\frac{-\log\log\rho}{\log\rho}+\frac{O_\p(1)}{\log\rho},\label{eq:tau2!}
\end{align}
where $t_{d,m}$ is the unique positive solution to \eqref{eq:1}.
\end{theorem}

The double negative sign on the term $\log\log\rho/\log\rho$ is to emphasize that, since $\log\log\rho<0<\log\rho$, the FPT $\tau_{d,m}$ admits a correction term and is smaller than what is predicted from a first moment computation. Similarly to the discussions following Theorem \ref{thm:main2}, one can prove that the error term in \eqref{eq:tau2!} is optimal, i.e., under the setting of Theorem \ref{thm:main rho->1}, we have
     for some $\delta>0$, 
      \begin{align}
          \limsup_{\substack{d\to\infty}}Q\bigg(\frac{\delta}{\log\rho},\tau_{d,m}\bigg)\leq 1-\delta.\label{eq:Q rho->1}
      \end{align}
      We also have the following corollary on cover times, which can be proved in an analogous way to Corollary \ref{coro:cover}.

      \begin{corollary}
          \label{coro:cover 2}
          Fix a function $\rho(d)$ satisfying that $\log\rho(d)$ is regularly varying of index $\gamma\in[-1,0)$ and $\log\rho(d)\geq L_2/d$ where $L_2$ is as in Proposition \ref{thm:non-const rho expansion}.  Then the sequence of random variables $\{(\log\rho(d))\tau_{\mathrm{cov}}(d)/d\}_{d\in\N}$ is tight. 
      \end{corollary}
      Obtaining finer asymptotics of $\tau_{\mathrm{cov}}(d)$ remains a difficult task.  In the next result, we show that the cover times $\{\tau_{\mathrm{cov}}(d)\}$ are tight around the median, after scaling by $\log\rho(d)$ (the regularly varying condition of Corollary \ref{coro:cover 2} is not needed). Denote by $\mathrm{med}(X)$ the median of a random variable $X$.

\begin{proposition}\label{prop:cover tight}
    The collection $\{(\log\rho(d))(\tau_{\mathrm{cov}}(d)-\mathrm{med}(\tau_{\mathrm{cov}}(d)))\}_{d\in\N}$ is tight.
\end{proposition}

In the rest of this section, let us sketch the arguments behind Theorem \ref{thm:main rho->1}. 
In the case where $\rho(d)\to 1$, due to the randomness of the first branching time, the first passage time cannot be concentrated in a window of length $O(1/\log\rho)$ (see \eqref{eq:Q rho->1}). However, the same computation as in the proof of Theorem \ref{thm:main} would fail even in the order $O(1/\log\rho)$---the second moment would not be of finite order if we set the first-moment equal to one. The intuitive reason is that counting the number of particles at a certain time $t$ does not capture the first passage time, since, with a high probability, a particle entering the state $\z$ will leave without branching for a time of order $1/\log\rho$. 

To address this issue, we consider the FPT by looking at all positions in a time frame $[t,t+1/\log\rho]$ for a certain $t$ as the postulated FPT. On average, the chance that $\z$ is hit is approximately $\asymp 1/\log\rho$ times the chance that $\z$ is hit exactly at time $t$, because branching occurs $\asymp 1$ times in a time frame of length $\asymp 1/\log\rho$. Therefore, it is reasonable to consider a value $t'>0$ where the first moment \eqref{eq:first moment} is of order $\log\rho$ instead of being equal to one. The goal is then to count the number of particles at time $t'$ that produce a descendant reaching $\z$ in the next $1/\log\rho$ time (or equivalently, until time $t'+1/\log\rho$). This is equivalent to counting particles at time $t'$ with a weight given by the probability that it produces a descendant that reaches $\z$ in the next $1/\log\rho$ time. We then perform similar (but more technical) weighted first and second moment computations to the case $\rho\in(1,e)$ and show that they are both of order $\asymp 1$.

We expect that similar techniques employing the (weighted) second moment method can be adapted to the discrete-time version of the above model, where the discrete-time version of Lemma \ref{thm:bingham} is given in \citep{kac1947random}. Our approach might also generalize to BRW on the $b$-ary cube $\{0,1,\dots,b-1\}^d$. However, we suspect that, in all the possible extensions above, the argument will be much more technical.

\section{Discussions and open questions}\label{sec:discussion}

Our work focuses on computing the asymptotics of first passage times of the continuous-time branching random walk on the hypercube $\{0,1,\dots,b-1\}^d$ (BRWSS) as $d\to\infty$. Given that the random walk has a unit transition rate, we provide tight asymptotics in two regimes: when the branching rate is constant, and when it is a regularly varying function in $d$ that decays more slowly than $1/d$. We conclude by discussing several open directions.

\subsection{The biological perspective}

Our BRWSS model serves as a basic framework for several extensions, which could potentially lead to more realistic and calibratable models of biological interest. We propose three directions as follows. 

\begin{enumerate}[(i)]
    \item Empirically, mutagenic nucleoside analogues such as molnupiravir are known not only to increase the overall mutation rate but also to induce strongly biased mutational spectra. For example, molnupiravir promotes transition mutations (as opposed to transversion mutations), primarily G-to-A and C-to-T \citep{gordon2021molnupiravir,sanderson2023molnupiravir}. This motivates the study of asymmetric (biased) BRWSS. In other words, the Jukes--Cantor model can be replaced by its extensions such as the Kimura and Tamura models \citep{kimura1980simple,tamura1992estimation}. We expect that the same (weighted) second moment techniques will lead to analogous asymptotic analysis.
    \item The mutation--selection balance is fundamental in quasispecies theory. To provide a full picture of viral quasispecies evolution, a natural model is BRWSS in \textit{random environments} \citep{gartner2005parabolic,konig2020branching}. In this setting, one samples i.i.d.~branching rates for each vertex of the hypercube in advance, and the BRWSS evolves according to the inhomogeneous rates. The frequency distribution over genotypes in the large-population limit is described by the parabolic Anderson model \citep{konig2020branching}. The work \citep{avena2020parabolic} describes, in the large-population limit, how long it takes for the main mass of the particle system to move to the fittest site. The authors provide the leading scale for Gaussian-like branching rates using a spectral analysis. Note that these results do not directly correspond to the FPT analysis, since their setting is mean-field. We leave the analysis of FPTs of BRWSS in random environments to future work.
    \item The exponential growth of the population, even at the master sequence, is an idealization. In practice, one would also like to understand the clearance phase of infection, where the death rate of the virus increases as host immunocompetence develops \citep{lobinska2023evolutionary,nowak2000virus}.  This motivates studying a multi-phase variation of BRWSS with time-inhomogeneous (e.g., piecewise constant) branching rates, leading to more realistic scenarios of within-host viral evolution.
\end{enumerate}

\subsection{The mathematical perspective}

In addition to the above practically motivated variants, we also point out a few mathematical challenges in understanding the basic BRWSS model. We expect that these studies will inspire useful techniques for understanding branching processes on discrete structures.

\begin{enumerate}[(i)]
    \item Phase transition phenomena for branching particle systems have a rich literature. Notable examples include time-inhomogeneous BRW \citep{fang2012branching,mallein2015interfaces,ouimet2018maxima}, CLT for the empirical measure of branching Markov processes \citep{adamczak2015clt,ren2014central}, the Kesten--Stigum reconstruction bound \citep{kesten1966additional}, catalytic branching random walks \citep{mailler2024localisation}, and BRW on Galton--Watson trees \citep{pemantle2001branching}. We have shown a phase transition at $\rho=e$ related to sublethal mutagenesis: if $\rho>e$, the population preserves its self-identity; if $\rho<e$, the population delocalizes and drifts away from the origin. The transition coincides with a change in the asymptotic behavior of FPTs. The critical case $\rho=e$ remains to be studied, particularly in the context of FPTs.
    \item In the slow branching case $\rho\in(1,e)$, we have shown that the cover time $\tau_{\mathrm{cov}}$ is of order $\asymp d$, and the rescaled cover time is tight as $d\to\infty$. Finding precise asymptotics for the cover time remains a difficult problem---even deriving the leading order seems to require new techniques. In general, similarly to FPT, the cover time is a natural alternative to the maxima of one-dimensional spatial branching processes, for processes supported on spaces without an order structure \citep{roberts2022cover}. Results on cover times for BRW on discrete structures include regular trees \citep{roberts2022cover} and coalescing BRW \citep{balelli2016branching,dutta2015coalescing}.
    
    \item In our main results (Theorems \ref{thm:main}, \ref{thm:main2}, and \ref{thm:main rho->1}), we have assumed that $m$ is much smaller than $d$. Obtaining uniform asymptotics for larger ranges of $m$ remains a question of great interest. 
    We expect that the first moment does \textit{not} always predict the correct asymptotics, even if $\rho\in(1,e)$ and $m\in[0,d/2]$.
\end{enumerate}

\section*{Acknowledgement}
We thank Haotian Gu for valuable feedback. 
The material in this paper is partly supported by the Air Force Office of Scientific Research under award number FA9550-20-1-0397 and ONR N000142412655. Support from NSF 2229012, 2312204, 2403007 is also gratefully acknowledged. ZZ gratefully acknowledges support from a Jump Trading Fellowship.

\bibliographystyle{plain} 
\bibliography{Submission_v1/reference}  


\newpage 
\appendix

\begin{center}
    \Huge{\bf Appendix}
\end{center}
\medskip

In the appendices, we provide proofs of our main results. We start by collecting a few useful elementary facts in Appendix \ref{sec:appendix}. Next, we study the first-moment equation \eqref{eq:1} in Appendix \ref{sec:expansions}, and in particular, the asymptotic expansion of its solution. Appendices \ref{sec:case1} and \ref{sec:case2} contain proofs for the constant $\rho$ and the ultra-slow branching cases, respectively. We conclude with the derivation of cover times in Appendix \ref{sec:cover times}.

\paragraph{Notation.} For a non-negative real number $n\geq 0$, we denote by $[n]$ the set of non-negative integers that are less than or equal to $n$.  We let $\#S$ denote the cardinality of a finite set $S$. We use the notation $\lawis$ to denote the law of a random variable (e.g., $\xi\lawis\Gamma(n,1)$ means that $\xi$ is Gamma-distributed with shape parameter $n$ and rate $1$).  We use $V_t$ to denote the collection of all particles at time $t$. For $v\in V_t$ for some $t\geq 0$, we use $\eta_v=\eta_{v}(t)\in\X$ to denote the location of the particle $v$ at time $t$. We write $A\ll B$ or $B\gg A$ if $A=O(B)$. We adopt the convention that a generic binomial coefficient $\binom{n}{k}$ is nonzero only if $k\in[n]$.

\section{Preliminary computations and elementary facts}\label{sec:appendix}

\begin{lemma}\label{lemma:bin bound}
    For each $d,m,\ell'\in\N$, it holds that $$\sum_\ell\sum_{i}\binom{m}{i}\binom{m-i}{i+\ell-\ell'}\binom{d-m}{i+\ell-m}\leq \binom{d}{\ell'}2^{\ell'}$$
\end{lemma}
\begin{proof}
    We have by the Vandermonde identity,
    \begin{align*}
        \sum_\ell\sum_{i}\binom{m}{i}\binom{m-i}{i+\ell-\ell'}\binom{d-m}{i+\ell-m}&=\sum_i \binom{m}{i}\binom{d-i}{\ell'-i}\\
        &=\sum_{k=0}^{\ell'}\binom{d-m}{\ell'-k}\binom{m}{k}2^{k}\leq \binom{d}{\ell'}2^{\ell'}.
    \end{align*}
    This completes the proof.
\end{proof}

\begin{lemma}\label{lemma:e0}
   For each fixed $c>0$ and $b\in\N_2$, the following statements hold.
   \begin{enumerate}[(i)]
       \item The function
$$s\mapsto b^s(1+(b-1)e^{-c})^{-s}(1+(b-1)e^{-cs}),\quad s\in[0,1]$$
attains its maximum equal to $b$ uniquely at the boundary points $s=0,1$.
\item The function
$$s\mapsto b^s(1+(b-1)e^{-c})^{-s}(1+(b-1)e^{-c(1-s)})(1+(b-1)e^{-cs})^2,\quad s\in[0,1]$$
attains its maximum $b^2(1+(b-1)e^{-c})$ uniquely at the boundary points $s=0,1$.
\item The function
\begin{align*}
    s\mapsto b^{s}(1+(b-1)e^{-c})^{-s}&((1+(b-1)e^{-c(1-s)})(1+(b-1)e^{-cs})^2\\
    &+(b-1)(1-e^{-c(1-s)})(1-e^{-cs})^2),\quad\quad\quad s\in[0,1]
\end{align*}
attains its maximum $b^2(1+(b-1)e^{-c})$ uniquely at the boundary points $s=0,1$.
   \end{enumerate}
   
\end{lemma}

\begin{proof}
We prove (iii), and (i) and (ii) will follow similarly. Denote this function by $h_c(s)$. It is easy to check that
$$h_c(0)=h_c(1)=b^2(1+(b-1)e^{-c}),$$
and hence it remains to show that $h_c$ is strictly convex on $[0,1]$ for each fixed $c>0$. Let $u(s)=e^{-c(1-s)}$, $v(s)=e^{-cs}$, and $w(s)=(\frac{b}{1+(b-1)e^{-c}})^s$. Note that every monomial in $u,v,w$ is convex in $s$. A simple calculation shows that
\begin{align*}
    h_c(s)&=w(s)((1+(b-1)u(s))(1+(b-1)v(s))^2+(b-1)(1-u(s))(1-v(s))^2)\\
    &=w(s)(b+2(b-1)bu(s)v(s)+(b-1)bv(s)^2+((b-1)^3-(b-1))u(s)v(s)^2).
\end{align*}
Observe that this is a positive linear combination of monomials in $u,v,w$, which must be convex in $s$. Therefore, $h_c(s)$ is convex and also strictly convex (as $w(s)v(s)^2$ is strictly convex), which completes the proof.
\end{proof}




\begin{lemma}\label{lemma:c}
Let $b\in\N_2$ and $c>0$. Define
$$
D_b(c):=
1+\frac{b}{(b-1)c}\,\log\Big(\frac{b}{1+(b-1)e^{-c}}\Big)
-\frac{2be^{-c}}{1+(b-1)e^{-c}}.
$$
Then
$$
0<
\frac{1}{D_b(c)}
<
\Big(\frac{1+(b-1)e^{-c}}{1-e^{-c}}\Big)^2.
$$
\end{lemma}

\begin{proof}
We first rewrite the logarithmic term as an integral:
$$
\log\Big(\frac{b}{1+(b-1)e^{-c}}\Big)
=
\int_0^c \frac{(b-1)e^{-u}}{1+(b-1)e^{-u}}\,\d u.
$$
Set
$$
f(u):=\frac{(b-1)e^{-u}}{1+(b-1)e^{-u}},\qquad u\ge 0.
$$
Then $f$ is strictly decreasing on $[0,\infty)$. Hence
$$
\frac1c\log\Big(\frac{b}{1+(b-1)e^{-c}}\Big)
=
\frac1c\int_0^c f(u)\,\d u
>
f(c)
=
\frac{(b-1)e^{-c}}{1+(b-1)e^{-c}}.
$$
Multiplying by $b/(b-1)$, we obtain
$$
\frac{b}{(b-1)c}\,\log\Big(\frac{b}{1+(b-1)e^{-c}}\Big)
>
\frac{be^{-c}}{1+(b-1)e^{-c}}.
$$
Therefore,
$$
D_b(c)
>
1+\frac{be^{-c}}{1+(b-1)e^{-c}}
-\frac{2be^{-c}}{1+(b-1)e^{-c}}
=
1-\frac{be^{-c}}{1+(b-1)e^{-c}}
=
\frac{1-e^{-c}}{1+(b-1)e^{-c}},
$$
and in particular, $D_b(c)>0$. 
Since $c>0$, we have
$$
0<\frac{1-e^{-c}}{1+(b-1)e^{-c}}<1,
$$
so
$$
D_b(c)
>
\frac{1-e^{-c}}{1+(b-1)e^{-c}}
>
\Big(\frac{1-e^{-c}}{1+(b-1)e^{-c}}\Big)^2.
$$
Taking reciprocals proves the claim.
\end{proof}

\begin{lemma}\label{lemma:md}
Let $b\in\N_2$ and $0<A<B<B'<\infty$ be constants. Then there exists $L_3>0$ large enough depending on $A,B,B',b$, such that uniformly in $m\leq d/L_3$,
\begin{align}
    \sum_{\ell\leq d/L_3}\Big(\frac{A}{B}\Big)^\ell \binom{d}{\ell}^{-1}\sum_i \binom{m}{i}\binom{d-m}{i+\ell-m}\Big(\frac{B(b-1+B')}{b-1}\Big)^{m-i}=O(1)\label{eq:AB}
\end{align}
as $d\to\infty$.
\end{lemma}

\begin{proof}
    Let $\ee=\ee(A,B,b)>0$ be such that 
    $$\frac{A}{B}\Big(\frac{B(b-1+B')}{b-1}\Big)^{\ee}<1.$$
We split the second sum over $i$ in \eqref{eq:AB} into $i\in[0,m-\ee\ell)$ and $i\in[m-\ee\ell,m]$. 

In the former case, we show that the term $i=\lfloor m-\ee\ell\rfloor$ dominates the sum. To see this, let 
$$U_i:=\binom{m}{i}\binom{d-m}{i+\ell-m}\Big(\frac{B(b-1+B')}{b-1}\Big)^{m-i}.$$
Then
$$\frac{U_{i+1}}{U_i}=\frac{(m-i)(d-i-\ell)(b-1)}{(i+1)(i+1+\ell-m)B(b-1+B')}.$$
Since $i\leq m\leq d/L_3$, we have for $L_3$ large enough (possibly depending on $b$ and $B$) that $U_{i+1}/U_i\geq 1$ for $i\leq m-\ee\ell$. 
Therefore, with $L_3$ chosen large enough (depending on $\ee$, which only depends on $A,B$), \eqref{eq:AB} becomes (with some $C(\ee)>0$)
\begin{align*}
        &\hspace{0.5cm}\sum_{\ell\leq d/L_3}\Big(\frac{A}{B}\Big)^\ell \binom{d}{\ell}^{-1}\sum_i \binom{m}{i}\binom{d-m}{i+\ell-m}\Big(\frac{B(b-1+B')}{b-1}\Big)^{m-i}\\
        &\ll \sum_{\ell\leq d/L_3}\Big(\frac{A}{B}\Big)^\ell \binom{d}{\ell}^{-1}\ell \binom{m}{\lceil\ee\ell\rceil}\binom{d-m}{\lceil(1-\ee)\ell\rceil}\Big(\frac{B(b-1+B')}{b-1}\Big)^{\ee\ell}\\
        &\leq \sum_{\ell\leq d/L_3}\Big(\frac{A}{B}\Big)^\ell \Big(\frac{B(b-1+B')}{b-1}\Big)^{\ee\ell}C(\ee)^\ell\Big(\frac{d}{m}\Big)^{-\ee\ell}\ll 1,
    \end{align*}
    as desired.

    In the latter case, we apply the trivial bound
\begin{align*}
    &\hspace{0.5cm}\sum_{\ell\leq d/L_3}\Big(\frac{A}{B}\Big)^\ell \binom{d}{\ell}^{-1}\sum_{m-\ee\ell\leq i\leq m} \binom{m}{i}\binom{d-m}{i+\ell-m}\Big(\frac{B(b-1+B')}{b-1}\Big)^{m-i}\\
    &\leq \sum_{\ell\leq d/L_3} e^{-\delta_{13}\ell}\binom{d}{\ell}^{-1}\sum_{m-\ee\ell\leq i\leq m} \binom{m}{i}\binom{d-m}{i+\ell-m}\\
    &\leq \sum_{\ell\leq d/L_3} e^{-\delta_{13}\ell}\\
    &\ll 1,
\end{align*}for some $\delta_{13}>0$. This
completes the proof.
\end{proof}

\begin{lemma}
    \label{lemma:e3}
    It holds for all $x>0$ that
    $$\Big(\frac{1+e^{-2x}}{2}\Big)^2\Big(1+\Big(\frac{1-e^{-2x}}{1+e^{-2x}}\Big)^2\Big)< 1.$$
\end{lemma}

\begin{proof}
Since $1+e^{-2x}\in(1,2)$ for all $x>0$, we have
$$\Big(\frac{1+e^{-2x}}{2}\Big)^2\Big(1+\Big(\frac{1-e^{-2x}}{1+e^{-2x}}\Big)^2\Big)<\Big(\frac{1+e^{-2x}}{2}\Big)^2\Big(1+\Big(\frac{1-e^{-2x}}{1+e^{-2x}}\Big)\Big)=\frac{1+e^{-2x}}{2}<1,$$
as desired.    
\end{proof}

The following lemma is elementary, and its proof, which uses a direct Taylor expansion argument, will be omitted.
\begin{lemma}\label{lemma:taylor}
Fix $b\in\N_2$.    The following statements hold.
    \begin{enumerate}[(i)]
        \item For any $\ee>0$, there exists $\delta>0$ such that uniformly for $u\in[0,\delta d],$
        $$\frac{1+(b-1)e^{-\frac{bu}{(b-1)d}}}{b}\leq e^{-(1-\ee)u/d}.$$
        Moreover, for all $u\geq 0$, $$\frac{1-e^{-\frac{bu}{(b-1)d}}}{b}\leq \frac{u}{(b-1)d}.$$
        \item For all $t,d>0$ and $u\in[0,t]$, we have
$$\frac{1+(b-1)e^{-\frac{bu}{(b-1)d}}}{1+(b-1)e^{-\frac{bt}{(b-1)d}}}\leq \exp\Big(\frac{(b-1)e^{-\frac{bt}{(b-1)d}}(e^{\frac{b(t-u)}{(b-1)d}}-1)}{1+(b-1)e^{-\frac{bt}{(b-1)d}}}\Big)$$
and
$$\frac{1-e^{-\frac{bu}{(b-1)d}}}{1-e^{-\frac{bt}{(b-1)d}}}\leq \exp\Big(-\frac{e^{-\frac{bt}{(b-1)d}}(e^{\frac{b(t-u)}{(b-1)d}}-1)}{1-e^{-\frac{bt}{(b-1)d}}}\Big).$$        \item For any $L>0$, there exists $\delta>0$ such that uniformly for $u\in[0,L d],$
        $$\frac{1+(b-1)e^{-\frac{bu}{(b-1)d}}}{b}\leq e^{-\delta u/d}.$$
    \end{enumerate}
\end{lemma}

Recall from \eqref{eq:tdm:rho>e} that 
$$t_{d,m}:=\frac{mW(\frac{(\log\rho-1)(b-1)d}{m})}{\log\rho-1}.$$

\begin{lemma}\label{lemma:key}
 Let $b\in\N_2$ and $\rho>e$ where $\rho$ does not depend on $d$.   There exists $L_1>0$, depending on $b,\rho$, such that the following statements hold regarding the quantity $t_{d,m}$ defined in \eqref{eq:tdm:rho>e}.
    \begin{enumerate}[(i)]
        \item Uniformly in $m\in[1,\sqrt{d}]$,
    \begin{align*}
    t_{d,m}=\frac{m(\log d-\log m-\log\log (\max\{d/m,e\})+O(1))}{\log\rho-1}.
\end{align*}
\item Uniformly in $m\in[1,d/L_1]$, $t_{d,m}\leq d/100$.
\item  Uniformly in $m\in[1,d/L_1]$, $t_{d,m+1}-t_{d,m}\asymp t_{d,m}/m.$ 
    \item  Uniformly in $m\in[1,d/L_1]$, $t_{d,m+C}-t_{d,m}\to\infty$ as $C\to\infty$.
    \end{enumerate}
\end{lemma}

\begin{proof}
    (i)   This follows from the standard asymptotic expansion $W(x)=\log x-\log\log x+o(1)$ of the Lambert W function as $x\to\infty$ (see \citep{corless1996lambert}).

    (ii)   The inequality $t_{d,m}\leq d/100$ is equivalent to $W(x)\leq x/(100(b-1))$ for $x=(\log\rho-1)(b-1)d/m$. If $L_1$ is picked large enough, $x$ is also large enough, so that $e^{W(x)}\geq 100(b-1)$ and $W(x)\leq x/(100(b-1))$ holds.

    (iii) Using the defining relationship of the Lambert W function, we have uniformly for $x$ large enough,
\begin{align}
    W'(x)=\frac{W(x)}{x(1+W(x))}\asymp \frac{1}{x}.\label{eq:W'}
\end{align}
It follows from \eqref{eq:W'} and the intermediate value theorem that for $m\in[1,d/L_1]$ and $d$ large enough,
\begin{align*}
    t_{d,m+1}-t_{d,m}&= \frac{W(\frac{(\log\rho-1)(b-1)d}{m+1})}{\log\rho-1}-\frac{m}{\log\rho-1}\Big(W(\frac{(\log\rho-1)(b-1)d}{m})-W(\frac{(\log\rho-1)(b-1)d}{m+1})\Big)\\
    &\asymp \frac{t_{d,m}}{m}-\frac{m}{\log\rho-1}\log(1+\frac{1}{m})\\
    &\asymp \frac{t_{d,m}}{m}-O(1),
\end{align*}
where in the second step we used that by \eqref{eq:tdm:rho>e}, if $L_1$ is chosen large enough, $t_{d,m}/m$ can be made large enough uniformly for $m\in[1,d/L_1]$.

(iv) Recall from the proof of Lemma \ref{lemma:key}(iii) that $t_{d,m}/m\to\infty$ as $L_1\to\infty$.
    It follows from the proof of Lemma \ref{lemma:key}(iii) that with $L_1$ chosen large enough, $t_{d,m+1}-t_{d,m}\gg t_{d,m}/m-O(1)\geq 2$ uniformly for $m\in[1,d/L_1]$. 
\end{proof}

\section{Solving the first-moment equation \texorpdfstring{\eqref{eq:1}}{}}\label{sec:expansions}
In this section, we justify that the first-moment equation \eqref{eq:1} admits a unique positive solution in the regimes $\rho\in(1,e)$ (where $\rho$ does not depend on $d$) and $\rho\to 1^+$ (where $\rho$ may depend on $d$), and develop asymptotic expansions of the solution to \eqref{eq:1}.

\subsection{Uniqueness of the solution}
\begin{lemma}\label{lemma:nozero}
 Suppose that $b\in\N_2$.   In both regimes $\rho\in(1,e)$ and $\rho\to 1^+$, there exists $\delta_1>0$, depending on $b,\rho$ but independent of $d$ and $m$, such that there is no solution $t$ to \eqref{eq:1} satisfying $t\in(0,\delta_1 d)$.
\end{lemma}

\begin{proof}
    We choose $\delta_1$ small enough such that for all $\delta<\delta_1$, $\rho^\delta (1+(b-1)e^{-\frac{b\delta}{b-1}})<b$, whose existence is guaranteed by Lemma \ref{lemma:taylor}(i). Then, for $t\in(0,\delta_1 d)$, we have
    $$\rho^tb^{-d}\big(1+(b-1)e^{-\frac{bt}{(b-1)d}}\big)^{d-m}\big(1-e^{-\frac{bt}{(b-1)d}}\big)^m\leq \rho^tb^{-d}(1+(b-1)e^{-\frac{bt}{(b-1)d}})^{d}<1,$$
    and hence $t$ cannot be a solution to \eqref{eq:1}.
\end{proof}

\begin{proposition}\label{prop:uniquesoln}
Suppose that $b\in\N_2$. There exists a constant $L_1>0$ large enough such that for $1\leq m\leq d/L_1$ in both regimes $\rho\in(1,e)$ and $\rho\to 1^+$, there exists a unique strictly positive solution to \eqref{eq:1}.
\end{proposition}

\begin{proof}
First, observe that \eqref{eq:1} is equivalent to
\begin{align*}
    \phi(t):=t\log\rho-d\log b+(d-m)\log(1+(b-1)e^{-\frac{bt}{(b-1)d}})+m\log(1-e^{-\frac{bt}{(b-1)d}})=0.
\end{align*}
Differentiation yields that with $u=e^{-\frac{bt}{(b-1)d}}\in(0,1)$,
\begin{align}
    \phi'(t)=\log\rho-\frac{bu}{1+(b-1)u}+\frac{b^2u}{(b-1)(1+(b-1)u)(1-u)}\frac{m}{d}.\label{eq:phi'}
\end{align}
Let $\delta_1$ be given by Lemma \ref{lemma:nozero}. Note that on the interval $u\in[0,e^{-b\delta_1/(b-1)}]$, the second derivative of $\frac{b^2u}{(b-1)(1+(b-1)u)(1-u)}$ is uniformly bounded, and the second derivative of $\frac{bu}{1+(b-1)u}$ is negative and bounded away from zero. Therefore, if $L_1$ is chosen large enough, the right-hand side of \eqref{eq:phi'} is a $C^1$ and convex function of $u\in[0,e^{-b\delta_1/(b-1)}]$. Moreover, it takes positive value as $u\to 0^+$. 

By Lemma \ref{lemma:nozero}, we know that for all $t\in(0,\delta_1d)$, $\phi(t)<0$. Next, we claim that $\phi'(\delta_1 d)<0$. Indeed, this follows directly from \eqref{eq:phi'}, the fact that $bu/(1+(b-1)u)\to 1$ as $u\to 1$, and our assumptions that $\rho<e$ and $m\leq d/L_1$, with $L_1$ chosen large enough.   By the convexity of $\phi'$ as a function of $u$ above and since $u=e^{-\frac{bt}{(b-1)d}}$ is decreasing in $t$, we must have that there exists some $s\geq \delta_1d$ such that $\phi'(t)<0$ on $(\delta_1d, s)$ and $\phi'(t)>0$ on $(s,\infty)$. Altogether, we conclude that there is a unique positive root for $\phi(t)=0$. 
\end{proof}



\subsection{Asymptotic expansions}
In what follows, we develop asymptotic expansions for the solution $t$ to \eqref{eq:1} and prove Propositions \ref{thm:const rho expansion} and \ref{thm:non-const rho expansion}. The asymptotic results are always interpreted as $d\to\infty$. Some of the results here will be useful in proving Theorems \ref{thm:main} and \ref{thm:main rho->1}.

\begin{proposition}\label{lemma:asymp 0} Suppose that $b\in\N_2$, $\rho\in(1,e)$, $1\leq m\leq d/L_1$, and $t=t_{d,m}$ is the unique solution to \eqref{eq:1}. Then $t\asymp d$. Furthermore, $t_{d,m+1}-t_{d,m}=O(1)$ uniformly in $m\in[1,d/L_1]$.
\end{proposition}
\begin{proof}
Since the solution is unique (Proposition \ref{prop:uniquesoln}), to prove $t\asymp d$ it remains to show that the left-hand side of \eqref{eq:1} is $<1$ uniformly as $t/d\to 0$ and is $>1$ uniformly as $t/d\to\infty$. The former claim follows from the same argument as in Lemma \ref{lemma:nozero}; the latter claim follows since for $m\leq d/L_1$, uniformly as $t/d\to\infty$,
\begin{align*}
    &\hspace{0.5cm}\rho^tb^{-d}\big(1+(b-1)e^{-\frac{bt}{(b-1)d}}\big)^{d-m}\big(1-e^{-\frac{bt}{(b-1)d}}\big)^m\\
    &\geq\rho^tb^{-d}(1+(b-1)e^{-\frac{bt}{(b-1)d}})^{(1-1/L_1)d}(1-e^{-\frac{bt}{(b-1)d}})^{d/L_1}\geq  \rho^tb^{-d}>1.
\end{align*}
This proves that $t\asymp d$.

To see that $t_{d,m+1}-t_{d,m}=O(1)$, we use the shorthand notation $t'=t_{d,m+1}$ and $t=t_{d,m}$. Clearly, $t'>t$. Applying \eqref{eq:1} twice yields
\begin{align}
    (t'-t)\log\rho=\log\Big(\frac{1+(b-1)e^{-\frac{bt'}{(b-1)d}}}{1-e^{-\frac{bt'}{(b-1)d}}}\Big)+(d-m)\log\Big(\frac{1+(b-1)e^{-\frac{bt}{(b-1)d}}}{1+(b-1)e^{-\frac{bt'}{(b-1)d}}}\Big)+m\log \Big(\frac{1-e^{-\frac{bt}{(b-1)d}}}{1-e^{-\frac{bt'}{(b-1)d}}}\Big).\label{eq:tunique}
\end{align}
Using Lemma \ref{lemma:taylor}(ii), we have  
\begin{align*}
    (d-m)\log\Big(\frac{1+(b-1)e^{-\frac{bt}{(b-1)d}}}{1+(b-1)e^{-\frac{bt'}{(b-1)d}}}\Big)+m\log \Big(\frac{1-e^{-\frac{bt}{(b-1)d}}}{1-e^{-\frac{bt'}{(b-1)d}}}\Big)&\leq d\log\Big(\frac{1+(b-1)e^{-\frac{bt}{(b-1)d}}}{1+(b-1)e^{-\frac{bt'}{(b-1)d}}}\Big)\\
    &\leq \frac{(b-1)d(e^{\frac{b(t'-t)}{(b-1)d}}-1)}{b-1+e^{\frac{bt'}{(b-1)d}}}.
\end{align*}
This leads to
\begin{align}
    (t'-t)\log\rho\leq \log\Big(\frac{1+(b-1)e^{-\frac{bt'}{(b-1)d}}}{1-e^{-\frac{bt'}{(b-1)d}}}\Big)+\frac{(b-1)d(e^{\frac{b(t'-t)}{(b-1)d}}-1)}{b-1+e^{\frac{bt'}{(b-1)d}}}.\label{eq:b2}
\end{align}
By continuity and \eqref{eq:h'} below (whose proof is independent of the current proof), we see that if $L_1$ is large enough, there is $\delta'>0$ such that
 \begin{align}
     \log\rho-\frac{b}{b-1+e^{\frac{bt'}{(b-1)d}}}>\delta'.\label{eq:b1}
 \end{align}
Define the function
$$q(s):=s\log\rho-\frac{(b-1)d(e^{\frac{bs}{(b-1)d}}-1)}{b-1+e^{\frac{bt'}{(b-1)d}}}.$$
Then for $s\in[0,\sqrt{d}]$ and $d$ large enough, we have by \eqref{eq:b1} that $q'(s)>\delta'/2$. On the other hand, \eqref{eq:b2} and the fact that $t'\asymp d$ implies $q(t'-t)\ll 1$. Since $q(0)=0$, this implies  $t'-t=O(1)$ as $t$ is the unique solution to \eqref{eq:tunique} with $t'$ fixed (Proposition \ref{prop:uniquesoln}).
\end{proof}

\begin{proof}[Proof of Proposition \ref{thm:const rho expansion}] (i) follows from Proposition \ref{prop:uniquesoln}. 

(ii) Let us define the function
$$h(x)=x\log \rho-\log b+\log(1+(b-1)e^{-\frac{bx}{b-1}}).$$
It is easy to compute that $h(0)=0$, $h'(x)=\log\rho-\frac{b}{b-1}(1+\frac{1}{b-1}e^{\frac{bx}{b-1}})^{-1}$, and $h''(x)>0$, so $h$ is strictly convex, and we denote its two roots by $0$ and $x_0>0$ (so that $x_0$ in \eqref{eq:x0  def} is well-defined). It is also clear from the expression of $h'$ that $\lim_{\rho\to 1^+}x_0(b,\rho)=\infty$ and $\lim_{\rho\to e^-}x_0(b,\rho)=0$. 
In addition, 
\begin{align}
   0< h'(x_0)=\log\rho-\frac{b}{b-1+e^{\frac{bx_0}{b-1}}}.\label{eq:h'}
\end{align}
In particular, $r$ in \eqref{eq:r   def} is well-defined. 

Next, we show that $x_0$ is decreasing in $\rho$. Differentiating \eqref{eq:x0  def} with respect to $\rho$ yields
$$\Big(\log\rho-\frac{b}{b-1+e^{\frac{bx_0}{b-1}}}\Big)\frac{\d x_0}{\d\rho}=-\frac{x_0}{\rho}.$$
By \eqref{eq:h'}, we have $\d x_0/\d\rho<0$. 


To prove \eqref{eq:x0r limits}, it remains to show that $\lim_{\rho\to 1^+}r(b,\rho)=0$ and $\lim_{\rho\to e^-}r(b,\rho)=\infty$. Denote by $\alpha=\alpha(\rho)=e^{-\frac{bx_0}{b-1}}$, so that \eqref{eq:r   def} implies
\begin{align}
    r=\frac{\log\Big(\frac{1+(b-1)\alpha}{1-\alpha}\Big)}{\log\rho-\frac{b}{b-1+1/\alpha}}.\label{eq:rexp}
\end{align}
 To derive the limits of $r$, we determine the asymptotic orders of the numerator and the denominator of \eqref{eq:rexp} in $\alpha$. If $\rho\to 1^+$, we have $\alpha\to 0$. First, $\log\big(\frac{1+(b-1)\alpha}{1-\alpha}\big)=\log(1+O(\alpha))=O(\alpha)$. Second, inserting the relation $x_0=-\frac{(b-1)\log\alpha}{b}$ into \eqref{eq:x0  def} yields $(-\log\alpha)\log\rho=C(b)+O(\alpha)$ for some constant $C(b)>0$, so that $\log\rho=(1+o(1))C(b)/(-\log\alpha)$. 
 Third, $\frac{b}{b-1+1/\alpha}=O(\alpha)$. Altogether,
 $$\lim_{\rho\to 1^+}r(b,\rho)=\lim_{\alpha\to 0}\frac{O(\alpha)}{\frac{(1+o(1))C(b)}{-\log\alpha}+O(\alpha)}=0.$$
 On the other hand, if $\rho\to e^-$, we have $\alpha\to 1$. First, $\log\big(\frac{1+(b-1)\alpha}{1-\alpha}\big)=\log b+o(1)-\log(1-\alpha)$. Second, $\log\rho\to 1$ and $\frac{b}{b-1+1/\alpha}\to 1$, so $\log\rho-\frac{b}{b-1+1/\alpha}\to 0^+$ by \eqref{eq:h'}.  
 Altogether,
 $$\lim_{\rho\to e^-}r(b,\rho)=\lim_{\alpha\to 1}\frac{\log b+o(1)-\log(1-\alpha)}{o(1)}=\infty.$$

(iii) Recall \eqref{eq:x0  def}. 
We first show that $t/d-x_0=O(m/d)$. To see this, note first that \eqref{eq:1} implies 
\begin{align}
    \frac{t}{d}\log\rho-\log b+\frac{d-m}{d}\log(1+(b-1)e^{-\frac{bt}{(b-1)d}})+\frac{m}{d}\log(1-e^{-\frac{bt}{(b-1)d}})=0.\label{eq:2 implies}
\end{align}
Subtracting \eqref{eq:x0  def} from \eqref{eq:2 implies} gives
\begin{align}
    \Big(\frac{t}{d}-x_0\Big)\log\rho+\log\Big(\frac{1+(b-1)e^{-\frac{bt}{(b-1)d}}}{1+(b-1)e^{-\frac{bx_0}{b-1}}}\Big)+\frac{m}{d}\,\log\Big(\frac{1-e^{-\frac{bt}{(b-1)d}}}{1+(b-1)e^{-\frac{bt}{(b-1)d}}}\Big)=0.\label{eq:lhs}
\end{align}
If $t/d=x_0$, the left-hand side of \eqref{eq:lhs} is negative; using \eqref{eq:h'} and a Taylor expansion, if $t/d=x_0+Km/d$ with $K$ large enough, the left-hand side of \eqref{eq:lhs} is positive. By Proposition \ref{prop:uniquesoln} and continuity, the unique solution $t$ must satisfy $t/d-x_0=O(m/d)$.

Recall now \eqref{eq:r   def}. 
Let us write
$$\frac{t}{d}=x_0+(r+r')\frac{m}{d},$$
where $r'$ is some error term which we want to prove is small. We have by using \eqref{eq:x0  def}, \eqref{eq:2 implies}, the fact that $t/d-x_0=O(m/d)$, and Taylor's expansion that
\begin{align*}
    (r+r')\frac{m}{d}\log\rho&=\frac{m}{d}\log\Big(\frac{1+(b-1)e^{-\frac{bt}{(b-1)d}}}{1-e^{-\frac{bt}{(b-1)d}}}\Big)+\log\Big(\frac{1+(b-1)e^{-\frac{bx_0}{(b-1)}}}{1+(b-1)e^{-\frac{bt}{(b-1)d}}}\Big)\\
    &=\frac{m}{d}\bigg(\log\Big(\frac{1+(b-1)e^{-\frac{bx_0}{(b-1)}}}{1-e^{-\frac{bx_0}{(b-1)}}}\Big)+O\Big(\frac{m}{d}\Big)\bigg)+\bigg((r+r')\frac{m}{d}\frac{b}{b-1+e^{\frac{bx_0}{b-1}}}+O\Big(\frac{m^2}{d^2}\Big)\bigg)\\
    &=\frac{m}{d}\bigg(\log\Big(\frac{1+(b-1)e^{-\frac{bx_0}{(b-1)}}}{1-e^{-\frac{bx_0}{(b-1)}}}\Big)+\frac{b(r+r')}{b-1+e^{\frac{bx_0}{b-1}}}\bigg)+O\Big(\frac{m^2}{d^2}\Big).
\end{align*}
Along with \eqref{eq:r   def}, this leads to
$$r'\Big(\log\rho-\frac{b}{b-1+e^{\frac{bx_0}{b-1}}}\Big)=O\Big(\frac{m}{d}\Big),$$
as desired.
\end{proof}

\begin{proposition}\label{lemma:asymp 2}
 Assume that $b=2$.    Suppose that $1\leq m\leq d/L_3$ for $L_3$ large enough,  $\log\rho(d)=o(1)$, and $t=t_{d,m}$ is the unique solution to \eqref{eq:1}. Then
    \begin{align*}
    t=\frac{d\log 2}{\log\rho(d)}+O\Big(\frac{de^{-2\log 2/\log\rho(d)}}{\log\rho(d)}\Big).
\end{align*}
In particular, under the assumptions of Theorem \ref{thm:main rho->1}, we have
\begin{align*}
    t=\frac{d\log 2}{\log\rho(d)}+o(1).
\end{align*}
\end{proposition}

\begin{proof}
  Recall \eqref{eq:1}, which is equivalent to
    \begin{align}
        t\log\rho(d)-d\log 2+(d-m)\log(1+e^{-2t/d})+m\log(1-e^{-2t/d})=0.\label{eq:log 1}
    \end{align}
It is elementary to verify that if $t=d\log 2/\log\rho(d)$, then the left-hand side of \eqref{eq:log 1} is positive for $d$ large; if $t=(d\log 2-d\log\rho(d))/\log\rho(d)$, the left-hand side of \eqref{eq:log 1} is negative for large $d$. Since the solution is unique by Proposition \ref{prop:uniquesoln}, we must have $t=(d\log 2-h)/\log\rho(d)$, where $h\in[0,d\log\rho(d)]$ for large $d$. It follows from \eqref{eq:log 1} and our assumption $m\leq d/L_3$ that
\begin{align*}
    h&=(d-m)\log(1+e^{-2t/d})+m\log(1-e^{-2t/d})\\
    &\leq d\log(1+e^{-2t/d})\ll de^{-2t/d}\ll de^{-2\log 2/\log\rho(d)}.
\end{align*}
This proves the claim.

Next, suppose that the assumptions of Theorem \ref{thm:main rho->1} hold, i.e., $\log\rho(d)$ is regularly varying of index $\gamma\in[-1,0)$ and $\log\rho(d)\geq L_2/d$ for some large enough constant $L_2>0$. Then by Potter's bound (Theorem 1.5.6 of \citep{bingham1989regular}), for any $\ee>0$, $\rho(d)\leq e^{\ee/\log d}$ for $d$ large enough. It follows that for $d$ large enough,
$$\frac{de^{-2\log 2/\log\rho(d)}}{\log\rho(d)}\ll d^2e^{-10\log d}=o(1),$$
as desired.
\end{proof}

Combining Propositions \ref{prop:uniquesoln} and \ref{lemma:asymp 2} yields Proposition \ref{thm:non-const rho expansion}.

\section{Proof for the constant \texorpdfstring{$\rho$}{} case}\label{sec:case1}

\subsection{Proof for the slow branching case \texorpdfstring{$\rho\in(1,e)$}{}}
\subsubsection{First moment computation}\label{eq:lb1}

We let $N_\z(t)=N_\z^{(d,m)}(t)$ denote the number of particles at location $\z$ at time $t$ in the branching random walk on $\X_d^{(b)}$ starting from the state $\bx_m$. Let also $\tilde{N}_\z(t)=\tilde{N}_\z^{(d,m)}(t)$ be the total amount of time spent at the origin $\z$ of all particles, i.e., $\tilde{N}_\z(t)=\int_0^tN_\z(s)\,\d s$. It follows from \eqref{eq:qij} that 
\begin{align}
    \E[\tilde{N}_\z(t)]=\int_0^t\E[N_\z(s)]\,\d s=\int_0^t\rho^s b^{-d}(1+(b-1)e^{-\frac{bs}{(b-1)d}})^{d-m}(1-e^{-\frac{bs}{(b-1)d}})^m\d s.\label{eq:tilde N}
\end{align}

\begin{lemma}\label{prop:first moment 1}
Suppose that $\rho\in(1,e)$. Let $t=t_{d,m}$ be the unique positive solution to \eqref{eq:1}.    For any $\ee>0$, there exists $C>0$ such that
\begin{align}\E[\tilde{N}_\z(t-C)]=\int_0^{t-C} \rho^sb^{-d}(1+(b-1)e^{-\frac{bs}{(b-1)d}})^{d-m}(1-e^{-\frac{bs}{(b-1)d}})^m\d s\leq\ee.
    \label{eq:int}
\end{align}
\end{lemma}

\begin{proof}
The first equality is \eqref{eq:tilde N}. 
We first separate the contribution from $s\in[0,\delta_2 d]$ for some $\delta_2$ depending only on $\delta_3:=(1-\log\rho)/b>0$. We have since $m\in[1, d/L_1]$ and $\rho\in(1,e)$, by choosing $\delta_2$ small enough and applying Lemma \ref{lemma:taylor}(i),
\begin{align*}
    \int_0^{\delta_2 d} \rho^sb^{-d}(1+(b-1)e^{-\frac{bs}{(b-1)d}})^{d-m}(1-e^{-\frac{bs}{(b-1)d}})^m\d s
    &=\int_0^{\delta_2 d} \rho^s\Big(\frac{1+(b-1)e^{-\frac{bs}{(b-1)d}}}{b}\Big)^{d-m}\Big(\frac{1-e^{-\frac{bs}{(b-1)d}}}{b}\Big)^m\d s\\
    &\leq \int_0^{\delta_2 d} \rho^se^{-(1-\delta_3)\frac{s(d-m)}{d}}\Big(\frac{s}{(b-1)d}\Big)^m\d s\\
    &\leq d^{-1}\int_0^{\delta_2 d} s\rho^se^{-(1-\delta_3)\frac{s(d-m)}{d}}\d s\ll d^{-1}.
\end{align*}

Next, we show that for any $\delta_2>0$, there exists some $\ee_1>0$ such that
\begin{align}
    \int_{\delta_2 d}^{(1-\delta_2)t} \rho^sb^{-d}(1+(b-1)e^{-\frac{bs}{(b-1)d}})^{d-m}(1-e^{-\frac{bs}{(b-1)d}})^m\d s\ll e^{-\ee_1d}.\label{eq:3}
\end{align}
By a change of variable and \eqref{eq:1}, we write
\begin{align}
    &\hspace{0.5cm}\int_{\delta_2 d}^{(1-\delta_2)t} \rho^s(1+(b-1)e^{-\frac{bs}{(b-1)d}})^{d-m}(1-e^{-\frac{bs}{(b-1)d}})^m\d s\nonumber\\
    &=t\int_{\delta_2 d/t}^{1-\delta_2}\Big(b^s(1+(b-1)e^{-\frac{bt}{(b-1)d}})^{-\frac{s(d-m)}{d}}(1-e^{-\frac{bt}{(b-1)d}})^{-\frac{sm}{d}}(1+(b-1)e^{-\frac{bst}{(b-1)d}})^{(d-m)/d}(1-e^{-\frac{bst}{(b-1)d}})^{m/d}\Big)^d\d s\nonumber\\
    &\ll \int_{\delta_2 d/t}^{1-\delta_2}\Big(b^se^\ee(1+(b-1)e^{-\frac{bt}{(b-1)d}})^{-s}(1+(b-1)e^{-\frac{bst}{(b-1)d}})\Big)^d\d s,\label{eq:4}
\end{align}
where $\ee\to 0$ as $L_1\to \infty$. By Lemma \ref{lemma:e0}(i) and Proposition \ref{lemma:asymp 0}, we see that in the range of integration, the integrand of \eqref{eq:4} is $<(b-\ee_1)^d$ for some $\ee_1>0$.  This establishes \eqref{eq:3}.

Finally, we show that the integrand grows at least exponentially for $s\in[(1-\delta_2)t,t]$. It remains to show that the derivative of the logarithm of the integrand in \eqref{eq:int} is bounded away from 0, i.e.,
\begin{align}
    \inf_{s\in[(1-\delta_2)t,t]}\Big(\log\rho-\frac{b}{e^{\frac{bs}{(b-1)d}}+(b-1)}+o(1)\Big)>0,\label{eq:inf>0}
\end{align}
 for some $\delta_2>0$ (where the $o(1)$ incorporates the terms involving $m$ and can be arbitrarily small if $L_1$ is large enough). But \eqref{eq:inf>0} follows immediately from \eqref{eq:h'} and Proposition \ref{thm:const rho expansion}(iii) if $\delta_2$ is chosen sufficiently small and $L_1$ chosen large by continuity. Therefore, the integrand grows exponentially in the region $s\in[(1-\delta_2)t,t]$, proving \eqref{eq:int}. 
\end{proof}

\begin{proof}[Proof of Theorem \ref{thm:main}, lower bound]
Let $\ee>0$ be fixed and our goal is to show that for some $C>0$, $\p(\tau_{d,m}\leq t_{d,m}-C)<\ee$. 
Pick $\delta_4>0$ such that $\int_0^{\delta_4} e^{-x}\d x<\ee/2$. 
Using Lemma \ref{prop:first moment 1}, we pick $C_1>0$ such that $\E[\tilde{N}_\z(t_{d,m}-C_1)]<\delta_4\ee/2$. By Markov's inequality, $\p(\tilde{N}_\z(t_{d,m}-C_1)\geq \delta_4)<\ee/2$. By definition, $\tilde{N}_\z(\tau_{d,m})=0$. By the strong Markov property of the branching random walk starting from $\tau_{d,m}$,
$$\p(\tilde{N}_\z(t_{d,m}-C_1)<\delta_4;\tau_{d,m}<t_{d,m}-C_1-\delta_4)\leq \p(\tilde{N}_\z(\tau_{d,m}+\delta_4)-\tilde{N}_\z(\tau_{d,m})<\delta_4)\leq 1-e^{-\delta_4}<\frac{\ee}{2}.$$
Therefore, by the union bound,
$$\p(\tau_{d,m}<t_{d,m}-C_1-\delta_4)\leq \p(\tilde{N}_\z(t_{d,m}-C_1)<\delta_4;\tau_{d,m}<t_{d,m}-C_1-\delta_4)+\p(\tilde{N}_\z(t_{d,m}-C_1)\geq \delta_4)<\ee.$$
This proves the lower bound of $\tau_{d,m}$.
\end{proof}

\subsubsection{Second moment computation}\label{sec:2nd}

We start with a combinatorial lemma on the count of vertices of a subset in the sequence space $\X$. 

\begin{lemma}\label{lemma:count2}
   Fix $b\in\N_2$,  $\ell,\ell',m\in[d]$, and $\bx\in \X_{d,m}$. The number of distinct $\by\in\X_{d,\ell'}$ satisfying $d_{\mathrm{H}}(\bx,\by)=\ell$ is given by\footnote{Here we use the convention $0^0=1$.}
   $$\sum_{i}\binom{m}{i}\binom{m-i}{i+\ell-\ell'}\binom{d-m}{i+\ell-m}(b-2)^{\ell'-\ell+m-2i}(b-1)^{i+\ell-m}.$$
   Moreover, we have the upper bound
    \begin{align*}
        &\sum_{i}\binom{m}{i}\binom{m-i}{i+\ell-\ell'}\binom{d-m}{i+\ell-m}(b-2)^{\ell'-\ell+m-2i}(b-1)^{i+\ell-m}\\
        &\leq \sum_{i}\binom{m}{i}\binom{m-i}{i+\ell-\ell'}\binom{d-m}{i+\ell-m}(b-1)^{\ell'-i}\leq \sum_{i}\binom{m}{i}\binom{m-i}{i+\ell-\ell'}\binom{d-m}{i+\ell-m}(b-1)^{\ell'}.
    \end{align*}
\end{lemma}

\begin{proof}
    Without loss of generality, we assume $\bx=(1,\dots,1,0,\dots,0)$ with $m$ ones and $d-m$ zeros. Let $i$ be the number of ones in the first $m$ entries of $\by$. Using our assumptions, we know that within the first $m$ entries of $\by$, $\ell-\ell'+i$ of them are zero and $\ell'-\ell+m-2i$ of them are neither one nor zero (i.e., with $b-2$ choices). Among the last $d-m$ entries of $\by$, $\ell-m+i$ of them are nonzero and $d-\ell-i$ of them are zero. Combining the above yields the desired formula.  
    To see the upper bound, note that for the binomial coefficient $\binom{m-i}{i+\ell-\ell'}$ to be nonzero, we must have $m-i\geq i+\ell-\ell'$, or equivalently $\ell'-\ell+m-2i\geq 0$. 
\end{proof}

\begin{proposition}\label{prop:second moment 1}
Suppose that  $b\in\N_2$ and $\rho\in(1,e)$.     There exist constants $L_1,L_4>0$ depending only on $b,\rho$ such that
$$\sup_{d\geq 1}\max_{1\leq m\leq  d/L_1}\E[N_\z(t)^2]\leq L_4,$$
where $t=t_{d,m}$ is the unique positive solution to \eqref{eq:1}.
\end{proposition}

\begin{proof}
Using Lemma \ref{lemma:count2} and the many-to-two formula for continuous-time spatial branching processes (see page 146 of \citep{ikeda1969branching} or Appendix II of \citep{sawyer1976branching} for the case of branching diffusions and our case follows similarly; we omit the details), the second moment $\E[N_\z(t)^2]$ is given by
\begin{align*}
    &\hspace{0.5cm}\E[N_\z(t)^2]\\
    &\leq \E[N_\z(t)]+2(\log\rho)\int_0^t \rho^{t+s}\sum_{\ell,\ell'}b^{-d}(1+(b-1)e^{-\frac{b(t-s)}{(b-1)d}})^{d-\ell}(1-e^{-\frac{b(t-s)}{(b-1)d}})^\ell\\
    &\hspace{1cm}\times(b^{-d}(1+(b-1)e^{-\frac{bs}{(b-1)d}})^{d-\ell'}(1-e^{-\frac{bs}{(b-1)d}})^{\ell'})^2\bigg(\sum_{i}\binom{m}{i}\binom{m-i}{i+\ell-\ell'}\binom{d-m}{i+\ell-m}(b-1)^{\ell'}\bigg)\d s\\
    &\ll 1+b^{-3d}\rho^t\sum_{\ell,\ell'}(b-1)^{\ell'}\sum_{i}\binom{m}{i}\binom{m-i}{i+\ell-\ell'}\binom{d-m}{i+\ell-m}\\
    &\hspace{1cm}\times\int_0^t \rho^s(1+(b-1)e^{-\frac{b(t-s)}{(b-1)d}})^{d-\ell}(1-e^{-\frac{b(t-s)}{(b-1)d}})^\ell\times((1+(b-1)e^{-\frac{bs}{(b-1)d}})^{d-\ell'}(1-e^{-\frac{bs}{(b-1)d}})^{\ell'})^2\d s.
\end{align*}
 In particular, the sum is over $(\ell,\ell')$ such that $(m,\ell,\ell')$ satisfies the triangle inequality (i.e., the sum of any two is at least the third). 
Our goal is to show that $\E[N_\z(t)^2]\ll 1$ uniformly for $m\leq  d/L_1$ as $d\to\infty$. In the following, we split the integral over $[0,t]$ into three parts: $[0,\delta_5 t]\cup[\delta_5 t,(1-\delta_5)t]\cup[(1-\delta_5)t,t]$ for some small $\delta_5>0$ to be determined, and bound them respectively in three steps.

\textbf{Step I: bounding the integral over $s\in[0,\delta_5 t]$}. 
First, we consider the contribution for a fixed $\ell'$. Using Lemma \ref{lemma:taylor}(i), we have for any $\ee_2>0$, we can find $\delta_5$ small enough such that
\begin{align}
    \begin{split}
        &\hspace{0.5cm}b^{-3d}\rho^t\sum_{m-\ell'\leq \ell\leq m+\ell'}(b-1)^{\ell'}\sum_{i}\binom{m}{i}\binom{m-i}{i+\ell-\ell'}\binom{d-m}{i+\ell-m}\\
    &\hspace{1cm}\times\int_0^{\delta_5 t} \rho^s(1+(b-1)e^{-\frac{b(t-s)}{(b-1)d}})^{d-\ell}(1-e^{-\frac{b(t-s)}{(b-1)d}})^\ell((1+(b-1)e^{-\frac{bs}{(b-1)d}})^{d-\ell'}(1-e^{-\frac{bs}{(b-1)d}})^{\ell'})^2\d s\\
    &\leq b^{-d}\rho^t\sum_{m-\ell'\leq \ell\leq m+\ell'}(b-1)^{\ell'}\sum_{i}\binom{m}{i}\binom{m-i}{i+\ell-\ell'}\binom{d-m}{i+\ell-m}\\
    &\hspace{1cm}\times\int_0^{\delta_5 t} \rho^s(1+(b-1)e^{-\frac{b(t-s)}{(b-1)d}})^{d-\ell}(1-e^{-\frac{b(t-s)}{(b-1)d}})^\ell\Big(\frac{s}{(b-1)d}\Big)^{2\ell'}e^{(-2(1-\ee_2)s/d)(d-\ell')}\d s.
    \end{split}\label{eq:ell'0}
\end{align}
Substituting the relation \eqref{eq:1} gives 
\begin{align*}
    &b^{-d}\rho^t\rho^s(1+(b-1)e^{-\frac{b(t-s)}{(b-1)d}})^{d-\ell}(1-e^{-\frac{b(t-s)}{(b-1)d}})^{\ell}\\
    &=\rho^s\Big(\frac{1+(b-1)e^{-\frac{b(t-s)}{(b-1)d}}}{1+(b-1)e^{-\frac{bt}{(b-1)d}}}\Big)^{d-m}\Big(\frac{1-e^{-\frac{b(t-s)}{(b-1)d}}}{1-e^{-\frac{bt}{(b-1)d}}}\Big)^{m} (1+(b-1)e^{-\frac{b(t-s)}{(b-1)d}})^{m-\ell}(1-e^{-\frac{b(t-s)}{(b-1)d}})^{\ell-m}\\
    &\leq \rho^s\Big(\frac{1+(b-1)e^{-\frac{b(t-s)}{(b-1)d}}}{1+(b-1)e^{-\frac{bt}{(b-1)d}}}\Big)^{d-m}\Big(\frac{1-e^{-\frac{b(t-s)}{(b-1)d}}}{1-e^{-\frac{bt}{(b-1)d}}}\Big)^{m}C_2^{|m-\ell|}
\end{align*}
for some constant $C_2>0$. By Lemmas \ref{lemma:bin bound} and \ref{lemma:taylor}(ii), we find that \eqref{eq:ell'0} is bounded by
\begin{align}
   & \sum_{m-\ell'\leq \ell\leq m+\ell'}(b-1)^{\ell'}\sum_{i}\binom{m}{i}\binom{m-i}{i+\ell-\ell'}\binom{d-m}{i+\ell-m}\int_0^{\delta_5 t} \rho^s\Big(\frac{1+(b-1)e^{-\frac{b(t-s)}{(b-1)d}}}{1+(b-1)e^{-\frac{bt}{(b-1)d}}}\Big)^{d-m}\nonumber\\
    &\hspace{4cm}\times\Big(\frac{1-e^{-\frac{b(t-s)}{(b-1)d}}}{1-e^{-\frac{bt}{(b-1)d}}}\Big)^{m}C_2^{|m-\ell|}\Big(\frac{s}{(b-1)d}\Big)^{2\ell'}e^{(-2(1-\ee_2)s/d)(d-\ell')}\d s\nonumber\\
    &\leq \sum_{m-\ell'\leq \ell\leq m+\ell'}(b-1)^{-\ell'}\sum_{i}\binom{m}{i}\binom{m-i}{i+\ell-\ell'}\binom{d-m}{i+\ell-m}\int_0^{\delta_5 t}  \rho^s\exp\Big(\frac{(d-m)(b-1)e^{-\frac{bt}{(b-1)d}}(e^{\frac{bs}{(b-1)d}}-1)}{1+(b-1)e^{-\frac{bt}{(b-1)d}}}\Big)\nonumber\\
    &\hspace{4cm}\times\exp\Big(-\frac{me^{-\frac{bt}{(b-1)d}}(e^{\frac{bs}{(b-1)d}}-1)}{1-e^{-\frac{bt}{(b-1)d}}}\Big) C_2^{|m-\ell|}\Big(\frac{s}{d}\Big)^{2\ell'}e^{(-2(1-\ee_2)s/d)(d-\ell')}\d s\nonumber\\
    &\leq C_3^{\ell'}d^{-2\ell'}\binom{d}{\ell'}\int_0^{\delta_5 t}  \rho^s\exp\Big(\frac{(d-m)(b-1)e^{-\frac{bt}{(b-1)d}}(e^{\frac{bs}{(b-1)d}}-1)}{1+(b-1)e^{-\frac{bt}{(b-1)d}}}\Big)\nonumber\\
    &\hspace{4cm}\times\exp\Big(-\frac{me^{-\frac{bt}{(b-1)d}}(e^{\frac{bs}{(b-1)d}}-1)}{1-e^{-\frac{bt}{(b-1)d}}}\Big)s^{2\ell'}e^{(-2(1-\ee_2)s/d)(d-\ell')}\d s\nonumber\\
    \begin{split}
    &\ll \Big(\frac{C_3}{d\ell'}\Big)^{\ell'}\int_0^{\delta_5 t}  \rho^s\exp\Big(\frac{(d-m)(b-1)e^{-\frac{bt}{(b-1)d}}(e^{\frac{bs}{(b-1)d}}-1)}{1+(b-1)e^{-\frac{bt}{(b-1)d}}}\Big)\\
    &\hspace{4cm}\times\exp\Big(-\frac{me^{-\frac{bt}{(b-1)d}}(e^{\frac{bs}{(b-1)d}}-1)}{1-e^{-\frac{bt}{(b-1)d}}}\Big)s^{2\ell'}e^{(-2(1-\ee_2)s/d)(d-\ell')}\d s,
   \end{split}\label{eq:ell'1}
\end{align}for some $C_3>0$, 
where in the last step we used Stirling's formula and by convention, we use $(C_3/(d\ell'))^{\ell'}=1$ if $\ell'=0$. 
Since 
$$\rho<e<
\exp\Big(2-\frac{be^{-\frac{bt}{(b-1)d}}}{1+(b-1)e^{-\frac{bt}{(b-1)d}}}\Big),$$
there exists $\delta_6\in(0,1)$ (depending only on $e-\rho$) such that the integrand in \eqref{eq:ell'1} is
\begin{align*}
    &\rho^s\exp\Big(\frac{(d-m)(b-1)e^{-\frac{bt}{(b-1)d}}(e^{\frac{bs}{(b-1)d}}-1)}{1+(b-1)e^{-\frac{bt}{(b-1)d}}}\Big)\exp\Big(-\frac{me^{-\frac{bt}{(b-1)d}}(e^{\frac{bs}{(b-1)d}}-1)}{1-e^{-\frac{bt}{(b-1)d}}}\Big)s^{2\ell'}e^{(-2(1-\ee_2)s/d)(d-\ell')}\\
&\leq \exp\Big(\frac{(d-m)(b-1)e^{-\frac{bt}{(b-1)d}}(e^{\frac{bs}{(b-1)d}}-\frac{(1+\delta_6)bs}{(b-1)d}-1)}{1+(b-1)e^{-\frac{bt}{(b-1)d}}}\Big)\exp\Big(-\frac{me^{-\frac{bt}{(b-1)d}}(e^{\frac{bs}{(b-1)d}}-\frac{bs}{(b-1)d}-1)}{1-e^{-\frac{bt}{(b-1)d}}}\Big)\\
&\hspace{6cm}\times s^{2\ell'}e^{2s-2((1-\ee_2)s/d)(d-\ell')}(1-\delta_6)^s\\
&\leq s^{2\ell'}e^{2s-2((1-\ee_2)s/d)(d-\ell')}(1-\delta_6)^s,
\end{align*}
and hence \eqref{eq:ell'1} is bounded by
$$\Big(\frac{C_3}{d\ell'}\Big)^{\ell'}\int_0^{\delta_5 t}s^{2\ell'}e^{2s-\frac{2(1-\ee_2)s(d-\ell')}{d}}(1-\delta_6)^s\d s.$$
If $\delta_5$ is small enough depending on $\delta_6$ ($e^{4\ee_2}(1-\delta_6)<1$ would be enough; recall that $\ee_2\to 0$ as $\delta_5\to 0$), 
this part of the second moment is bounded by
\begin{align}
    \begin{split}
        \Big(\frac{C_3}{d\ell'}\Big)^{\ell'}\int_0^{\delta_5 t}s^{2\ell'}e^{2s-\frac{2(1-\ee_2)s(d-\ell')}{d}}(1-\delta_6)^s\d s&\leq \Big(\frac{C_3}{d\ell'}\Big)^{\ell'}\int_0^{\delta_5 t}s^{2\ell'}e^{2s(\ee_2+\ell'/d)}(1-\delta_6)^s\d s\\
    &\leq \Big(\frac{C_3}{d\ell'}\Big)^{\ell'}\int_0^{\delta_5 t}s^{2\ell'}e^{-\delta_7s}\d s\ll \Big(\frac{C_3}{\delta_7^2d\ell'}\Big)^{\ell'}\Gamma(2\ell'+1)
    \end{split}\label{eq:gamma}
\end{align}
for some $\delta_7>0$ that depends on $\delta_6$ but not on $\delta_5$. Let $C_4=100C_3/\delta_7^2$ and $L_5>2C_4$. 
Summing \eqref{eq:gamma} over $\ell'\in[d/L_5]$, the contribution is
\begin{align*}
    \sum_{\ell'=0}^{d/L_5}\Big(\frac{C_3}{\delta_7^2d\ell'}\Big)^{\ell'}\Gamma(2\ell'+1)\ll \sum_{\ell'=0}^{d/L_5}\Big(\frac{C_4\ell'}{d}\Big)^{\ell'}\leq \sum_{\ell'=0}^{d/L_5}2^{-\ell'}=O(1).
\end{align*}
 Otherwise if $\ell'\geq d/L_5$, for a given $L_5>0$, we may bound the contribution by an explicit exponential decay if $\delta_5$ is chosen small enough depending on $L_5$. Indeed, for $\delta_5$ small enough, $(s/d)^{2\ell'}\leq (\delta_5 t/d)^{2d/L_5}\leq (10K(b,\rho))^{-d}$ for some large constant $K(b,\rho)>0$ to be determined, which depends only on $b,\rho$. Applying Lemma \ref{lemma:bin bound}, we have that the right-hand side of \eqref{eq:ell'0} is bounded from above by
 \begin{align}
    \begin{split}
    & \hspace{0.5cm}b^{-d}\rho^t\sum_{m-\ell'\leq \ell\leq m+\ell'}(b-1)^{\ell'}\sum_{i}\binom{m}{i}\binom{m-i}{i+\ell-\ell'}\binom{d-m}{i+\ell-m}\\
    &\hspace{1cm}\times\int_0^{\delta_5 t} \rho^s(1+(b-1)e^{-\frac{b(t-s)}{(b-1)d}})^{d-\ell}(1-e^{-\frac{b(t-s)}{(b-1)d}})^\ell\Big(\frac{s}{(b-1)d}\Big)^{2\ell'}e^{(-2(1-\ee_2)s/d)(d-\ell')}\d s\\
    &\ll b^{-d}\rho^t\sum_{m-\ell'\leq \ell\leq m+\ell'}10^d\int_0^{\delta_5 t} \rho^s(1+(b-1)e^{-\frac{b(t-s)}{(b-1)d}})^{d-\ell}(1-e^{-\frac{b(t-s)}{(b-1)d}})^{\ell} (10K(b,\rho))^{-d}\d s\\
    &\ll \rho^{2t} K(b,\rho)^{-d}\ll 2^{-d},
    \end{split}\label{eq:100^{-d}}
 \end{align}
 where the last step is due to Proposition \ref{thm:const rho expansion}(iii), with $K(b,\rho)$ chosen large enough.

\textbf{Step II: bounding the integral over $s\in[\delta_5 t,(1-\delta_5)t]$}. We first consider $\ell'\in[0,m]$. Using \eqref{eq:1} and Lemma \ref{lemma:bin bound},
\begin{align}
&b^{-3d}\rho^t \sum_{m-\ell'\leq \ell\leq m+\ell'}(b-1)^{\ell'}\sum_{i}\binom{m}{i}\binom{m-i}{i+\ell-\ell'}\binom{d-m}{i+\ell-m}\nonumber\\
        &\hspace{1cm}\times\int_{\delta_5 t}^{(1-\delta_5)t}\rho^s (1+(b-1)e^{-\frac{b(t-s)}{(b-1)d}})^{d-\ell}(1-e^{-\frac{b(t-s)}{(b-1)d}})^{\ell}\times(1+(b-1)e^{-\frac{bs}{(b-1)d}})^{2(d-\ell')}(1-e^{-\frac{bs}{(b-1)d}})^{2\ell'}\d s\nonumber\\
    &\ll b^{-3d}\rho^t\sum_{m-\ell'\leq \ell\leq m+\ell'}(b-1)^{\ell'}\sum_{i}\binom{m}{i}\binom{m-i}{i+\ell-\ell'}\binom{d-m}{i+\ell-m}\nonumber\\
    &\hspace{1cm}\times\int_{\delta_5 t}^{(1-\delta_5)t}\rho^s (1+(b-1)e^{-\frac{b(t-s)}{(b-1)d}})^{d-m+\ell'}(1-e^{-\frac{b(t-s)}{(b-1)d}})^{m-\ell'}(1+(b-1)e^{-\frac{bs}{(b-1)d}})^{2(d-\ell')}(1-e^{-\frac{bs}{(b-1)d}})^{2\ell'}\d s\nonumber\\
    &=  tb^{-3d}\rho^t\sum_{m-\ell'\leq \ell\leq m+\ell'}(b-1)^{\ell'}\sum_{i}\binom{m}{i}\binom{m-i}{i+\ell-\ell'}\binom{d-m}{i+\ell-m}\nonumber\\
    &\hspace{1.5cm}\times\int_{\delta_5}^{1-\delta_5}\Big(b^s(1+(b-1)e^{-\frac{bt}{(b-1)d}})^{-\frac{s(d-m)}{d}}(1-e^{-\frac{bt}{(b-1)d}})^{-\frac{sm}{d}}(1+(b-1)e^{-\frac{bt(1-s)}{(b-1)d}})^{1-m/d+\ell'/d}\nonumber\\
    &\hspace{3cm}\times(1-e^{-\frac{bt(1-s)}{(b-1)d}})^{(m-\ell')/d}(1+(b-1)e^{-\frac{bst}{(b-1)d}})^{2(d-\ell')/d}(1-e^{-\frac{bst}{(b-1)d}})^{2\ell'/d}\Big)^d\d s\nonumber\\
\begin{split}
    &\ll tb^{-3d}\rho^t\binom{d}{\ell'}(2(b-1))^{\ell'}\\
&\hspace{1.5cm}\times\int_{\delta_5}^{1-\delta_5}\Big(b^s(1+(b-1)e^{-\frac{bt}{(b-1)d}})^{-\frac{s(d-m)}{d}}(1-e^{-\frac{bt}{(b-1)d}})^{-\frac{sm}{d}}(1+(b-1)e^{-\frac{bt(1-s)}{(b-1)d}})^{1-m/d+\ell'/d}\\
    &\hspace{3cm}\times(1-e^{-\frac{bt(1-s)}{(b-1)d}})^{(m-\ell')/d}(1+(b-1)e^{-\frac{bst}{(b-1)d}})^{2(d-\ell')/d}(1-e^{-\frac{bst}{(b-1)d}})^{2\ell'/d}\Big)^d\d s.
\end{split}\label{eq:sum k 2'}
\end{align}
The next intuition is to show that what is inside the bracket in the integrand of \eqref{eq:sum k 2'}, as a function of $s\in[\delta_5,1-\delta_5]$, attains its maximum at the boundaries $s=\delta_5$ or $s=1-\delta_5$ with values strictly less than $b^2(1+(b-1)e^{-c})$, where we denote by $c=\frac{bt}{(b-1)d}>0$ (recall that $t$ is the solution to \eqref{eq:1} and $t\asymp d$ by Proposition \ref{lemma:asymp 0}).\footnote{There may be lower order terms but the proof below allows for small room of $\ee$ (given from the positive $\delta_5$) possibly subtracted from $c$.}  We first remove terms whose exponentials involve $m,\ell'$, which is possible since $\ell'\in[0,m]$ and the tuple $(m,\ell,\ell')$ satisfies the triangle inequality. Denote by $\ee_3$ an arbitrarily small positive number that goes to $0$ as $L_1\to\infty$ (recall that $m\leq d/L_1$). It follows that what is inside the bracket in \eqref{eq:sum k 2'} is at most
\begin{align*}
    &e^{\ee_3}b^s(1+(b-1)e^{-\frac{bt}{(b-1)d}})^{-s}(1+(b-1)e^{-\frac{bt(1-s)}{(b-1)d}})(1+(b-1)e^{-\frac{bst}{(b-1)d}})^2\\&=e^{\ee_3}b^s(1+(b-1)e^{-c})^{-s}(1+(b-1)e^{-c(1-s)})(1+(b-1)e^{-cs})^2.
\end{align*}
By Lemma \ref{lemma:e0}(ii), we see that given $\delta_5>0$, this part of the second moment is controlled by
\begin{align*}
    \begin{split}
        &tb^{-3d}\rho^t\binom{d}{\ell'}\int_{\delta_5}^{1-\delta_5} (e^{\ee_3}b^s(1+(b-1)e^{-c})^{-s}(1+(b-1)e^{-c(1-s)})(1+(b-1)e^{-cs})^2)^d\d s\\
    &\ll tb^{-3d}\rho^t\binom{d}{\ell'}(2(b-1))^{\ell'}\Big(\sup_{\delta_5\leq s\leq 1-\delta_5}e^{\ee_3}b^s(1+(b-1)e^{-c})^{-s}(1+(b-1)e^{-c(1-s)})(1+(b-1)e^{-cs})^2\Big)^d\\
    &\leq tb^{-3d}\rho^t\binom{d}{\ell'}(2(b-1))^{\ell'}(b^2(1+(b-1)e^{-c})e^{-\ee_3})^d\\
    &=t\binom{d}{\ell'}(2(b-1))^{\ell'}(1+(b-1)e^{-c})^d(1+(b-1)e^{-\frac{bt}{(b-1)d}})^{-(d-m)}(1-e^{-\frac{bt}{(b-1)d}})^{-m}e^{-\ee_3d}\\
    &\ll e^{-\ee_3d/2},
    \end{split}
\end{align*}
for $\ee_3$ chosen small enough, 
where we have used \eqref{eq:1}, $c=\frac{bt}{(b-1)d}$, and that $\ell'\leq m\leq d/L_1$ (note that given $\ee_3$, there exists $L_1>0$ such that $\binom{d}{d/L_1}\ll e^{\ee_3d/3}$, a consequence of Stirling's formula).
Summing over $\ell'\in[m]$ still shows that this part of the contribution to the second moment is exponentially small: $e^{-\ee_3d/4}$ for some $\ee_3>0$.

Let now $\ell'\in(m,d]$. We need the following more precise bound: using \eqref{eq:1}, we have by a similar analysis to the case $\ell'\leq m$ and using $m\leq d/L_1$, and denoting by
$$R(s):=\frac{1-e^{-\frac{bt(1-s)}{(b-1)d}}}{1+(b-1)e^{-\frac{bt(1-s)}{(b-1)d}}}\Big(\frac{1-e^{-\frac{bst}{(b-1)d}}}{1+(b-1)e^{-\frac{bst}{(b-1)d}}}\Big)^2,$$
we have
\begin{align*}
    &\hspace{0.5cm}\int_{\delta_5 t}^{(1-\delta_5 )t} \rho^s(1+(b-1)e^{-\frac{b(t-s)}{(b-1)d}})^{d-\ell}(1-e^{-\frac{b(t-s)}{(b-1)d}})^\ell((1+(b-1)e^{-\frac{bs}{(b-1)d}})^{d-\ell'}(1-e^{-\frac{bs}{(b-1)d}})^{\ell'})^2\d s\\
    &\ll e^{\ee_3 d}\int_{\delta_5 }^{1-\delta_5} \rho^{ts}(1+(b-1)e^{-\frac{bt(1-s)}{(b-1)d}})^{d-\ell'}(1-e^{-\frac{bt(1-s)}{(b-1)d}})^{\ell'}((1+(b-1)e^{-\frac{bst}{(b-1)d}})^{d-\ell'}(1-e^{-\frac{bst}{(b-1)d}})^{\ell'})^2\d s\\
    &\ll \int_{\delta_5}^{1-\delta_5}\Big(e^{\ee_3}b^s(1+(b-1)e^{-\frac{bt}{(b-1)d}})^{-s}(1+(b-1)e^{-\frac{bt(1-s)}{(b-1)d}})(1+(b-1)e^{-\frac{bst}{(b-1)d}})^{2}\Big)^d R(s)^{\ell'}\d s.
\end{align*}
By repeatedly using the binomial sum formula, we arrive at
\begin{align*}
        &\hspace{0.5cm}b^{-3d}\rho^t\sum_{\ell'=m}^d\sum_{m-\ell'\leq \ell\leq m+\ell'}(b-1)^{\ell'}\sum_{i}\binom{m}{i}\binom{m-i}{i+\ell-\ell'}\binom{d-m}{i+\ell-m}\\
    &\hspace{1cm}\times\int_{\delta_5 t}^{(1-\delta_5 )t} \rho^s(1+(b-1)e^{-\frac{b(t-s)}{(b-1)d}})^{d-\ell}(1-e^{-\frac{b(t-s)}{(b-1)d}})^\ell((1+(b-1)e^{-\frac{bs}{(b-1)d}})^{d-\ell'}(1-e^{-\frac{bs}{(b-1)d}})^{\ell'})^2\d s\\
    &\ll b^{-3d}\rho^t\sum_{\ell'=m}^d\sum_{m-\ell'\leq \ell\leq m+\ell'}(b-1)^{\ell'}\sum_{i}\binom{m}{i}\binom{m-i}{i+\ell-\ell'}\binom{d-m}{i+\ell-m}\\
    &\hspace{2cm}\times\int_{\delta_5}^{1-\delta_5}\Big(e^{\ee_3}b^s(1+(b-1)e^{-\frac{bt}{(b-1)d}})^{-s}(1+(b-1)e^{-\frac{bt(1-s)}{(b-1)d}})(1+(b-1)e^{-\frac{bst}{(b-1)d}})^{2}\Big)^d R(s)^{\ell'}\d s\\
    &= b^{-3d}\rho^t\sum_{\ell}\sum_{i}\binom{m}{i}\binom{d-m}{i+\ell-m}\int_{\delta_5}^{1-\delta_5}\Big(e^{\ee_3}b^s(1+(b-1)e^{-\frac{bt}{(b-1)d}})^{-s}(1+(b-1)e^{-\frac{bt(1-s)}{(b-1)d}})(1+(b-1)e^{-\frac{bst}{(b-1)d}})^{2}\Big)^d\\
    &\hspace{2cm}\times (R(s)(b-1))^{2i+\ell-m}(1+R(s)(b-1))^{m-i}\d s\\
    &= b^{-3d}\rho^t\sum_{i}\binom{m}{i}\int_{\delta_5}^{1-\delta_5}\Big(e^{\ee_3}b^s(1+(b-1)e^{-\frac{bt}{(b-1)d}})^{-s}(1+(b-1)e^{-\frac{bt(1-s)}{(b-1)d}})(1+(b-1)e^{-\frac{bst}{(b-1)d}})^{2}\Big)^d\\
    &\hspace{2cm}\times (R(s)(b-1))^{i}(1+R(s)(b-1))^{d-i}\d s\\
    &= b^{-3d}\rho^t\int_{\delta_5}^{1-\delta_5}\Big(e^{\ee_3}b^s(1+(b-1)e^{-\frac{bt}{(b-1)d}})^{-s}(1+(b-1)e^{-\frac{bt(1-s)}{(b-1)d}})(1+(b-1)e^{-\frac{bst}{(b-1)d}})^{2}\Big)^d (1+R(s)(b-1))^{d}\\
    &\hspace{2cm}\times\Big(\frac{1+2R(s)(b-1)}{1+R(s)(b-1)}\Big)^m\d s\\
    &\ll b^{-3d}\rho^t\int_{\delta_5}^{1-\delta_5}\Big(e^{2\ee_3}b^s(1+(b-1)e^{-\frac{bt}{(b-1)d}})^{-s}(1+(b-1)e^{-\frac{bt(1-s)}{(b-1)d}})(1+(b-1)e^{-\frac{bst}{(b-1)d}})^{2}(1+R(s)(b-1))\Big)^d \d s.
\end{align*}
 Then, by Lemma \ref{lemma:e0}(iii) applied with $c=\frac{bt}{(b-1)d}$ and similar considerations as above, the contribution to the second moment is bounded by
 $$b^{-3d}\rho^t (b^2(1+(b-1)e^{-\frac{bt}{(b-1)d}}))^d e^{-\ee_3d}=\rho^tb^{-d}(1+(b-1)e^{-\frac{bt}{(b-1)d}})^de^{-\ee_3d}\ll e^{-\ee_3d/2}.$$

\textbf{Step III: bounding the integral over $s\in[(1-\delta_5)t,t]$}. 
Define
\begin{align}
    T(s):=(b-1)\Big(\frac{1-e^{-\frac{bs}{(b-1)d}}}{1+(b-1)e^{-\frac{bs}{(b-1)d}}}\Big)^2=:\frac{b-1}{B(s)}.\label{eq:T(s)B}
\end{align}
Note that $T(s)$ is increasing in $s$, $B(s)$ is decreasing in $s$, and
\begin{align}
    \lim_{s\to t}B(s)=B:=\Big(\frac{1+(b-1)e^{-c}}{1-e^{-c}}\Big)^2.\label{eq:Bs}
\end{align}
We apply Lemma \ref{lemma:taylor}(i) in the first step, \eqref{eq:1} in the second step, and the binomial sum formula and \eqref{eq:T(s)B} in the third step  to obtain that for some $\ee_2\to 0$ as $\delta_5\to 0$,
\begin{align}
    \begin{split}
        &\hspace{0.5cm}b^{-3d}\rho^t\sum_{\ell,\ell'}\sum_{i}(b-1)^{\ell'-i}\binom{m}{i}\binom{m-i}{i+\ell-\ell'}\binom{d-m}{i+\ell-m}\int_{(1-\delta_5)t}^t \rho^s(1+(b-1)e^{-\frac{b(t-s)}{(b-1)d}})^{d-\ell}(1-e^{-\frac{b(t-s)}{(b-1)d}})^\ell\\
    &\hspace{3cm}\times((1+(b-1)e^{-\frac{bs}{(b-1)d}})^{d-\ell'}(1-e^{-\frac{bs}{(b-1)d}})^{\ell'})^2\d s\\
    &\leq b^{-2d}\rho^t\sum_{\ell,\ell'}\sum_{i}(b-1)^{\ell'-i}\binom{m}{i}\binom{m-i}{i+\ell-\ell'}\binom{d-m}{i+\ell-m}\int_{(1-\delta_5)t}^{t} \rho^se^{-(1-\ee_2)(t-s)(d-\ell)/d}\Big(\frac{t-s}{(b-1)d}\Big)^\ell\\
    &\hspace{4cm}\times((1+(b-1)e^{-\frac{bs}{(b-1)d}})^{d-\ell'}(1-e^{-\frac{bs}{(b-1)d}})^{\ell'})^2\d s\\
    &=(1+(b-1)e^{-\frac{bt}{(b-1)d}})^{2(m-d)}(1-e^{-\frac{bt}{(b-1)d}})^{-2m}\sum_{\ell,\ell'}\sum_{i}(b-1)^{\ell'-\ell-i}\binom{m}{i}\binom{m-i}{i+\ell-\ell'}\binom{d-m}{i+\ell-m}\\
    &\hspace{1cm}\times\int_{(1-\delta_5)t}^{t} \rho^{s-t}e^{-(1-\ee_2)(t-s)(d-\ell)/d}\Big(\frac{t-s}{d}\Big)^\ell((1+(b-1)e^{-\frac{bs}{(b-1)d}})^{d-\ell'}(1-e^{-\frac{bs}{(b-1)d}})^{\ell'})^2\d s\\
    &=(1+(b-1)e^{-\frac{bt}{(b-1)d}})^{2(m-d)}(1-e^{-\frac{bt}{(b-1)d}})^{-2m}\sum_{\ell}\sum_{i}(b-1)^{-\ell-i}\binom{m}{i}\binom{d-m}{i+\ell-m}\\
    &\hspace{1cm}\times\int_{(1-\delta_5)t}^{t} \rho^{s-t}e^{-(1-\ee_2)(t-s)(d-\ell)/d}\Big(\frac{t-s}{d}\Big)^\ell(1+(b-1)e^{-\frac{bs}{(b-1)d}})^{2d}T(s)^{2i+\ell-m}(1+T(s))^{m-i}\d s\\
    &= \sum_{\ell}\sum_{i}(b-1)^{-\ell-i}\binom{m}{i}\binom{d-m}{i+\ell-m}\Big(\frac{1+(b-1)e^{-\frac{bt}{(b-1)d}}}{1-e^{-\frac{bt}{(b-1)d}}}\Big)^{2m}d^{-\ell}\\
    &\hspace{1cm}\times\int_{(1-\delta_5)t}^{t} \rho^{s-t}e^{-(1-\ee_2)(t-s)(d-\ell)/d}(t-s)^\ell\Big(\frac{1+(b-1)e^{-\frac{bs}{(b-1)d}}}{1+(b-1)e^{-\frac{bt}{(b-1)d}}}\Big)^{2d}T(s)^{2i+\ell-m}(1+T(s))^{m-i}\d s.
    \end{split}\label{eq:sum k}
\end{align}
\sloppy Next, using Lemma \ref{lemma:taylor}(ii), a change of variable, and the fact that for any $\ee$, there exists $\delta$ such that $x\in(0,\delta)\implies e^x\leq 1+(1+\ee)x$, we have for some $\ee_4\to 0$ as $\delta_5\to 0$,
\begin{align}
    \begin{split}
        &\hspace{0.5cm}\int_{(1-\delta_5)t}^{t} \rho^{s-t}e^{-(1-\ee_2)(t-s)(d-\ell)/d}(t-s)^\ell\Big(\frac{1+(b-1)e^{-\frac{bs}{(b-1)d}}}{1+(b-1)e^{-\frac{bt}{(b-1)d}}}\Big)^{2d}T(s)^{2i+\ell-m}(1+T(s))^{m-i}\d s\\
    &\leq \int_{(1-\delta_5)t}^{t} \rho^{s-t}e^{-(1-\ee_2)(t-s)(d-\ell)/d}(t-s)^\ell\exp\Big(\frac{2d(b-1)e^{-\frac{bt}{(b-1)d}}(e^{\frac{b(t-s)}{(b-1)d}}-1)}{1+(b-1)e^{-\frac{bt}{(b-1)d}}}\Big)T(s)^{2i+\ell-m}(1+T(s))^{m-i}\d s\\
    &\leq \int_0^{\delta_5 t}s^\ell\rho^{-s}e^{-(1-\ee_2)s(d-\ell)/d}\exp\Big(\frac{2b(1+\ee_4)se^{-\frac{bt}{(b-1)d}}}{1+(b-1)e^{-\frac{bt}{(b-1)d}}}\Big)\Big(\frac{b-1}{B(t-s)}\Big)^{2i+\ell-m}\Big(1+\frac{b-1}{B(t-s)}\Big)^{m-i}\d s\\
    &\leq B^{-i-\ell}(b-1)^{2i+\ell-m}(b-1+B((1-\delta_5)t))^{m-i}\int_0^{\delta_5 t}s^\ell e^{-s/A}\d s \\
    &\ll B^{-i-\ell}(b-1)^{2i+\ell-m}(b-1+B((1-\delta_5)t))^{m-i}A^{\ell+1}\ell!,
    \end{split}\label{eq:int 3}
\end{align}
where, by definition, we have
\begin{align*}
    A=A(d,\ell)&:=\bigg(\log\rho+(1-\ee_2)(1-\frac{\ell}{d})-\frac{2b(1+\ee_4)e^{-\frac{bt}{(b-1)d}}}{1+(b-1)e^{-\frac{bt}{(b-1)d}}}\bigg)^{-1}\\
    &\to \bigg(\log\rho+1-\frac{2be^{-\frac{bt}{(b-1)d}}}{1+(b-1)e^{-\frac{bt}{(b-1)d}}}\bigg)^{-1},
\end{align*}
where the last limit holds as $\delta_5\to 0$ (so that $\ee_2,\ee_4\to 0$) in the integration range $s\in[(1-\delta_5)t,t]$ and we have removed the case with $\ell>d/L$ where $L$ is arbitrarily large as $\delta_5\to 0$ using a trivial bound analogous to \eqref{eq:100^{-d}} (strictly speaking, we use $A+\ee$ where $\ee\to 0$ as $\delta_5\to 0$ here) for the range $s\in[(1-\delta_5)t,t]$. By the formula \eqref{eq:1}, we may further write
$$A\to \bigg(\frac{d}{t}\log\Big(\frac{b}{1+(b-1)e^{-\frac{bt}{(b-1)d}}}\Big)+1-\frac{2be^{-\frac{bt}{(b-1)d}}}{1+(b-1)e^{-\frac{bt}{(b-1)d}}}\bigg)^{-1},$$
where we also send $L_1\to\infty$. 
Next, we apply Lemma \ref{lemma:c} with $c={\frac{bt}{(b-1)d}}$ (which is bounded away from $0$ by Proposition \ref{lemma:asymp 0}) and $B=\Big(\frac{1+(b-1)e^{-c}}{1-e^{-c}}\Big)^2$ (recalling \eqref{eq:Bs}, this is the limit of $B(s)$ as $\delta_5\to 0$ for $s\in[(1-\delta_5)t,t]$), so that $A<B$. Combining \eqref{eq:sum k} and \eqref{eq:int 3} and using Lemma \ref{lemma:md} with $B':=B((1-\delta_5)t)$, we see that the contribution to the second moment is at most 
\begin{align*}
    &\hspace{0.5cm}\sum_{\ell\leq d/L}\sum_{i}(b-1)^{-\ell-i}\binom{m}{i}\binom{d-m}{i+\ell-m}\Big(\frac{1+(b-1)e^{-\frac{bt}{(b-1)d}}}{1-e^{-\frac{bt}{(b-1)d}}}\Big)^{2m}d^{-\ell}\\
    &\hspace{2cm}\times A^{\ell+1}\ell!B^{-i-\ell}(b-1)^{2i+\ell-m}(b-1+B((1-\delta_5)t))^{m-i}\\
    &\leq \sum_{\ell\leq d/L}\sum_{i}\binom{m}{i}\binom{d-m}{i+\ell-m}\Big(\frac{A}{B}\Big)^\ell \binom{d}{\ell}^{-1}\Big(\frac{B(b-1+B((1-\delta_5)t))}{b-1}\Big)^{m-i}\\
    &\ll 1,
\end{align*}
as desired.

\textbf{Step IV: conclusion}. We conclude from the above three steps that under the assumption $m\leq d/L_1$ where $L_1$ depends only on $b,\rho$, $\E[N_\z(t)^2]\ll 1$, where the asymptotic constant does not depend on $m$ as long as $m\leq d/L_1$ is satisfied, as can be checked from the proof. 
\end{proof}

The plan for establishing the upper bound of Theorem \ref{thm:main} is to apply the second moment method (or the Paley--Zygmund inequality) to show that with probability close to one, $N_\z(t)\geq 1$. However, a direct use of the Paley--Zygmund inequality does not suffice due to the factor $L_1$ from Lemma \ref{prop:first moment 1}. To deal with this issue, we apply a bootstrapping argument by first running the BRW for a fixed time before applying Paley--Zygmund to the individual particles. 
In the following, for a set $S\subseteq[d]$, let $N_S(t)$ denote the number of particles in $S$ at time $t$ (with the BRW starting at $m$).

\begin{lemma}\label{lemma:MT}
    Let $C,\ee>0$. Then there exist $M\in\N,\,T>0$ such that for all $d$ large enough and $m\leq  d/L_1$, 
    $$\p(N_{\{m,m+1,\dots,m+M\}}(T)\leq C)<\frac{\ee}{2}.$$
    Note that here we do not assume that $\rho\in(1,e)$ and $\rho$ can be any constant in $(1,\infty)$ that does not depend on $d$.
\end{lemma}

\begin{proof}
We may assume that $C\geq 2$. 
Let $K>0$ be large enough such that 
\begin{align}
    \p(\#V_{K\log C/\log \rho}\leq C)<\frac{\ee}{6},\label{eq:use1}
\end{align}whose existence can be verified by \citep[Theorem 2, Chapter III, Section 7]{athreya2004branching}.
Choose an even number $M>0$ large enough such that for $\xi\lawis \Gamma(M/2,1)$, we have  
\begin{align}
    C\,\p\Big(\xi\leq \frac{K\log C}{\log\rho}\Big)<\frac{\ee}{6}.\label{eq:use2}
\end{align}
Let $\tilde{\tau}$ be the first passage time from the state $m$ to the state $m+M/2$. Using our assumption $m\leq d/L_1$ and a standard coupling argument (e.g.~of the Ehrenfest chain with a biased random walk with transition probability $1-(m+M/2)/d$ if $b=2$), there exists $t_0>0$ independent of $d,m$ such that 
\begin{align}
    \p(\tilde{\tau}\leq t_0)> 1-\frac{\ee}{6}.\label{eq:use3}
\end{align}
We let $T=t_0+K\log C/\log \rho$. 

To construct $C$ particles at time $T$ in $\{m,\dots,m+M\}$, our strategy is to wait until a particle reaches $m+M/2$ and run it for time $K\log C/\log \rho$, while keeping in mind that $M$ is large enough so that it is rare for the reproduced particles to leave the set $\{m,\dots,m+M\}$. More precisely, by the strong Markov property and \eqref{eq:use1}, with probability $>1-\ee/6$, there exist at least $C$ particles at time $\tilde{\tau}+K\log C/\log\rho$ with a common ancestor at location $m+M/2$ at time $\tilde{\tau}$. By \eqref{eq:use2}, with probability $>1-\ee/3$, all such $C$ particles (regardless of their branches) remain in $\{m,\dots,m+M\}$ at time $\tilde{\tau}+K\log C/\log\rho$. Therefore, by \eqref{eq:use3}, we conclude that
$$\p(N_{\{m,m+1,\dots,m+M\}}(T)> C)\geq 1-\frac{\ee}{3}-\p(\tilde{\tau}>t_0)>1-\frac{\ee}{2},$$
    as desired.
\end{proof}

\begin{proof}[Proof of Theorem \ref{thm:main}, upper bound]
    Fix $\ee>0$. Let $L_4$ denote the constant in Proposition \ref{prop:second moment 1}. Let $C$ be such that $(1-1/L_4)^C<\ee/2$ and $M,T$ be as given in Lemma \ref{lemma:MT}.  By Paley--Zygmund's inequality and Proposition \ref{prop:second moment 1}, uniformly for each $j\in\{m,m+1,\dots,m+M\}$, 
    \begin{align*}
        \p(\tau_{d,j}\geq t_{d,j})\leq \p(N_\z^{(d,j)}(t_{d,j})=0)\leq 1-\frac{1}{\E[N_\z^{(d,j)}(t_{d,j})^2]}\leq 1-\frac{1}{L_4}, 
    \end{align*}
where we recall that $\E[N_\z^{(d,j)}(t_{d,j})]=1$.
Therefore, by conditioning on time $T$, we have
$$\p(\tau_{d,m}\geq t_{d,m+M}+T)\leq \p(N_{\{m,m+1,\dots,m+M\}}(T)\leq C)+\Big(1-\frac{1}{L_4}\Big)^C< \ee.$$
    Since $M$ is a large constant that does not depend on $d$ or $m$, we have $t_{d,m+M}-t_{d,m}=O(1)$ by Proposition \ref{lemma:asymp 0}. This completes the proof of the upper bound of $\tau_{d,m}$.
\end{proof}

We conclude this subsection with the proof of Corollary \ref{coro:lambda2mono}.
\begin{proof}
Rescale time by $t\mapsto \lambda_2 t$ so that the mutation rate becomes $1$. Under this rescaling, the branching rate becomes $\lambda_1/\lambda_2$, hence $\rho=\exp(\lambda_1/\lambda_2)\in(1,e)$, and Theorem \ref{thm:main} together with Proposition \ref{thm:const rho expansion}(iii) implies $\tau_{d,m}=x_0(b,\rho)d+o_\p(d)$ for $m=o(d)$ in the rescaled units. Converting back to the original time scale yields \eqref{eq:lambda2:asymp}. For \eqref{eq:lambda2:diff}, write $\alpha=\lambda_2'/\lambda_2>1$ and $\rho'=\exp(\lambda_1/\lambda_2')=\rho^{1/\alpha}$. To see the inequality
$x_0(b,\rho)<x_0(b,\rho')/\alpha$, we assume that $\lambda_2=1$.  Let $x_0=x_0(b,\rho)$ and $x_0'=x_0(b,\rho^{1/\lambda_2'})$. By Proposition \ref{thm:const rho expansion}(iii), Theorem \ref{thm:main}, and a time-scaling argument, it remains to show that $x_0<x_0'/\lambda_2'$. To see this, recall from Proposition \ref{thm:const rho expansion}(ii) that $x_0<x_0'$. We then apply \eqref{eq:x0  def} twice to get
\begin{align*}
    x_0\log\rho=\log\Big(\frac{b}{1+(b-1)e^{-\frac{bx_0}{b-1}}}\Big)<\log\Big(\frac{b}{1+(b-1)e^{-\frac{bx_0'}{b-1}}}\Big)=x_0'\log\rho^{1/\lambda_2'}=\frac{x_0'}{\lambda_2'}\log\rho,
\end{align*}
as desired. This shows that the coefficient in \eqref{eq:lambda2:diff} is strictly positive.
\end{proof}

\subsection{Proof for the fast branching case \texorpdfstring{$\rho>e$}{}}

\subsubsection{First moment computation---lower bound of FPT}

Recall from Lemma \ref{thm:bingham} that the expected number of particles at the origin $\z$ at time $t\in[0,t_{d,m}]$
is given by
\begin{align*}
    \rho^tb^{-d}\big(1+(b-1)e^{-\frac{bt}{(b-1)d}}\big)^{d-m}\big(1-e^{-\frac{bt}{(b-1)d}}\big)^m.
\end{align*}
Similar to Lemma \ref{prop:first moment 1}, the following proposition bounds the expected total time spent at $\z$ by all particles, at time $t_{d,m-C}$. Here, we use the convention that $t_{d,k}=0$ for $k\leq 0$. 

\begin{lemma}\label{prop:first moment 3}
    Let $L>0$ be arbitrary and $t_{d,m}$ be given by \eqref{eq:tdm:rho>e}. Then for any $\ee>0$, there exists $C>0$ such that
    $$\int_0^{t_{d,m-C}} \rho^sb^{-d}(1+(b-1)e^{-\frac{bs}{(b-1)d}})^{d-m}(1-e^{-\frac{bs}{(b-1)d}})^m\d s<\ee$$
    uniformly for $m\in[1,L\sqrt{d}/\log d]$ and $d$ large enough. 
\end{lemma}

\begin{proof}
  Using the elementary inequality $(1+(b-1)e^{-\frac{bx}{b-1}})/b\leq e^{-x+x^2/b}$, our assumption that $m=O(\sqrt{d}/\log d)$, and that $t_{d,m}\ll m\log d$ by Lemma \ref{lemma:key}(i), we have
    \begin{align*}
    \rho^tb^{-d}\big(1+(b-1)e^{-\frac{bt}{(b-1)d}}\big)^{d-m}\big(1-e^{-\frac{bt}{(b-1)d}}\big)^m&\leq \rho^t e^{-(\frac{t}{d}-\frac{t^2}{bd^2})(d-m)}\Big(\frac{t}{(b-1)d}\Big)^m\ll \Big(\frac{\rho}{e}\Big)^t \Big(\frac{t}{(b-1)d}\Big)^m.
\end{align*}
  It follows from \eqref{eq:tdm} and Lemma \ref{lemma:key}(i) that
$$\int_0^{t_{d,m-C}} \rho^sb^{-d}(1+(b-1)e^{-\frac{bs}{(b-1)d}})^{d-m}(1-e^{-\frac{bs}{(b-1)d}})^m\d s\ll \Big(\frac{\rho}{e}\Big)^{t_{d,m-C}}\Big(\frac{t_{d,m-C}}{(b-1)d}\Big)^m\ll \Big(\frac{\rho}{e}\Big)^{t_{d,m-C}-t_{d,m}}.$$
The claim then follows from Lemma \ref{lemma:key}(iv) and \eqref{eq:tdm}. 
\end{proof}

\begin{proof}[Proof of Theorem \ref{thm:main2}, lower bound] 
  The proof follows in the same way as the proof of the lower bound of Theorem \ref{thm:main}, using the first moment method and applying Lemma \ref{prop:first moment 3} instead of Lemma \ref{prop:first moment 1}. This yields that for any $\ee>0$, there exists $C>0$ such that $\p(\tau_{d,m}\leq t_{d,m-C})<\ee$. We then conclude the proof using Lemma \ref{lemma:key}(i) and (iii), which together imply that $t_{d,m}-t_{d,m-C}\ll Ct_{d,m}/m\ll C\log d$.
\end{proof}

\subsubsection{Setting up the barrier and proof of the upper bound of FPT}

Suppose that the BRW is initiated at some $\bx_m\in\X_{d,m}$ and we consider the first-passage event to $\z$. We project the hypercube $\{0,1,\dots,b-1\}^d$ to the set $\{0,1,\dots,d\}$ according to the Hamming distance from $\z$. The intuition is that the projected position of the first-passage particle strictly decreases from $m$ to $0$ at each mutation event. Moreover, given that a trajectory moves $m$ times in time $[0,t]$, the mutation times can be described using an empirical process with $m$ samples on $[0,t]$.

Recall the definition of $N_\z(t)$ from Section \ref{eq:lb1}. We now define $M_\z(t)=M_\z^{(d,m)}(t)$ as the number of particles at location $\z$ at time $t$ in the BRW starting from a fixed state $\bx_m\in\X_{d,m}$ that satisfy the following.
\begin{itemize}
    \item The projected position of the particle moves strictly monotonically from $m$ to $0$ at each mutation event; in particular, the mutation events occur exactly $m$ times in time $t$. 
    \item The location $\eta_v(s),~s\in[0,t]$ satisfies $\eta_v(s)\geq w(s)$, where
    \begin{align}
        w(s)=w_t(s):=m-\frac{sm}{t}-1.\label{eq:wdef}
    \end{align}
    Note that here we use $\eta_v(s)$ to represent the location in $[d]$, i.e., after the Hamming distance projection.
\end{itemize}
It follows that $M_\z(t)\leq N_\z(t)$ and
\begin{align}
    \p(\tau_{d,m}\leq t)\geq \p(N_\z(t)\geq 1)\geq \p(M_\z(t)\geq 1).\label{eq:M}
\end{align}

We need an estimate on the probability that an empirical process does not cross a constant barrier.   Let $U_1,\dots,U_n$ be i.i.d.~random variables sampled from $\mathrm{U}[0,1]$ and define 
$N^{(n)}_t=\sum_{i=1}^n\bone_{\{U_i\leq t\}}.$
Also define the empirical process (with $n$ samples)
    \begin{align*}
        M^{(n)}_t:=\frac{1}{\sqrt{n}}(N^{(n)}_t-nt):=\frac{1}{\sqrt{n}}\Big(\sum_{i=1}^n\bone_{\{U_i\leq t\}}-nt\Big),\quad t\in[0,1].
    \end{align*}
    For $a,b>0$ and $n\in\N$, we define 
    \begin{align}
        q_n(a,b):=\p\Big(\forall t\in[0,1], ~M^{(n)}_t\leq \frac{1}{\sqrt{n}}(a+(b-a)t)\Big)=\p\big(\forall t\in[0,1], ~N^{(n)}_t-nt\leq a+(b-a)t\big)\label{eq:qn}
    \end{align}
    and $q_0(a,b):=1$. 
    The following result, established in \citep{smirnov1944approximate}, is one of the most fundamental results regarding the Kolmogorov--Smirnov statistic; see also \citep{lauwerier1963asymptotic}.
\begin{lemma}\label{prop:empirical ballot}
Uniformly in $1\leq\lambda\leq \sqrt{n}$, 
    \begin{align*}
        \p\Big(\sup_{t\in[0,1]}(N^{(n)}_t-nt)<\lambda\Big)\asymp \frac{\lambda^2}{n}.
    \end{align*}
\end{lemma}

Let us define
\begin{align}
    \bar{t}=\bar{t}_{d,m}:=\frac{m}{\log\rho-1}W\Big(\frac{(\log\rho-1)(b-1)dm^{\frac{1}{m}}}{m}\Big).\label{eq:bar t def}
\end{align} 
It is elementary to verify that $\bar{t}$ is the unique positive number that satisfies
\begin{align}
    \Big(\frac{\rho}{e}\Big)^{\bar{t}}\Big(\frac{\bar{t}}{(b-1)d}\Big)^m=m.\label{eq:bar t relation}
\end{align}
Note that the right-hand side of \eqref{eq:bar t relation} is different from the first-moment equation \eqref{eq:tdm}. Their difference can be attributed to Lemma \ref{prop:empirical ballot}.  
Moreover, following the same argument as Lemma \ref{lemma:key}(iii), it holds \begin{align}
    \bar{t}=t_{d,m}+O(\log d)\label{eq:to use}
\end{align} uniformly in $m\in[1,d/L_1]$ as $d\to\infty$. Indeed, by the fundamental theorem of calculus and \eqref{eq:W'},
\begin{align*}
    \bar{t}-t_{d,m}\asymp m\int_{\frac{(\log\rho-1)(b-1)d}{m}}^{\frac{(\log\rho-1)(b-1)dm^{1/m}}{m}}W'(x)\,\d x\asymp  m\int_{\frac{(\log\rho-1)(b-1)d}{m}}^{\frac{(\log\rho-1)(b-1)dm^{1/m}}{m}}\frac{1}{x}\,\d x= m\log(m^{1/m})\leq\log d.
\end{align*}
Our next goal is to establish bounds on the first and second moments of $M_\z(\bar{t})$, from which we will deduce the proof of the upper bound of Theorem \ref{thm:main2}.

\begin{proposition}\label{prop:Mt first moment}
    Uniformly in $m\in[1,d]$ as $d\to\infty$, we have $\E[M_\z(\bar{t})]\asymp 1$.
\end{proposition}

\begin{proof}
    Applying the many-to-one formula, we have
    \begin{align}
        \E[M_\z(\bar{t})]&=\rho^{\bar{t}} \p(\eta_v(s)\text{ is decreasing and }\eta_v(s)\geq w(s),~s\in[0,\bar{t}]).\label{eq:p2}
    \end{align}
    Let $E_{d,m,\bar{t}}$ be the event that $m$ transition events occur in the time interval $[0,\bar{t}]$ and that $\eta_v(s),~s\in[0,\bar{t}]$ is decreasing. 
    Using the independence of the occurrence of transition events and the direction of the transition, we have \begin{align}
        \p(E_{d,m,\bar{t}})=\Big(\frac{m!}{((b-1)d)^m}\Big)\Big(\frac{\bar{t}^me^{-\bar{t}}}{m!}\Big)=e^{-\bar{t}}\Big(\frac{\bar{t}}{(b-1)d}\Big)^m.\label{eq:edm}
    \end{align} Next, conditioned on $E_{d,m,\bar{t}}$, the $m$ transition events occur at times that are i.i.d.~uniformly distributed on $[0,\bar{t}]$ (see Theorem 2.3.1 of \citep{ross1995stochastic}). Altogether, we have
    \begin{align}
        \begin{split}
            &\p(\eta_v(s)\text{ is decreasing and }\eta_v(s)\geq w(s),~s\in[0,\bar{t}])\\
        &=\p(E_{d,m,\bar{t}})\,\p(\eta_v(s)\geq w_{\bar{t}}(s),~s\in[0,\bar{t}] \mid E_{d,m,\bar{t}})\\
        &=e^{-\bar{t}}\Big(\frac{\bar{t}}{(b-1)d}\Big)^m\p\Big(\forall s\in[0,\bar{t}],~\sum_{i=1}^m\bone_{\{\bar{t}U_i\leq s\}}\leq m-w_{\bar{t}}(s)\Big),
        \end{split}\label{eq:p3}
    \end{align}
    where $\{U_i\}_{1\leq i\leq m}$ is a collection of $m$ i.i.d.~random variables uniformly distributed on $[0,1]$. Using \eqref{eq:wdef} and a change of variable $s=\bar{t}\,t$, we have
    \begin{align*}
        \forall s\in[0,\bar{t}],~\sum_{i=1}^m\bone_{\{\bar{t}U_i\leq s\}}\leq m-w_{\bar{t}}(s)
        \quad\Longleftrightarrow\quad \forall t\in[0,1],~N^{(m)}_t-mt\leq 1.
    \end{align*}
    Applying Lemma \ref{prop:empirical ballot} with $\lambda=1$ yields
\begin{align}
    \p\Big(\forall s\in[0,\bar{t}],~\sum_{i=1}^m\bone_{\{\bar{t}U_i\leq s\}}\leq m-w_{\bar{t}}(s)\Big)\asymp \frac{1}{m}.\label{eq:p4}
\end{align}
Combining \eqref{eq:bar t relation}, \eqref{eq:p2}, \eqref{eq:p3}, and \eqref{eq:p4} gives the desired result.  
\end{proof}

In the remainder of this section, we prove the following proposition, which is then used to conclude the proof of Theorem \ref{thm:main2}.

\begin{proposition}\label{prop:Mt second moment}
    Fix $L>0$. Uniformly in $m\in [0,L\sqrt{d}/\log d]$ as $d\to\infty$, we have $\E[M_\z(\bar{t})^2]\ll d^4$.
\end{proposition}

\begin{proof}The case $m=0$ is trivial, so we assume $m\geq 1$. 
Recall \eqref{eq:qn}. We apply the many-to-two formula, \eqref{eq:edm}, and the trivial bound $q_n(a,b)\leq 1$ to obtain
\begin{align*}
    \E[M_\z(\bar{t})^2]&=\E[M_\z(\bar{t})]+2(\log\rho)\int_0^{\bar{t}}\rho^{2\bar{t}-s}\sum_{k=0}^{\min\{m,1+ms/\bar{t}\}} \frac{s^ke^{-s}}{((b-1)d)^k}\Big(\frac{(\bar{t}-s)^{m-k}e^{-(\bar{t}-s)}}{((b-1)d)^{m-k}}\Big)^2\binom{m}{k}\\
    &\hspace{7cm}\times q_{k}\Big(1,1+\frac{ms}{\bar{t}}-k\Big)q_{m-k}\Big(1+\frac{ms}{\bar{t}}-k,1\Big)^2\d s\\
    &\ll 1+\int_0^{\bar{t}}\rho^{2\bar{t}-s}\sum_{k=0}^{\min\{m,1+ms/\bar{t}\}} \frac{s^ke^{-s}}{((b-1)d)^k}\Big(\frac{(\bar{t}-s)^{m-k}e^{-(\bar{t}-s)}}{((b-1)d)^{m-k}}\Big)^2\binom{m}{k}\d s.
\end{align*}
It remains to bound the integral. 
 Using \eqref{eq:bar t relation} and Stirling's formula, we get
\begin{align}
   \begin{split}
        &\int_0^{\bar{t}}\rho^{2\bar{t}-s}\sum_{k=0}^{\min\{m,1+ms/\bar{t}\}} \frac{s^ke^{-s}}{((b-1)d)^k}\Big(\frac{(\bar{t}-s)^{m-k}e^{-(\bar{t}-s)}}{((b-1)d)^{m-k}}\Big)^2\binom{m}{k}\d s\\
    &=\int_0^{\bar{t}}\Big(\frac{\rho}{e}\Big)^{2\bar{t}-s}\Big(\frac{\bar{t}-s}{(b-1)d}\Big)^{2m}\sum_{k=0}^{\min\{m,1+ms/\bar{t}\}}\Big(\frac{s(b-1)d}{(\bar{t}-s)^2}\Big)^k\binom{m}{k}\d s\\
     &\ll\int_0^{\bar{t}}m^{2-s/\bar{t}}\Big(\frac{\bar{t}}{(b-1)d}\Big)^{-2m+ms/\bar{t}}\Big(\frac{\bar{t}-s}{(b-1)d}\Big)^{2m}m^m\sqrt{m}\sum_{k=0}^{\min\{m,1+ms/\bar{t}\}}\Big(\frac{s(b-1)d}{(\bar{t}-s)^2}\Big)^k(m-k)^{-(m-k)}k^{-k}\d s.
   \end{split}\label{eq:logd2}
\end{align}
To deal with the sum over $k$, we define for $x\in[0,\min\{m,1+ms/\bar{t}\}]$
$$\psi(x)=\psi_s(x):=\Big(\frac{s(b-1)d}{(\bar{t}-s)^2}\Big)^x(m-x)^{-(m-x)}x^{-x}.$$
It is then straightforward to compute that on the range $s\leq \bar{t}/3$ and if $m\geq 2$,
\begin{align*}
    \frac{\d}{\d x}\log\psi(x)&=\log\Big(\frac{s(b-1)d}{(\bar{t}-s)^2}\Big)-\log x+\log(m-x)\geq \log\Big(\frac{d}{(\bar{t}-s)}-\frac{d\bar{t}}{(\bar{t}-s)^2m}\Big)\geq \log\Big(\frac{d}{6\bar{t}}\Big).
\end{align*}
By Lemma \ref{lemma:key}(ii), we have $\frac{\d}{\d x}\log\psi(x)\geq 2$ for $d$ large enough. If $m\geq 2$ and $s>\bar{t}/3$, we have that for each $k\geq 0$,  
$$\frac{(\frac{s(b-1)d}{(\bar{t}-s)^2})^{k+1}\binom{m}{k+1}}{(\frac{s(b-1)d}{(\bar{t}-s)^2})^k\binom{m}{k}}=\frac{s(b-1)d(m-k)}{(\bar{t}-s)^2(k+1)}\geq \frac{(b-1)d}{2m\bar{t}}\geq \frac{\log d}{K(\rho)}$$
for some constant $K(\rho)>0$ depending only on $b,\rho$, where we have used Lemma \ref{lemma:key}(i) in the last step. 
Therefore, the contribution to the sum over $k$ is always dominated by the term $k=\min\{m,1+ms/\bar{t}\}$ for $m\geq 2$ and $d$ large enough. The same claim is obviously true for $m=1$ because the only terms are $k=0,1$, and from \eqref{eq:logd2} we see that the term $k=0$ contributes at most
$$\int_0^{\bar{t}}\Big(\frac{\rho}{e}\Big)^{2\bar{t}-s}\Big(\frac{\bar{t}-s}{(b-1)d}\Big)^{2m}\d s=\int_0^{\bar{t}}\Big(\frac{\rho}{e}\Big)^{-s}\Big(\frac{\bar{t}-s}{\bar{t}}\Big)^{2}\d s=O(1),$$
where we have used $m=1$ and \eqref{eq:bar t relation}.

To this end, we split the range of the integral into $s\in[0,\bar{t}-2\bar{t}/m]$ and $s\in [\bar{t}-2\bar{t}/m,\bar{t}]$. We have $k\leq 1+ms/\bar{t}$ in the former case and apply the bound $k\leq m$ in the latter case. If $m=1$, we always apply the bound $k\leq 1$.


\textbf{Case I: $s\in [\bar{t}-2\bar{t}/m,\bar{t}]$}. In this case, inserting $k=m$ into \eqref{eq:logd2} yields
\begin{align*}
    \int_{\bar{t}-2\bar{t}/m}^{\bar{t}}\rho^{2\bar{t}-s}\sum_{k=0}^{m} \frac{s^ke^{-s}}{((b-1)d)^k}\Big(\frac{(\bar{t}-s)^{m-k}e^{-(\bar{t}-s)}}{((b-1)d)^{m-k}}\Big)^2\binom{m}{k}\d s    &\ll \int_{\bar{t}-2\bar{t}/m}^{\bar{t}}\rho^{2\bar{t}-s} \frac{s^me^{-s}}{((b-1)d)^m}e^{-2(\bar{t}-s)}\d s\\
    &\asymp \int_{\bar{t}-2\bar{t}/m}^{\bar{t}}m^{2-s/\bar{t}}\Big(\frac{\bar{t}}{(b-1)d}\Big)^{-2m+sm/\bar{t}}\Big(\frac{s}{(b-1)d}\Big)^m\d s\\
    &\ll \frac{md^2}{\bar{t}}.
\end{align*}

\textbf{Case II: $s\in[0,\bar{t}-2\bar{t}/m]$ and $k\in[ms/\bar{t}+1]$}. The contribution is thus given by (inserting $k=ms/\bar{t}+1$)
\begin{align*}
    &\leq\int_0^{\bar{t}-2\bar{t}/m}m^{2-s/\bar{t}}\Big(\frac{\bar{t}}{(b-1)d}\Big)^{-2m+ms/\bar{t}}\Big(\frac{\bar{t}-s}{(b-1)d}\Big)^{2m}\Big(\frac{s(b-1)d}{(\bar{t}-s)^2}\Big)^{ms/\bar{t}+1}\binom{m}{ms/\bar{t}+1}\d s\\
     &\ll \frac{md\bar{t}}{(2\bar{t}/m)^2}\int_0^{\bar{t}-2\bar{t}/m}m^{2-s/\bar{t}}\Big(\frac{\bar{t}}{(b-1)d}\Big)^{-2m+ms/\bar{t}}\Big(\frac{\bar{t}-s}{(b-1)d}\Big)^{2m}\Big(\frac{s(b-1)d}{(\bar{t}-s)^2}\Big)^{ms/\bar{t}}\binom{m}{ms/\bar{t}}\d s\\
     &\ll d^3\int_0^{\bar{t}-2\bar{t}/m}\Big(\frac{\bar{t}}{(b-1)d}\Big)^{-2m+ms/\bar{t}}\Big(\frac{\bar{t}-s}{(b-1)d}\Big)^{2m}\Big(\frac{s(b-1)d}{(\bar{t}-s)^2}\Big)^{ms/\bar{t}}\Big(\frac{\bar{t}}{s}\Big)^{ms/\bar{t}}\Big(\frac{\bar{t}}{\bar{t}-s}\Big)^{m-ms/\bar{t}}\d s\\
     &\ll d^3\int_0^{\bar{t}-2\bar{t}/m}\Big(\frac{\bar{t}}{\bar{t}-s}\Big)^{-m+ms/\bar{t}}\d s\\
     &\ll d^4.
\end{align*}
Combining all the above estimates, Lemma \ref{lemma:key}(i), and our assumption $m\ll \sqrt{d}/\log d$ yields $\E[M_\z(\bar{t})^2]\ll d^4$.
\end{proof}

Recall that for a set $S\subseteq[d]$, $N_S(t)$ denotes the number of particles in $S$ at time $t$ (with the BRW starting at $m$).

\begin{lemma}\label{lemma:bootstrap rho>e}
    Let $\ee>0$ and consider the BRW starting from a fixed state $\bx_m\in\X_{d,m}$ where $m\leq d/L_1$. Then there exist $M=M(\ee)>0$, $T>0$, and $c=c(\ee)>0$, all independent of $m$, such that 
    $$\p\bigg(\forall t>T,~N_{\{m,m+1,\dots,m+M\}}(t)>c\Big(\frac{\rho}{e}\Big)^t\bigg)>1-\ee.$$
\end{lemma}

\begin{proof}
Using a standard coupling argument, given that there exists a particle at the state $m+j$ at some stopping time $\tau_j$, the number of particles at $m+j$ at time $\tau_j+t$ is bounded from below by the number of particles in a birth-death process with birth rate $\log\rho$ and death rate $1$ (by killing all particles that transit to neighbor states). With a uniformly positive probability (say $>\delta_{12}$), this number is bounded from below by $c_0(\rho/e)^t$ for all $t$ for some $c_0>0$ (see Theorems 21.1 and 22.1 of \citep{Harris1963Branching}). 

Now start from a particle at the state $m$. By Lemma \ref{lemma:MT} applied with $C=\log(\ee/2)/\log(1-\delta_{12})$,  there exist $M,T>0$ such that 
$$\p(N_{\{m,m+1,\dots,m+M\}}(T)\leq C)<\frac{\ee}{2}.$$
Conditioning on time $T$ on the event $N_{\{m,m+1,\dots,m+M\}}(T)> C$, we have that the probability that, for some particle at $\{m,m+1,\dots,m+M\}$ at time $T$, the number of its descendants killed upon exiting its initial state is at least $c_0(\rho/e)^{t-T}$ at all times $t\geq T$, is at least $1-(1-\delta_{12})^C>1-\ee/2$. Thus, the desired claim follows with $c=c_0(\rho/e)^{-T}$.
\end{proof}

\begin{proof}[Proof of Theorem \ref{thm:main2}, upper bound] 
The proof follows a similar bootstrapping idea as the proof of the upper bound of Theorem \ref{thm:main}. Let $\ee>0$ be arbitrary, $M,c$ be the constants in Lemma \ref{lemma:bootstrap rho>e}. Recall the definition of $\bar{t}$ from \eqref{eq:bar t def}. By \eqref{eq:M}, Paley--Zygmund's inequality, and Propositions \ref{prop:Mt first moment} and \ref{prop:Mt second moment}, we have uniformly for $j\in[M]$,
$$\p(\tau_{d,m+j}\geq \bar{t}_{d,m+j})\leq \p(N_\z^{(d,m+j)}(\bar{t}_{d,m+j})=0)\leq \p(M_\z^{(d,m+j)}(\bar{t}_{d,m+j})=0)\leq 1-\frac{1}{L_{11}d^4}$$
for some $L_{11}>0$. Let $C_8>0$ be large enough such that for $d$ large enough,
$$\Big(1-\frac{1}{L_{11}d^4}\Big)^{c(\rho/e)^{C_8\log d}}<\ee.$$
It then follows from Lemma \ref{lemma:bootstrap rho>e} that by conditioning on time $C_8\log d$,
\begin{align*}
    &\p(\tau_{d,m}\geq \bar{t}_{d,m+M}+C_8\log d)\\
    &\leq \max_{j\in[M]}\p(\tau_{d,m+j}\geq \bar{t}_{d,m+j})^{c(\rho/e)^{C_8\log d}}+\p\bigg(N_{\{m,m+1,\dots,m+M\}}(C_8\log d)\leq c\Big(\frac{\rho}{e}\Big)^{C_8\log d}\bigg)\\
    &\leq 2\ee.
\end{align*}
Letting $\ee\to 0$ and applying Lemma \ref{lemma:key}(iii) and \eqref{eq:to use} shows that $$\tau_{d,m}\leq \bar{t}_{d,m+M}+O_\p(\log d)\leq \bar{t}+O_\p(\log d),$$ as desired.
\end{proof}

\section{Proof for the case \texorpdfstring{$\log\rho=o(1)$}{}}\label{sec:case2}

\subsection{Lower bound}

The proof follows a route similar to Section \ref{eq:lb1}. We first establish the following estimate on the first moment.

\begin{lemma}\label{prop:first moment 2}
Suppose that $\rho=\rho(d)$ satisfies that $\log\rho(d)$ is regularly varying of index $\gamma\in[-1,0)$ and $\log\rho(d)\geq L_2/d$ for some large enough constant $L_2>0$. Let $t=t_{d,m}$ be the unique positive solution to \eqref{eq:1}.    For any $\ee>0$, there exists $C>0$ such that
\begin{align}\E\Big[\tilde{N}_\z\Big(t-\frac{-\log\log\rho}{\log\rho}-\frac{C}{\log\rho}\Big)\Big]=\int_0^{t-\frac{-\log\log\rho}{\log\rho}-\frac{C}{\log\rho}} \rho^s2^{-d}(1+e^{-2s/d})^{d-m}(1-e^{-2s/d})^m\d s\leq\ee.
    \label{eq:int2}
\end{align}
\end{lemma}

\begin{proof}
The first equality is \eqref{eq:tilde N}. Let $\bar{t}=t-\frac{-\log\log\rho}{\log\rho}-\frac{C}{\log\rho}$. Recall \eqref{eq:1}.
We let $L_6,\delta_8>0$ be constants to be determined and split the integral in \eqref{eq:int2} into the ranges: $s\in[0,L_6d],~s\in[L_6d,t-\delta_8 d/\log\rho],$ and $s\in[t-\delta_8 d/\log\rho,\bar{t}]$. 

\textbf{Step I: bounding the integral over $s\in[0,L_6d]$}. We have for some $\delta_9>0$, by Lemma \ref{lemma:taylor}(iii),
\begin{align*}
    \int_0^{L_6 d} \rho^s2^{-d}(1+e^{-2s/d})^{d-m}(1-e^{-2s/d})^m\d s
    &=\int_0^{L_6 d} \rho^s\Big(\frac{1+e^{-2s/d}}{2}\Big)^{d-m}\Big(\frac{1-e^{-2s/d}}{2}\Big)^m\d s\\
    &\leq \int_0^{L_6 d} \rho^se^{-\delta_9\frac{s(d-m)}{d}}\Big(\frac{s}{d}\Big)^m\d s.
\end{align*}
We split the integral into two parts: over $s\in[0,d]$ and over $s\in[d,L_6d]$. 
\begin{itemize}
    \item First, 
\begin{align*}
    \int_0^{d} \rho^se^{-\delta_9\frac{s(d-m)}{d}}\Big(\frac{s}{d}\Big)^m\d s&\leq \frac{1}{d}\int_0^{L_6 d} s\rho^se^{-\delta_9\frac{s(d-m)}{d}}\d s\ll \frac{1}{d}=o(1).
\end{align*}
    \item Second, 
    \begin{align*}\int_d^{L_6 d} \rho^se^{-\delta_9\frac{s(d-m)}{d}}\Big(\frac{s}{d}\Big)^m\d s &\ll (L_6e^{\delta_9 L_6})^m \int_d^{L_6d}(\rho e^{-\delta_9})^s\d s\\
    &\ll (L_6e^{\delta_9 L_6})^me^{-\delta_9 d/2}= o(1),
\end{align*}
where we have used $m\leq d/L_3$ for some large $L_3>0$. 
\end{itemize}
Altogether, we arrive at
$$ \int_0^{L_6 d} \rho^s2^{-d}(1+e^{-2s/d})^{d-m}(1-e^{-2s/d})^m\d s=o(1).$$

\textbf{Step II: bounding the integral over $s\in[L_6d,t-\delta_8 d/\log\rho]$}. Define the function $g(v)=v\log\rho+e^{-2v}$ for $v\in[L_6,t/d-\delta_8/\log\rho]$. The function $g$ is convex and hence may attain maximum at the endpoints $v=L_6$ or $v=t/d-\delta_8/\log\rho$. For $L_6$ large enough (independent of $\rho$),
$e^{g(L_6)}=e^{L_6\log\rho}e^{e^{-2L_6}}<3/2$, since $\rho(d)\to 1$. On the other hand, for $\delta_8$ small enough, $$e^{g(t/d-\delta_8/\log\rho)}=\rho^{t/d-\delta_8/\log\rho}e^{e^{-2(t/d-\delta_8/\log\rho)}}\leq 2e^{-\delta_8/2},$$
where we have used Proposition \ref{lemma:asymp 2}. These altogether show that 
\begin{align}
    e^{g(v)}\leq 2e^{-\delta_8/2},\quad v\in [L_6,t/d-\delta_8/\log\rho].\label{eq:gv}
\end{align}
We therefore have by \eqref{eq:1}, a change of variable, and Lemma \ref{lemma:taylor}(ii),
\begin{align}
    \begin{split}
        \int_{L_6d}^{t-\delta_8 d/\log\rho} \rho^s2^{-d}(1+e^{-2s/d})^{d-m}(1-e^{-2s/d})^m\d s&=\int_{L_6d}^{t-\delta_8 d/\log\rho} \rho^{s-t}\Big(\frac{1+e^{-2s/d}}{1+e^{-2t/d}}\Big)^{d-m}\Big(\frac{1-e^{-2s/d}}{1-e^{-2t/d}}\Big)^m\d s\\
    &\leq \int_{L_6d}^{t-\delta_8 d/\log\rho} \rho^{s-t}\Big(\frac{1+e^{-2s/d}}{1+e^{-2t/d}}\Big)^{d}\d s\\
    &\ll d\int_{L_6}^{t/d-\delta_8/\log\rho} \rho^{vd-t}\exp\Big(\frac{de^{-2t/d}(e^{2t/d-2v}-1)}{1+e^{-2t/d}}\Big)\d v\\
    &\leq d\int_{L_6}^{t/d-\delta_8/\log\rho} \rho^{vd-t}\exp(de^{-2v})\,\d v\\
   &\ll d(2e^{-\delta_8/4})^{-d}\int_{L_6}^{t/d-\delta_8/\log\rho} e^{dg(v)}\d v.
    \end{split}
    \label{eq:e}
\end{align}
By \eqref{eq:gv} and Proposition \ref{lemma:asymp 2}, we have
\begin{align}
    d(2e^{-\delta_8/4})^{-d}\int_{L_6}^{t/d-\delta_8/\log\rho} e^{dg(v)}\d v&\leq d(2e^{-\delta_8/4})^{-d}\Big(\frac{t}{d}\Big)(2e^{-\delta_8/2})^d\ll \frac{d}{\log\rho}\,e^{-\delta_8d/4}.\label{eq:e'}
\end{align}
Using our assumption $\log\rho(d)\geq L_2/d$ and combining \eqref{eq:e} and \eqref{eq:e'}, we conclude that
$$\int_{L_6d}^{t-\delta_8 d/\log\rho} \rho^s2^{-d}(1+e^{-2s/d})^{d-m}(1-e^{-2s/d})^m\d s=o(1).$$

\textbf{Step III: bounding the integral over $s\in[t-\delta_8 d/\log\rho,\bar{t}]$}. 
We proceed in a similar way to \eqref{eq:e} and use Lemma \ref{lemma:taylor}(ii) to obtain
\begin{align*}
    \int_{t-\delta_8 d/\log\rho}^{\bar{t}} \rho^s2^{-d}(1+e^{-2s/d})^{d-m}(1-e^{-2s/d})^m\d s
    &\leq  \int_{t-\delta_8 d/\log\rho}^{\bar{t}} \rho^{s-t}\Big(\frac{1+e^{-2s/d}}{1+e^{-2t/d}}\Big)^{d}\d s\\
    &\leq \int_{t-\delta_8 d/\log\rho}^{\bar{t}} \rho^{s-t}e^{de^{-2t/d}(e^{2(t-s)/d}-1)}\d s.
\end{align*}
Using a change of variable $u=(t-s)\log\rho$, this becomes
\begin{align}
    \int_{t-\delta_8 d/\log\rho}^{\bar{t}} \rho^{s-t}e^{de^{-2t/d}(e^{2(t-s)/d}-1)}\d s&=\frac{1}{\log\rho}\int_{-\log\log\rho+C}^{\delta_8 d}e^{-u+de^{-\frac{2t}{d}}(e^{\frac{2u}{d\log\rho}}-1)}\d u.\label{eq:tt}
\end{align}
Note that for $u\in[-\log\log\rho+C,\delta_8d]$, we have by Proposition \ref{lemma:asymp 2},
$$-u+de^{-\frac{2t}{d}}(e^{\frac{2u}{d\log\rho}}-1)\leq -u+de^{-\frac{2t}{d}+\frac{2u}{d\log\rho}}\leq -u+de^{-\frac{1}{\log\rho}}\leq -u+O(1),$$
where $\delta_8$ is chosen small enough in the second step and we have used the regularly varying property in the third step along with Potter's bound. Together, we have
\begin{align*}
    \int_{-\log\log\rho+C}^{\delta_8 d}e^{-u+de^{-\frac{2t}{d}}(e^{\frac{2u}{d\log\rho}}-1)}\d u&\leq \int_{-\log\log\rho+C}^{\delta_8 d}e^{-u+O(1)}\d u\ll (\log \rho)e^{-C},
\end{align*}where the implicit constant in $\ll$ does not depend on $C$. 
Hence, by \eqref{eq:tt}, we arrive at 
$$\int_{t-\delta_8 d/\log\rho}^{\bar{t}} \rho^{s-t}e^{de^{-2t/d}(e^{2(t-s)/d}-1)}\d s\ll e^{-C}.$$

Since the constants in $\ll$ do not depend on $C$ in the above analysis, combining the above three cases yields \eqref{eq:int2} by choosing $C$ large enough.
\end{proof}

\begin{proof}[Proof of Theorem \ref{thm:main rho->1}, lower bound]
The proof follows in the same way as the proof of the lower bound of Theorem \ref{thm:main}, using the first moment method and applying Lemma \ref{prop:first moment 2} instead of Lemma \ref{prop:first moment 1}. We omit the details.
\end{proof}

\subsection{Upper bound}

Let 
\begin{align}
    t':=t_{d,m}-\frac{-\log\log\rho}{\log\rho}\label{eq:t' def}
\end{align}
be the claimed asymptotics of the first passage time $\tau_{d,m}$.  
To prove the upper bound of $\tau_{d,m}$, we apply a weighted second moment approach similarly to Proposition 8 of \citep{blanchet2024tight}.\footnote{Note that directly computing $N_\z(t')$ using the second moment method (as in Section \ref{sec:2nd}) does not work, since $\E[N_\z(t')]=o(1)$.} 
 The plan of our proof is to condition on time $t'$, so that each particle at time $t'$ carries a probability that some of its descendants reach the state $0$ in the next $1/\log\rho$ period of time. We then apply a weighted second moment method to count the weighted number of particles at time $t'$, where the weight of $v\in V_{t'}$ refers to the first passage probability of the descendants of $v$ to state $0$ until time $t'+1/\log\rho$. The following lemma establishes an asymptotic lower bound for these weights. Recall that $\X_{d,m}=\{\bx\in\X:d_{\mathrm{H}}(\bx,\z)=m\}$.

\begin{lemma}\label{lemma:probs}
  Let $\rho(d)\to 1^+$ and $\log\rho(d)\geq 3/d$.  It holds uniformly in $k\in[1/\log\rho(d)]$ that 
  \begin{align}
      \p\Big(\tau_{d,k}\leq \frac{3}{\log\rho(d)}\Big)\gg \binom{d}{k}^{-1}.\label{eq:tau??}
  \end{align}
\end{lemma}

\begin{proof}
The case $k=0$ is trivial, so we assume $k\geq 1$. Observe that the projected process of a simple random walk on the binary hypercube, onto the number of ones, is given by the Ehrenfest chain, which is a Markov chain on $[d]$ with transition probabilities $p_{i,i+1}=(d-i)/d$ and $p_{i,i-1}=i/d$.  
Let us identify the initial starting point of the BRW as state $0$ in the Ehrenfest chain. The target, which is of Hamming distance $k$ from the starting point, is located in $\X_{d,k}$. Consider the first time that the Ehrenfest chain reaches the state $k$ (we may assume that there is no branching since this only makes the left-hand side of \eqref{eq:tau??} smaller). Since $\#\X_{d,k}=\binom{d}{k}$, by the symmetry of the hypercube, the first visited vertex is the target with probability $\binom{d}{k}^{-1}$. Therefore, it remains to show that the Ehrenfest chain reaches the state $k$ in time $3k$ with probability $\gg 1$. Note that for $\ell\leq 1/\log\rho(d)\leq d/3$, the transition probability from state $\ell$ to $\ell+1$ is $p_{\ell,\ell+1}=(d-\ell)/d>2/3$. Using a standard coupling argument, the chain reaches the target in the first $3k$ transition events with probability $\gg 1$. Since the total transition time is given by $T_{3k}\lawis \Gamma(3k,1)$, we have uniformly in $k\in[1/\log\rho(d)]$ that $\p(T_{3k}<3/\log\rho(d))\gg 1$. Combining the above and using the independence between transition times and directions completes the proof.    
\end{proof}

\begin{lemma}\label{lemma:enumeration}
Fix $\bx_0\in \X_{d,m}$, $\ell_1,\ell_2,\ell_3\in[d]$, and $k,k'\in[d]$. The number of tuples $(\bx_1,\bx_2,\bx_3)\in \X\times\X_{d,k}\times\X_{d,k'}$ such that $d_{\mathrm{H}}(\bx_0,\bx_1)=\ell_1$, $d_{\mathrm{H}}(\bx_1,\bx_2)=\ell_2$, and $d_{\mathrm{H}}(\bx_1,\bx_3)=\ell_3$, is given by
\begin{align}
    q_{\ell_1,\ell_2,\ell_3,m}:=\sum_{\ell=0}^{d}\binom{m}{\frac{m+\ell_1-\ell}{2}}\binom{d-m}{\frac{\ell+\ell_1-m}{2}}\binom{\ell}{\frac{\ell+\ell_2-k}{2}}\binom{d-\ell}{\frac{\ell_2+k-\ell}{2}} \binom{\ell}{\frac{\ell+\ell_3-k'}{2}}\binom{d-\ell}{\frac{\ell_3+k'-\ell}{2}} .\label{eq:q?}
\end{align}
Here, the effective range of the sum in \eqref{eq:q?} is a subset of $[\min\{\ell_2+k,\ell_3+k'\}]$. 
\end{lemma}
\begin{proof}
    Let $\ell\in[d]$ be such that $\bx_1\in\X_{d,\ell}$. For $j\in\{1,\dots,d\}$ and $\bx\in\X$ we let $(\bx)_j$ denote the $j$-th entry of $\bx$, which is either $0$ or $1$. It follows that the number of indices $j\in\{1,\dots,d\}$ such that $(\bx_0)_j=1$ and $(\bx_1)_j=0$ is $(m+\ell_1-\ell)/2$ and the number of indices $j\in\{1,\dots,d\}$ such that $(\bx_0)_j=0$ and $(\bx_1)_j=1$ is $(\ell+\ell_1-m)/2$.
    Therefore, there are $\binom{m}{({m+\ell_1-\ell})/{2}}\binom{d-m}{({\ell+\ell_1-m})/{2}}$ many choices of $\bx_1$. 
  Given any such $\bx_1$, we may apply the same analysis which gives that there are $\binom{\ell}{({\ell+\ell_2-k})/{2}}\binom{d-\ell}{({\ell_2+k-\ell})/{2}}$ choices of $\bx_2$ and $\binom{\ell}{({\ell+\ell_3-k'})/{2}}\binom{d-\ell}{({\ell_3+k'-\ell})/{2}}$ choices of $\bx_3$. Multiplying the three quantities and summing over $\ell\in[d]$ yields \eqref{eq:q?}.
\end{proof}

\begin{proof}[Proof of Theorem \ref{thm:main rho->1}, upper bound]
\sloppy We write $\rho=\rho(d)$ and recall \eqref{eq:t' def}. Define a collection of independent $\{0,1\}$-valued random variables $\{\delta_{v,k}\}_{v\in V_{t'},1\leq k\leq 1/\log\rho}$ independent of everything else, such that 
\begin{align}
    \p(\delta_{v,k}=1)=\frac{1}{L_7}\binom{d}{k}^{-1},\label{eq:delta def}
\end{align}
where $L_7$ is the implicit constant in \eqref{eq:tau??}. We slightly abuse notation and let $\eta_v$ denote the location of particle $v$ in the corresponding Ehrenfest chain on $\{0,\dots,d\}$ (instead of the location on the hypercube). By Lemma \ref{lemma:probs} and conditioning on time $t'$, it is clear that 
$$\p\Big(\tau_{d,m}>t'+\frac{3}{\log\rho}\Big)\leq \p\bigg(\sum_{v\in V_{t'}}\sum_{1\leq k\leq 1/\log\rho}\bone_{\{\eta_v=k\}}\,\delta_{v,k}=0\bigg).$$
We then apply the second moment method to the weighted sum
\begin{align}
    \zeta_{t'}:=\sum_{v\in V_{t'}}\sum_{1\leq k\leq 1/\log\rho}\bone_{\{\eta_v=k\}}\delta_{v,k}.\label{eq:zetat}
\end{align}
Using Lemma \ref{thm:bingham}, we may compute the first moment. By independence and \eqref{eq:delta def}, we have
\begin{align}
   \begin{split}
        \E[\zeta_{t'}]&=\frac{1}{L_7}\sum_{1\leq k\leq 1/\log\rho}\binom{d}{k}^{-1}\E\bigg[\sum_{v\in V_{t'}}\bone_{\{\eta_v=k\}}\bigg]\\
    &=\frac{1}{L_7}\rho^{t'}\sum_{1\leq k\leq 1/\log\rho}\sum_{\ell=0}^k\binom{m}{\ell}\binom{d-m}{k-\ell}2^{-d}(1+e^{-2t'/d})^{d-(m+k-2\ell)}(1-e^{-2t'/d})^{m+k-2\ell}\binom{d}{k}^{-1}.
   \end{split}\label{eq:i}
\end{align}
Since $t'/d\gg 1/\log\rho$ (from Proposition \ref{lemma:asymp 2}), we have 
$$\frac{1+e^{-2t'/d}}{1-e^{-2t'/d}}\leq e^{L_8e^{-2t'/d}}\leq e^{L_8e^{-1/(L_8\log\rho)}}$$
for some $L_8>0$. Using $0\leq\ell\leq k\leq 1/\log\rho$ and $1/\log\rho\to\infty$ as $d\to\infty$, we have
$$(1+e^{-2t'/d})^{-k+2\ell}(1-e^{-2t'/d})^{k-2\ell}\geq e^{-L_8e^{-1/(L_8\log\rho)}(2/\log\rho)}\gg 1.$$
Inserting into \eqref{eq:i} we obtain
\begin{align*}
     \E[\zeta_{t'}]&\gg \rho^{t'}\sum_{1\leq k\leq 1/\log\rho}\sum_{\ell=0}^k\binom{m}{\ell}\binom{d-m}{k-\ell}2^{-d}(1+e^{-2t'/d})^{d-m}(1-e^{-2t'/d})^{m}\binom{d}{k}^{-1}\\
     &=\rho^{t'}\sum_{1\leq k\leq 1/\log\rho}2^{-d}(1+e^{-2t'/d})^{d-m}(1-e^{-2t'/d})^{m}.
\end{align*}
 Since $t_{d,m}$ solves \eqref{eq:1}, we have by Lemma \ref{lemma:taylor}(ii),
\begin{align}
    \begin{split}
        \rho^{t'}={\rho^{t_{d,m}}}({\log\rho})&=({\log\rho})2^d(1+e^{-2t_{d,m}/d})^{m-d}(1-e^{-2t_{d,m}/d})^{-m}\\
    &\asymp ({\log\rho})2^d(1+e^{-2t'/d})^{m-d}(1-e^{-2t'/d})^{-m}.
    \end{split}
\label{eq:rhot'}
\end{align}
Altogether, we conclude that
\begin{align}
    \E[\zeta_{t'}]&\gg \sum_{1\leq k\leq 1/\log\rho}\log\rho\asymp 1.\label{eq:zeta 1}
\end{align}

Next, we consider the second moment $\E[\zeta_{t'}^2]$. For $v,w\in V_t$ and $s\in[0,t]$, we write $v\sim_sw$ if $v$ and $w$ share a common ancestor at time $t-s$ but not later. By the definition \eqref{eq:zetat}, the many-to-two formula (see Appendix II of \citep{sawyer1976branching}), \eqref{eq:delta def}, and \eqref{eq:zeta 1}, we have
\begin{align}
   \E[\zeta_{t'}^2]&= \E\bigg[\bigg(\sum_{v\in V_{t'}}\sum_{1\leq k\leq 1/\log\rho}\bone_{\{\eta_v=k\}}\delta_{v,k}\bigg)^2\bigg]\nonumber\\
   &=\E[\zeta_{t'}]+2(\log\rho)\int_0^{t'}\rho^{t'+s}\sum_{1\leq k\leq 1/\log\rho}\sum_{1\leq k'\leq 1/\log\rho}\E[\delta_{v,k}\delta_{w,k'}]\p(v\sim_sw,\eta_v=k,\eta_w=k')\,\d s\nonumber\\
   &\ll1+(\log\rho)\sum_{1\leq k\leq 1/\log\rho}\sum_{1\leq k'\leq 1/\log\rho}\binom{d}{k}^{-1}\binom{d}{k'}^{-1}\int_0^{t'}\rho^{t'+s}\p(v\sim_sw,\eta_v=k,\eta_w=k')\,\d s.\label{eq:sum 2}
\end{align}
We apply Lemma \ref{lemma:enumeration} to write
\begin{align*}
    \p(v\sim_sw,\eta_v=k,\eta_w=k')=\sum_{\ell_1,\ell_2,\ell_3}p_{\ell_1,\ell_2,\ell_3}(s)\,q_{\ell_1,\ell_2,\ell_3,m},
\end{align*}
where by Lemma \ref{thm:bingham},  
\begin{align*}
    p_{\ell_1,\ell_2,\ell_3}(s)&:=q_{\ell_1}(t'-s)\,q_{\ell_2}(s)\,q_{\ell_3}(s)\\
    &=2^{-3d}(1+e^{-2(t'-s)/d})^{d-\ell_1}(1-e^{-2(t'-s)/d})^{\ell_1}(1+e^{-2s/d})^{d-\ell_2}(1-e^{-2s/d})^{\ell_2}\\
    &\hspace{7cm}\times(1+e^{-2s/d})^{d-\ell_3}(1-e^{-2s/d})^{\ell_3}.
\end{align*}
Let $\delta_{10}>0$ be a small constant to be determined, and we decompose the integral in \eqref{eq:sum 2} into three parts: over $s\in[0,\delta_{10} d]$, $s\in[\delta_{10} d,(d\log d)/2]$, and $s\in[(d\log d)/2,t']$.

\textbf{Step I: bounding the integral over $s\in[0,\delta_{10} d]$}. We proceed by fixing $(k,k')$ and performing integration in $s$ first.  Note first that by our regular variation assumption and Potter's bound, for any $\ee>0$, $\rho\leq e^{\ee/\log d}$ for $d$ large enough, which implies that for $s\leq (d\log d)/2$,
\begin{align}
    (1+e^{-2(t'-s)/d})^{d-\ell_1}(1-e^{-2(t'-s)/d})^{\ell_1}\leq (1+e^{-\log d})^d\ll 1.\label{eq:<<1}
\end{align}
Note that $q_{\ell_1,\ell_2,\ell_3,m}$ does not depend on $s$, so we first bound using \eqref{eq:<<1} and Lemma \ref{lemma:taylor}(i) that
\begin{align*}
    \int_0^{\delta_{10} d} \rho^sp_{\ell_1,\ell_2,\ell_3}(s)\,\d s&\ll 2^{-3d}\int_0^{\delta_{10} d}\rho^s(1+e^{-2s/d})^{d-\ell_2}(1-e^{-2s/d})^{\ell_2}(1+e^{-2s/d})^{d-\ell_3}(1-e^{-2s/d})^{\ell_3}\d s\\
    &\leq 2^{-d}\int_0^{\delta_{10} d}\rho^se^{-\frac{9}{10}s(\frac{d-\ell_2}{d}+\frac{d-\ell_3}{d})}\Big(\frac{s}{d}\Big)^{\ell_2+\ell_3}\d s.
\end{align*}
 If $\ell_2+\ell_3>d/10$, then if we pick $\delta_{10}$ small enough, we have $(s/d)^{\ell_2+\ell_3}\leq \delta_{10}^{d/10}\ll 100^{-d}$. In this case, similarly to \eqref{eq:100^{-d}},
\begin{align*}
    \int_0^{\delta_{10} d}\rho^{t'+s}\p(v\sim_sw,\eta_v=k,\eta_w=k')\,\d s&=\rho^{t'}\sum_{\ell_1,\ell_2,\ell_3}q_{\ell_1,\ell_2,\ell_3,m}\int_0^{\delta_{10} d}\rho^sp_{\ell_1,\ell_2,\ell_3}(s)\,\d s\\
    &\ll \rho^{t'}100^{-d}\sum_{\ell_1,\ell_2,\ell_3}q_{\ell_1,\ell_2,\ell_3,m}\leq \rho^{t'}100^{-d}\,2^d\binom{d}{k}\binom{d}{k'}\ll 2^{-d},
\end{align*}
where we have used \eqref{eq:rhot'}. 
The contribution to \eqref{eq:sum 2} is thus $O(1)$. 

Otherwise, we may assume that $\ell_2+\ell_3\leq d/10$, in which case 
\begin{align*}
    2^{-d}\int_0^{\delta_{10} d}\rho^se^{-\frac{9}{10}s(\frac{d-\ell_2}{d}+\frac{d-\ell_3}{d})}\Big(\frac{s}{d}\Big)^{\ell_2+\ell_3}\d s&\ll 2^{-d} d^{-(\ell_2+\ell_3)}  \int_0^{\delta_{10} d}e^{-s}s^{\ell_2+\ell_3}\d s\leq 2^{-d}\binom{d}{\ell_2+\ell_3}^{-1}  .
\end{align*}
We then have for all $k,k'\in[1/\log\rho]$,
\begin{align}
    &\int_0^{\delta_{10} d}\rho^{t'+s}\p(v\sim_sw,\eta_v=k,\eta_w=k')\,\d s\nonumber\\
    &=\rho^{t'}\sum_{\ell_1,\ell_2,\ell_3}q_{\ell_1,\ell_2,\ell_3,m}\int_0^{\delta_{10} d}\rho^sp_{\ell_1,\ell_2,\ell_3}(s)\,\d s\nonumber\\
    &\ll \rho^{t'}2^{-d} \sum_{\ell_1,\ell_2,\ell_3}\binom{d}{\ell_2+\ell_3}^{-1} q_{\ell_1,\ell_2,\ell_3,m}\nonumber\\
    &\ll (\log\rho)\sum_{\ell_1,\ell_2,\ell_3}\binom{d}{\ell_2+\ell_3}^{-1}  \sum_{\ell=0}^{d}\binom{m}{\frac{m+\ell_1-\ell}{2}}\binom{d-m}{\frac{\ell+\ell_1-m}{2}}\binom{\ell}{\frac{\ell+\ell_2-k}{2}}\binom{d-\ell}{\frac{\ell_2+k-\ell}{2}} \binom{\ell}{\frac{\ell+\ell_3-k'}{2}}\binom{d-\ell}{\frac{\ell_3+k'-\ell}{2}}\nonumber\\
    &=(\log\rho)\sum_{\ell_2,\ell_3}\binom{d}{\ell_2+\ell_3}^{-1}  \sum_{\ell=0}^{d}\binom{d}{\ell}\binom{\ell}{\frac{\ell+\ell_2-k}{2}}\binom{d-\ell}{\frac{\ell_2+k-\ell}{2}} \binom{\ell}{\frac{\ell+\ell_3-k'}{2}}\binom{d-\ell}{\frac{\ell_3+k'-\ell}{2}}.\label{eq:+}
\end{align}
Here, we have used the binomial identity
  \begin{align}
      \sum_{j\in\bZ}\binom{A}{r+j}\binom{B}{s+j}=\sum_{j\in\bZ}\binom{A}{r+j}\binom{B}{B-s-j}=\binom{A+B}{B+r-s}\label{eq:binomial}
  \end{align}
  for $A,B\in \N$ and $r,s\in\bZ$.

Next, we show that the contribution from $\ell<\ell_2+\ell_3-2\log d$ is negligible in \eqref{eq:+}. In this case, since $\ell_2+\ell_3\leq d/10$,
\begin{align*}
    &\hspace{0.5cm}\sum_{\ell_2,\ell_3}\binom{d}{\ell_2+\ell_3}^{-1}  \sum_{\ell=0}^{\ell_2+\ell_3-2\log d}\binom{d}{\ell}\binom{\ell}{\frac{\ell+\ell_2-k}{2}}\binom{d-\ell}{\frac{\ell_2+k-\ell}{2}} \binom{\ell}{\frac{\ell+\ell_3-k'}{2}}\binom{d-\ell}{\frac{\ell_3+k'-\ell}{2}}\\
    &\leq\sum_{\ell=0}^{d/10}\sum_{\substack{\ell_2,\ell_3\\ \ell+2\log d\leq\ell_2+\ell_3\leq d/10}}\binom{d}{\ell_2+\ell_3}^{-1}  \binom{d}{\ell}\binom{\ell}{\frac{\ell+\ell_2-k}{2}}\binom{d-\ell}{\frac{\ell_2+k-\ell}{2}} \binom{\ell}{\frac{\ell+\ell_3-k'}{2}}\binom{d-\ell}{\frac{\ell_3+k'-\ell}{2}}\\
    &\leq\binom{d}{k}\binom{d}{k'}\sum_{\ell=0}^{d/10}\binom{d}{\lfloor \ell+2\log d\rfloor}^{-1}\binom{d}{\ell}\\
    &\ll  \binom{d}{k}\binom{d}{k'}d^{-2}.
\end{align*}Inserting into \eqref{eq:sum 2} yields a contribution of $d^{-1}/\log\rho=O(1)$.

It remains to consider $\ell\geq \ell_2+\ell_3-2\log d$. Since $\ell\leq \min\{\ell_2+k,\ell_3+k'\}$ (see Lemma \ref{lemma:enumeration}), to make the sum over $\ell$ non-empty, we need $\ell_3\leq k+2\log d$ and $\ell_2\leq k'+2\log d$. In particular, since $\log\rho\geq L_2/d$, we have $\ell_2,\ell_3\in[0,d/(L_2-1)]$ for $d$ large enough. We thus obtain\footnote{In this display, the sum is always restricted to quantities where the arguments in the factorial are non-negative integers.}
\begin{align*}
    &\hspace{0.5cm}\sum_{\ell_2,\ell_3\in[0,d/(L_2-1)]}\binom{d}{\ell_2+\ell_3}^{-1}\sum_{\ell=0}^{d}\binom{d}{\ell}\binom{\ell}{\frac{\ell+\ell_2-k}{2}}\binom{d-\ell}{\frac{\ell_2+k-\ell}{2}} \binom{\ell}{\frac{\ell+\ell_3-k'}{2}}\binom{d-\ell}{\frac{\ell_3+k'-\ell}{2}}\\
    &\leq  \sum_{\ell_2,\ell_3\in[0,d/(L_2-1)]}\binom{d}{\ell_2+\ell_3}^{-1}\sum_{\ell=0}^{d}4^\ell\frac{d^\ell}{\ell!}\frac{d^{\frac{\ell_2+k-\ell}{2}}}{(\frac{\ell_2+k-\ell}{2})!}\frac{d^{\frac{\ell_3+k'-\ell}{2}}}{(\frac{\ell_3+k'-\ell}{2})!}\\
    &=\sum_{\ell_2,\ell_3\in[0,d/(L_2-1)]}d^{\frac{\ell_2+\ell_3+k+k'}{2}}\binom{d}{\ell_2+\ell_3}^{-1}\sum_{\ell=0}^{d}\frac{4^\ell}{\ell!(\frac{\ell_2+k-\ell}{2})!(\frac{\ell_3+k'-\ell}{2})!}\\
    &\leq \sum_{\ell_2,\ell_3\in[0,d/(L_2-1)]}\frac{(6d)^{\frac{\ell_2+\ell_3+k+k'}{2}}}{(\frac{\ell_2+\ell_3+k+k'}{2})!}\binom{d}{\ell_2+\ell_3}^{-1}\\
    &\leq \sum_{j=0}^{2d/(L_2-1)}\frac{(j+1)(6d)^{\frac{j+k+k'}{2}}}{\Gamma(\frac{j+k+k'}{2}+1)}\binom{d}{j}^{-1}.
\end{align*}
Denote the summand by
$$a_j=\frac{(j+1)(6d)^{\frac{j+k+k'}{2}}}{\Gamma(\frac{j+k+k'}{2}+1)}\binom{d}{j}^{-1}.$$
Then for $d$ large enough and $j\geq 1$,
$$\frac{a_{j+2}}{a_j}\leq \frac{7d(j+3)^2}{\frac{j+k+k'}{2}(d-j-2)^2}\leq \frac{500dj}{(d-j)^2}\leq \frac{1500}{L_2(1-3/L_2)^2}.$$
Therefore, for $L_2$ chosen large enough, the terms with $j=0,1$ dominate, and hence
\begin{align}
    \sum_{j=0}^{2d/(L_2-1)}\frac{(j+1)(6d)^{\frac{j+k+k'}{2}}}{\Gamma(\frac{j+k+k'}{2}+1)}\binom{d}{j}^{-1}\ll \frac{(6d)^{\frac{k+k'}{2}}}{\Gamma(\frac{k+k'}{2}+1)}+\frac{(6d)^{\frac{k+k'+1}{2}}}{d\Gamma(\frac{1+k+k'}{2}+1)}\ll \frac{(6d)^{\frac{k+k'}{2}}}{\Gamma(\frac{k+k'}{2}+1)}.\label{eq:!decay}
\end{align}
We conclude that
\begin{align*}
    \sum_{\ell_2,\ell_3}\binom{d}{\ell_2+\ell_3}^{-1}\sum_{\ell=0}^{d}\binom{d}{\ell}\binom{\ell}{\frac{\ell+\ell_2-k}{2}}\binom{d-\ell}{\frac{\ell_2+k-\ell}{2}} \binom{\ell}{\frac{\ell+\ell_3-k'}{2}}\binom{d-\ell}{\frac{\ell_3+k'-\ell}{2}}\ll \frac{(6d)^{\frac{k+k'}{2}}}{\Gamma(\frac{k+k'}{2}+1)}.
\end{align*}
It is then easy to check using Stirling's formula that the total contribution to \eqref{eq:sum 2} is at most
$$\sum_{1\leq k\leq 1/\log\rho}\sum_{1\leq k'\leq 1/\log\rho}\binom{d}{k}^{-1}\binom{d}{k'}^{-1}\frac{(6d)^{\frac{k+k'}{2}}}{\Gamma(\frac{k+k'}{2}+1)}\leq \bigg(\sum_{1\leq k\leq d/L_2}\frac{(6k)^{k/2}}{d^{k/2}}\bigg)^2\ll 1,$$
where we have used $\log\rho\geq L_2/d$ with $L_2$ chosen large enough.

\textbf{Step II: bounding the integral over $s\in[\delta_{10} d,(d\log d)/2]$}. Observe from \eqref{eq:binomial} that
\begin{align}
    \sum_{\ell_2}\binom{\ell}{\frac{\ell+\ell_2-k}{2}}\binom{d-\ell}{\frac{\ell_2+k-\ell}{2}}= \binom{d}{k} \quad\text{ and }\quad \sum_{\ell_3}\binom{\ell}{\frac{\ell+\ell_3-k'}{2}}\binom{d-\ell}{\frac{\ell_3+k'-\ell}{2}}=\binom{d}{k'}.\label{eq:binomial 2}
\end{align}
To deal with the contribution from the second term of \eqref{eq:sum 2}, we first consider the part of the sum where $\ell_2+\ell_3\leq (1-\ee_5')d$ for some $\ee_5'>0$ to be determined, and integrate in $s$ first. In this case, 
by using $\rho\leq e^{\ee/\log d}$, for $d$ large enough,
\begin{align}
    \begin{split}
        \int_{\delta_{10} d}^{(d\log d)/2}\rho^s(1+e^{-2s/d})^{2d-\ell_2-\ell_3}(1-e^{-2s/d})^{\ell_2+\ell_3}\d s&\ll e^{\ee_5 d}\sup_{u\geq \delta_{10}}(1+e^{-2u})^{2d-\ell_2-\ell_3}(1-e^{-2u})^{\ell_2+\ell_3}\\
    &\leq e^{\ee_5 d}(1+e^{-2\delta_{11}})^{2d}\Big(\frac{1-e^{-2\delta_{11}}}{1+e^{-2\delta_{11}}}\Big)^{\ell_2+\ell_3},
    \end{split}\label{eq:l2l3}
\end{align}
where $\ee_5>0$ is a small constant to be determined and $\delta_{11}>0$ is some constant independent of $\ell_2,\ell_3$, whose existence is justified by $\ell_2+\ell_3\leq (1-\ee_5')d$. It follows from \eqref{eq:rhot'} and \eqref{eq:<<1} that the contribution to \eqref{eq:sum 2} is of order at most
\begin{align*}
    & (\log\rho)^2\sum_{1\leq k\leq 1/\log\rho}\sum_{1\leq k'\leq 1/\log\rho}\binom{d}{k}^{-1}\binom{d}{k'}^{-1}\sum_{\substack{\ell_2,\ell_3\\ \ell_2+\ell_3\leq (1-\ee_5')d}}\sum_{\ell=0}^{d}\binom{\ell}{\frac{\ell+\ell_2-k}{2}}\binom{d-\ell}{\frac{\ell_2+k-\ell}{2}} \binom{\ell}{\frac{\ell+\ell_3-k'}{2}}\binom{d-\ell}{\frac{\ell_3+k'-\ell}{2}}\\
    &\hspace{2cm} \times\int_{\delta_{10} d}^{(d\log d)/2}\rho^{s}2^{-2d}(1+e^{-2s/d})^{2d-\ell_2-\ell_3}(1-e^{-2s/d})^{\ell_2+\ell_3}\,\d s \sum_{\ell_1}\binom{m}{\frac{m+\ell_1-\ell}{2}}\binom{d-m}{\frac{\ell+\ell_1-m}{2}}\\
    &\ll (\log\rho)^2\sum_{1\leq k\leq 1/\log\rho}\sum_{1\leq k'\leq 1/\log\rho}\binom{d}{k}^{-1}\binom{d}{k'}^{-1}\sum_{\ell_2,\ell_3}\sum_{\ell=0}^{d}\binom{\ell}{\frac{\ell+\ell_2-k}{2}}\binom{d-\ell}{\frac{\ell_2+k-\ell}{2}} \binom{\ell}{\frac{\ell+\ell_3-k'}{2}}\binom{d-\ell}{\frac{\ell_3+k'-\ell}{2}}\\
    &\hspace{2cm} \times\binom{d}{\ell}e^{\ee_5 d}\Big(\frac{1+e^{-2\delta_{11}}}{2}\Big)^{2d}\Big(\frac{1-e^{-2\delta_{11}}}{1+e^{-2\delta_{11}}}\Big)^{\ell_2+\ell_3}\\
    &=(\log\rho)^2e^{\ee_5 d}\Big(\frac{1+e^{-2\delta_{11}}}{2}\Big)^{2d}\sum_{1\leq k\leq 1/\log\rho}\sum_{1\leq k'\leq 1/\log\rho}\binom{d}{k}^{-1}\binom{d}{k'}^{-1}\sum_{\ell=0}^{d}\binom{d}{\ell}\\
    &\hspace{2cm} \times\sum_{\ell_2}\binom{\ell}{\frac{\ell+\ell_2-k}{2}}\binom{d-\ell}{\frac{\ell_2+k-\ell}{2}} \Big(\frac{1-e^{-2\delta_{11}}}{1+e^{-2\delta_{11}}}\Big)^{\ell_2}\sum_{\ell_3}\binom{\ell}{\frac{\ell+\ell_3-k'}{2}}\binom{d-\ell}{\frac{\ell_3+k'-\ell}{2}}\Big(\frac{1-e^{-2\delta_{11}}}{1+e^{-2\delta_{11}}}\Big)^{\ell_3},
\end{align*}
where in the second step, we used that $\rho\leq e^{\ee/\log d}$ for $d$ large enough.
Next, we study the two sums over $\ell_2,\ell_3$ with $k,\ell$ fixed. There are two cases:
\begin{itemize}
    \item if $k\geq \ell$, we have $\ell\leq 1/\log\rho\leq d/L_2$, thus
$$\sum_{\ell_2}\binom{\ell}{\frac{\ell+\ell_2-k}{2}}\binom{d-\ell}{\frac{\ell_2+k-\ell}{2}} \Big(\frac{1-e^{-2\delta_{11}}}{1+e^{-2\delta_{11}}}\Big)^{\ell_2}\leq 2^{d/L_2}\sum_{\ell_2}\binom{d-\ell}{\frac{\ell_2+k-\ell}{2}} \Big(\frac{1-e^{-2\delta_{11}}}{1+e^{-2\delta_{11}}}\Big)^{\ell_2}.$$
Using a similar argument leading to \eqref{eq:!decay}, we see that the summand is maximized at $\ell_2=k+\ell$ or $\ell_2=k+\ell-1$. Therefore,
$$\sum_{\ell_2}\binom{\ell}{\frac{\ell+\ell_2-k}{2}}\binom{d-\ell}{\frac{\ell_2+k-\ell}{2}} \Big(\frac{1-e^{-2\delta_{11}}}{1+e^{-2\delta_{11}}}\Big)^{\ell_2}\leq 2^{d/L_2}d\binom{d}{k}\Big(\frac{1-e^{-2\delta_{11}}}{1+e^{-2\delta_{11}}}\Big)^{k+\ell}.$$
\item  if $k\leq \ell$, we have $\ell_2\geq \ell-k$. Since $\log\rho\geq L_2/d$ with $L_2$ chosen large enough, we obtain
\begin{align*}
    \sum_{\ell_2}\binom{\ell}{\frac{\ell+\ell_2-k}{2}}\binom{d-\ell}{\frac{\ell_2+k-\ell}{2}} \Big(\frac{1-e^{-2\delta_{11}}}{1+e^{-2\delta_{11}}}\Big)^{\ell_2}&\leq \binom{d}{k}\Big(\frac{1-e^{-2\delta_{11}}}{1+e^{-2\delta_{11}}}\Big)^{\ell-k}\ll e^{\ee_5 d}\binom{d}{k}\Big(\frac{1-e^{-2\delta_{11}}}{1+e^{-2\delta_{11}}}\Big)^{\ell}.
\end{align*}
\end{itemize}
In both cases, we have that 
$$\sum_{\ell_2}\binom{\ell}{\frac{\ell+\ell_2-k}{2}}\binom{d-\ell}{\frac{\ell_2+k-\ell}{2}} \Big(\frac{1-e^{-2\delta_{11}}}{1+e^{-2\delta_{11}}}\Big)^{\ell_2}\ll e^{\ee_5 d}\binom{d}{k}\Big(\frac{1-e^{-2\delta_{11}}}{1+e^{-2\delta_{11}}}\Big)^{\ell}.$$
A similar inequality holds for the sum over $\ell_3$.
Therefore,
\begin{align*}
    &\hspace{0.5cm}(\log\rho)^2e^{\ee_5 d}\Big(\frac{1+e^{-2\delta_{11}}}{2}\Big)^{2d}\sum_{1\leq k\leq 1/\log\rho}\sum_{1\leq k'\leq 1/\log\rho}\binom{d}{k}^{-1}\binom{d}{k'}^{-1}\sum_{\ell=0}^{d}\binom{d}{\ell}\\
    &\hspace{2cm} \sum_{\ell_2}\binom{\ell}{\frac{\ell+\ell_2-k}{2}}\binom{d-\ell}{\frac{\ell_2+k-\ell}{2}} \Big(\frac{1-e^{-2\delta_{11}}}{1+e^{-2\delta_{11}}}\Big)^{\ell_2}\sum_{\ell_3}\binom{\ell}{\frac{\ell+\ell_3-k'}{2}}\binom{d-\ell}{\frac{\ell_3+k'-\ell}{2}}\Big(\frac{1-e^{-2\delta_{11}}}{1+e^{-2\delta_{11}}}\Big)^{\ell_3}\\
    &\ll (\log\rho)^2e^{3\ee_5 d}\Big(\frac{1+e^{-2\delta_{11}}}{2}\Big)^{2d}\sum_{1\leq k\leq 1/\log\rho}\sum_{1\leq k'\leq 1/\log\rho}\sum_{\ell=0}^{d}\binom{d}{\ell}\Big(\frac{1-e^{-2\delta_{11}}}{1+e^{-2\delta_{11}}}\Big)^{2\ell}\\
    &\ll e^{3\ee_5 d}\Big(\frac{1+e^{-2\delta_{11}}}{2}\Big)^{2d}\Big(1+\Big(\frac{1-e^{-2\delta_{11}}}{1+e^{-2\delta_{11}}}\Big)^2\Big)^d.
\end{align*}
By Lemma \ref{lemma:e3}, we have, by choosing $\ee_5$ small enough, the above is $O(1)$.

Next, we consider the case $\ell_2+\ell_3>(1-\ee_5')d$. In this case, 
\begin{align*}
    (1+e^{-2s/d})^{2d-\ell_2-\ell_3}(1-e^{-2s/d})^{\ell_2+\ell_3}&\leq (1+e^{-2s/d})^{2d-2\ell_2-2\ell_3}\\
    &\leq  (1+e^{-2s/d})^{2\ee_5'd}\leq 2^{2\ee_5'd},
\end{align*}
so with $\ee_5'$ chosen small enough, we replace \eqref{eq:l2l3} by 
$$\int_{\delta_{10} d}^{(d\log d)/2}\rho^s(1+e^{-2s/d})^{2d-\ell_2-\ell_3}(1-e^{-2s/d})^{\ell_2+\ell_3}\d s\ll e^{\ee_5 d}.$$
Along the same lines, one can show that the contribution to \eqref{eq:sum 2} is of order at most
\begin{align*}
    & \hspace{0.5cm}(\log\rho)^2\sum_{1\leq k\leq 1/\log\rho}\sum_{1\leq k'\leq 1/\log\rho}\binom{d}{k}^{-1}\binom{d}{k'}^{-1}\sum_{\substack{\ell_2,\ell_3\\ \ell_2+\ell_3> (1-\ee_5')d}}\sum_{\ell=0}^{d}\binom{\ell}{\frac{\ell+\ell_2-k}{2}}\binom{d-\ell}{\frac{\ell_2+k-\ell}{2}} \binom{\ell}{\frac{\ell+\ell_3-k'}{2}}\binom{d-\ell}{\frac{\ell_3+k'-\ell}{2}}\\
    &\hspace{2cm} \times\int_{\delta_{10} d}^{(d\log d)/2}\rho^{s}2^{-2d}(1+e^{-2s/d})^{2d-\ell_2-\ell_3}(1-e^{-2s/d})^{\ell_2+\ell_3}\,\d s \sum_{\ell_1}\binom{m}{\frac{m+\ell_1-\ell}{2}}\binom{d-m}{\frac{\ell+\ell_1-m}{2}}\\
    &\ll(\log\rho)^2e^{\ee_5 d}2^{-2d}\sum_{1\leq k\leq 1/\log\rho}\sum_{1\leq k'\leq 1/\log\rho}\binom{d}{k}^{-1}\binom{d}{k'}^{-1}\\
    &\hspace{4cm}\sum_{\ell=0}^{d}\binom{d}{\ell}\sum_{\ell_2}\binom{\ell}{\frac{\ell+\ell_2-k}{2}}\binom{d-\ell}{\frac{\ell_2+k-\ell}{2}} \sum_{\ell_3}\binom{\ell}{\frac{\ell+\ell_3-k'}{2}}\binom{d-\ell}{\frac{\ell_3+k'-\ell}{2}}\\
    &\ll e^{\ee_5 d}2^{-d}\ll 1,
\end{align*}
as desired.

\textbf{Step III: bounding the integral over $s\in[(d\log d)/2,t']$}. 
In contrast to the other two cases, we organize the sums before integrating in $s$. Observe that by \eqref{eq:binomial} and \eqref{eq:binomial 2},
\begin{align*}
    \sum_{\ell_2,\ell_3}q_{\ell_1,\ell_2,\ell_3,m}&=\sum_{\ell=0}^{d}\binom{m}{\frac{m+\ell_1-\ell}{2}}\binom{d-m}{\frac{\ell+\ell_1-m}{2}}\sum_{\ell_2}\binom{\ell}{\frac{\ell+\ell_2-k}{2}}\binom{d-\ell}{\frac{\ell_2+k-\ell}{2}} \sum_{\ell_3}\binom{\ell}{\frac{\ell+\ell_3-k'}{2}}\binom{d-\ell}{\frac{\ell_3+k'-\ell}{2}}\\
    &\leq \binom{d}{k}\binom{d}{k'}\sum_{\ell=0}^{d}\binom{m}{\frac{m+\ell_1-\ell}{2}}\binom{d-m}{\frac{\ell+\ell_1-m}{2}}\leq \binom{d}{k}\binom{d}{k'}\binom{d}{\ell_1}.
\end{align*}
Moreover, since $s\geq (d\log d)/2$, we have
\begin{align*}
    p_{\ell_1,\ell_2,\ell_3}(s)&=2^{-3d}(1+e^{-2(t'-s)/d})^{d-\ell_1}(1-e^{-2(t'-s)/d})^{\ell_1}\\
    &\hspace{1cm}\times(1+e^{-2s/d})^{d-\ell_2}(1-e^{-2s/d})^{\ell_2}(1+e^{-2s/d})^{d-\ell_3}(1-e^{-2s/d})^{\ell_3}\\
    &\leq 2^{-3d}(1+e^{-2(t'-s)/d})^{d-\ell_1}(1-e^{-2(t'-s)/d})^{\ell_1} e^{2de^{-2s/d}}\\
    &\ll 2^{-3d}(1+e^{-2(t'-s)/d})^{d-\ell_1}(1-e^{-2(t'-s)/d})^{\ell_1}.
\end{align*}
It follows that the second term in \eqref{eq:sum 2} is bounded by
\begin{align*}
    &(\log\rho)\sum_{1\leq k\leq 1/\log\rho}\sum_{1\leq k'\leq 1/\log\rho}\binom{d}{k}^{-1}\binom{d}{k'}^{-1}\int_{(d\log d)/2}^{t'}\rho^{t'+s}\p(v\sim_sw,\eta_v=k,\eta_w=k')\,\d s\\
    &\ll (\log\rho)\sum_{1\leq k\leq 1/\log\rho}\sum_{1\leq k'\leq 1/\log\rho}\binom{d}{k}^{-1}\binom{d}{k'}^{-1}\int_{(d\log d)/2}^{t'}\rho^{t'+s}\\
    &\hspace{6cm}\sum_{\ell_1}2^{-3d}(1+e^{-2(t'-s)/d})^{d-\ell_1}(1-e^{-2(t'-s)/d})^{\ell_1}\binom{d}{k}\binom{d}{k'}\binom{d}{\ell_1}\,\d s\\
    &\ll \int_{(d\log d)/2}^{t'}\rho^{s}\sum_{\ell_1}2^{-2d}(1+e^{-2(t'-s)/d})^{d-\ell_1}(1-e^{-2(t'-s)/d})^{\ell_1}\binom{d}{\ell_1}\,\d s\\
    &\leq (\log\rho)\int_0^\infty \rho^{-s}(1+e^{-2s/d})^{d}2^{-d}\sum_{\ell_1}\Big(\frac{1-e^{-2s/d}}{1+e^{-2s/d}}\Big)^{\ell_1}\binom{d}{\ell_1}\,\d s\\
    &=(\log\rho)\int_0^\infty \rho^{-s}\d s\ll 1,
\end{align*}
where we have repeatedly used \eqref{eq:rhot'}, \eqref{eq:<<1}, and applied a change of variable in the third step.

\textbf{Step IV: conclusion.} Combining the above three steps yields $\E[\zeta_{t'}^2]\ll 1$, where the implicit constant does not depend on $m$ as long as $m\leq d/L_3$ with $L_3$ large enough. Together with \eqref{eq:zeta 1} and a bootstrapping argument similar to Lemma \ref{lemma:MT} (which we omit for brevity), the desired upper bound for $\tau_{d,m}$ follows. 
\end{proof}

\section{Proofs of results on cover times}\label{sec:cover times}

In this section, we analyze the cover times of branching random walks on the hypercube, leveraging results on the first passage times.

\begin{lemma}\label{lemma:uniform LB}
 Suppose that $b\in\N_2$ and $\rho\in(1,e)$ does not depend on $d$.   There exists a constant $L_9>0$ such that uniformly for any $\bx,\by\in \X_d^{(b)}$, 
    $$\p(\tau_{d,d_{\mathrm{H}}(\bx,\by)}>L_9d)<\frac{1}{b}.$$
\end{lemma}

\begin{proof}
    Using Proposition \ref{lemma:asymp 0}, we first find a large constant $C_5>0$ such that for all $1\leq m\leq d/L_1$, $t_{d,m}\leq C_5d$. By Theorem \ref{thm:main}, there exists another constant $C_6>0$ such that uniformly in $1\leq m\leq d/L_1$, $\p(\tau_{d,m}\geq t_{d,m}+C_6)<1/(bL_1)$. 
Let $L_9=L_1(C_5+C_6)$. Let $\bx,\by\in\X_d^{(b)}$ be arbitrary. Then there exists a collection of vertices $\bx=\bx_0,\bx_1,\dots,\bx_{L_1}=\by$ in $\X_d^{(b)}$ such that for each $j\in[L_1-1]$, $d_{\mathrm{H}}(\bx_j,\bx_{j+1})\leq d/L_1$.\footnote{For simplicity, we assume that $L_1$ is an integer.} By the union bound and the strong Markov property of the branching random walk,
$$\p(\tau_{d,d_{\mathrm{H}}(\bx,\by)}>L_9d)\leq \sum_{j=0}^{L_1-1}\p(\tau_{d,d_{\mathrm{H}}(\bx_j,\bx_{j+1})}>C_5d+C_6)\leq \sum_{j=0}^{L_1-1}\frac{1}{bL_1}= \frac{1}{b},$$
as desired.
\end{proof}

\begin{proof}[Proof of Corollary \ref{coro:cover}]
The lower bound follows immediately from Theorem \ref{thm:main}, Proposition \ref{lemma:asymp 0}, and the fact that the cover time is lower bounded by first passage time to any fixed vertex. 
To prove an upper bound of $\tau_{\mathrm{cov}}(d)$, we fix $\ee\in(0,1/2)$. 
Without loss of generality, we assume that $d/\ee$ is always an integer. Recall that $V_t$ is the set of all particles at time $t$. Pick $L_{10}>0$ such that 
\begin{align}
    \p\Big(\#V_{L_{10}d}\leq \frac{db^d}{\ee}\Big)<\frac{\ee}{2}.\label{eq:un1}
\end{align}On the event $\{\#V_{L_{10}d}\geq db^d/\ee\}$, we condition on time $L_{10}d$ and label (part of) the existing particles $w_1,\dots,w_{db^d/\ee}$. Also, label $\X_d^{(b)}=\{u_1,\dots,u_{b^d}\}$. By Lemma \ref{lemma:uniform LB} and independence, for a fixed $j\in\{1,\dots,b^d\}$, the probability that $u_j$ is not reached by any descendant of the $d/\ee$ particles $\{w_{1+d(j-1)/\ee},\dots,w_{dj/\ee}\}$ before time $(L_{9}+L_{10})d$ is at most $b^{-d/\ee}$. Taking a union bound over $j\in\{1,\dots,b^d\}$ shows that for $d$ large enough,
\begin{align}
    \p\Big(\tau_{\mathrm{cov}}(d)>(L_{9}+L_{10})d;\,\#V_{L_{10}d}\geq \frac{db^d}{\ee}\Big)\leq b^db^{-d/\ee}<\frac{\ee}{2}.\label{eq:un2}
\end{align}
Applying a union bound to \eqref{eq:un1} and \eqref{eq:un2} completes the proof.
\end{proof}

\begin{proof}[Proof of Proposition \ref{prop:cover tight}]
    Consider $\ee\in(0,1)$ and $t>0$ such that $\p(\tau_{\mathrm{cov}}(d)>t)=\ee$. By the strong Markov property and conditioning on the first branching time, we have
\begin{align}
    \ee=\p(\tau_{\mathrm{cov}}(d)>t)\leq (\log\rho(d))\int_0^\infty e^{-(\log\rho(d))s}\p(\tau_{\mathrm{cov}}(d)>t-s)^2\d s.\label{eq:ee1}
\end{align}
Let $s_*=-\log(\ee-\ee^{4/3})/\log\rho(d)$. By \eqref{eq:ee1} and since $s\mapsto\p(\tau_{\mathrm{cov}}(d)>t-s)^2$ is non-decreasing, 
\begin{align*}
    \ee&\leq \p(\tau_{\mathrm{cov}}(d)>t-s_*)^2\int_0^{s_*} (\log\rho(d))e^{-(\log\rho(d))s}\d s+\int_{s_*}^\infty (\log\rho(d))e^{-(\log\rho(d))s}\d s\\
    &<\p(\tau_{\mathrm{cov}}(d)>t-s_*)^2+\ee-\ee^{4/3}.
\end{align*}
It follows that
\begin{align}
    \p\Big(\tau_{\mathrm{cov}}(d)>t+\frac{\log(\ee-\ee^{4/3})}{\log\rho(d)}\Big)\geq\ee^{2/3}.\label{eq:ee2}
\end{align}
Applying \eqref{eq:ee2} recursively, we see that there exists $C(\ee)>0$ such that $\p(\tau_{\mathrm{cov}}(d)>t-C(\ee)/\log\rho(d))>1-\ee$. This completes the proof of tightness.
\end{proof}

\end{document}